\documentclass[preprint,9pt]{elsarticle}

\usepackage{epsfig}

\usepackage{amssymb}
\usepackage{amsthm}

	\usepackage{amsmath}
    \usepackage{color}
	\usepackage{amsthm}
	\usepackage{amsfonts}
    \usepackage{graphicx, amssymb, subcaption,graphics}
	\usepackage{nicefrac}
	\usepackage{url}
	\usepackage[toc,page]{appendix}

	\newtheorem{thm}{Theorem}[section]
	\newtheorem{prop}[thm]{Proposition}

	\newdefinition{rmk}[thm]{Remark}

	\newproof{pf}{Proof}
	\allowdisplaybreaks

\journal{Computer Methods in Applied Mechanics and Engineering}

\begin{document}

\begin{frontmatter}



\title{Mass Conservative and Energy Stable Finite Difference Methods for the Quasi-incompressible Navier-Stokes-Cahn-Hilliard system: Primitive Variable and Projection-Type Schemes}



\author[uci]{Z.~Guo\corref{cor1}}
\ead{zhenling@math.uci.edu}
\author[ud]{P.~Lin\corref{cor1}}
\ead{plin@maths.dundee.ac.uk}
\author[uci]{J.~Lowengrub\corref{cor1}}
\ead{lowengrb@math.uci.edu}
\author[utk]{S.M.~Wise\corref{cor1}}
\ead{swise@math.utk.edu}

\cortext[cor1]{Corresponding author}
\address[uci]{Department of Mathematics, University of California, Irvine, CA, 92697}
\address[ud]{Department of Mathematics, University of Dundee, Dundee, UK, DD1 4HN}
\address[utk]{Department of Mathematics, University of Tennessee, Knoxville, TN 37996}

\begin{abstract}
In this paper we describe two fully mass conservative, energy stable, finite difference methods on a staggered grid for the quasi-incompressible Navier-Stokes-Cahn-Hilliard (q-NSCH) system governing a binary incompressible fluid flow with variable density and viscosity. Both methods, namely the primitive method (finite difference method in the primitive variable formulation) and the projection method (finite difference method in a projection-type formulation), are so designed that the mass of the binary fluid is preserved, and the energy of the system equations is always non-increasing in time at the fully discrete level. We also present an efficient, practical nonlinear multigrid method - comprised of a standard FAS method for the Cahn-Hilliard equation, and a method based on the Vanka-type smoothing strategy for the Navier-Stokes equation - for solving these equations. We test the scheme in the context of Capillary Waves, rising droplets and Rayleigh-Taylor instability. Quantitative comparisons are made with existing analytical solutions or previous numerical results that validate the accuracy of our numerical schemes. Moreover, in all cases, mass of the single component and the binary fluid was conserved up to $10^{-8}$ and energy decreases in time. 
\end{abstract}

\begin{keyword}
Energy Stability \sep Staggered Finite Differences \sep Multigrid \sep Binary fluid flow \sep Variable Density \sep Phase-field method.


\end{keyword}

\end{frontmatter}


\section{Introduction}
Phase-field, or diffuse-interface models \cite{Anderson1998, Kim2012}, have now emerged as a powerful method to simulate many types of multiphase flows, including drop coalescence, break-up, rising and deformations in shear flows \cite{Guo2014, Hua2011, Lee2001, Lee2002, LiuShen2002, Yue2004}, contact line dynamics \cite{Gao2012,Min-JCP-2014,Guo2015}, thermocapillary effects \cite{Guo2014JFM, Guo2013}, and tumor growth \cite{Oden2012,Oden2010}. Phase-field model are based on models of fluid free energy which goes back to the work of Cahn and collaborators \cite{CahnAllen1978,CahnHilliard1958}. The basic idea is to introduce a phase variable (order parameter) to characterize the different phases that varies continuously over thin interfacial layers and is mostly uniform in the bulk phases. Sharp interfaces are then replaced by the thin but nonzero thickness transition regions where the interfacial forces are smoothly but locally distributed in the bulk fluid. One set governing equations for the whole computational domain can be derived variationally from the free energy, where the order parameter fields satisfy an advection-diffusion equation (usually the advective Cahn-Hilliard equations) and is coupled to the Navier-Stokes equations through extra reactive stresses that mimic surface tension.\\
The classical phase-field model, the Model H \cite{Hohenberg1977}, was initially developed for simulating a binary incompressible fluid where components are density matched, and was later generalized for simulating binary incompressible fluids with variable density components \cite{Abels2012,Quasi-Aki-2014, Boyer2002, Ding2007,  Dong2012, JieSIAM2010,  Lowengrub1998, DG-discrete-2016, vanderZee2016}, in which some models, however, do not satisfy the Galilean invariance or are not thermodynamic consistency. As the phase-field model can be derived through a variational procedure, thermodynamic consistency of the model equations can serve as a justification for the model. In addition, this approach ensures the model compatible with the laws of thermodynamics, and to have a strict relaxational behavior of the free energy, hence the models are more than a phenomenological description of an interfacial problem. Lowengrub and Truskinovsky \cite{Lowengrub1998} and Abels et al. \cite{Abels2012} extended the Model H to a thermodynamically consistent model for variable density using two different modelling assumptions on the phase variable (mass concentration \cite{Lowengrub1998} or volume fraction \cite{Abels2012}) and the velocity field (mass averaged \cite{Lowengrub1998} or volume averaged velocity \cite{Abels2012}). Although the two models are developed to represent the same type of flow dynamics, the resulting equations have significant differences due to the underlying modeling choices. In particular, the quasi-incompressible NSCH model (q-NSCH) developed by Lowengrub and Truskinovsky \cite{Lowengrub1998} adopts a mass-averaged velocity, and the fluids are mixing at the interfacial region which generates the changes in density. Such a system was called quasi-incompressible, which leads to a (generally) non-solenoidal velocity field ($\boldsymbol{\nabla}\cdot \boldsymbol{u} \ne 0$ but was given through the quasi-incompressibility condition) and an extra pressure term appears in the Cahn-Hilliard equation comparing to Model H. In the model of Abels et. al. \cite{Abels2012}, a solenoidal (divergence-free) velocity field is obtained due to the volume-averaged mixture velocity modeling assumption. However the mass conservation equation of their model is modified by adding a mass correction term. Most recently, another quasi-incompressible phase-field model \cite{Zhao2017} was developed to study the binary fluid with variable density, where the volume fraction is employed as the phase variable leading to a different free energy.\\
Solving the q-NSCH model is quite a challenging problem. The CH equation is a fourth order nonlinear parabolic PDE, which contains an extra pressure term; the solution of the phase variable varies sharply through the thin diffuse-interface region where the velocity field is non-solenoidal; the variable density is a non-linear function of the phase variable; the NS and CH equations are strongly coupled, which further increases the mathematical complexity of the model and that makes it difficult to design provably stable numerical schemes. Recently, it has been reported that thermodynamic consistency can serve as not only a critical justification for the phase-field modeling, but also an important criterion for the design of numerical methods. When the thermodynamic consistency is preserved at the discrete level, it guarantees the energy stability of the numerical method and also the accuracy of the solution, especially for the case where a rapid change or a singularity occurs in the solution, such as occurs in non-Newtonian hydrodynamic systems \cite{Lin2006, Lin2007}. Therefore it is highly desirable to design such an energy stable method for the q-NSCH model, which dissipates the energy (preserve thermodynamic consistency) at the discrete level. Many time-discrete or fully discrete level energy stable methods \cite{Ying-Multigrid, Han2016,Han2015,JieSIAM2010,Lin2007,CSY-2015} have been presented for the other types of NSCH models for binary incompressible fluid with the solenoidal velocity field. However for the q-NSCH model presented by Lowengrub and Truskinovsky \cite{Lowengrub1998} or the other quasi-incompressible type models with the non-solenoidal velocity, relatively few time-discrete energy stable methods are available \cite{Giesselmann2013,vanderZee2016,Guo2014}. Very recently, a $C^0$ finite element method for the q-NSCH system with a consistent discrete energy law was presented by Guo et. al. \cite{Guo2014}. where interface topological transitions are captured and the quasi-incompressibility is handled smoothly. At the fully discrete level, however, there are no available energy stable numerical methods for the q-NSCH model.\\
Another important criterion for the method design is to guarantee mass conservation of the binary incompressible fluid at the fully discrete level. Due to appearance of the diffusion term and numerical dissipation introduced in discretization of the convective term in the Cahn-Hilliard equation, the total mass of the binary fluid is usually not preserved exactly. This phenomenon has been reported in several works \cite{mass2015,mass2007}, where the phase-field models were used to study the binary incompressible fluid. The mass loss can get even worse for binary fluid flows with large density ratios, as large numerical dissipation is needed to obtain a stable solution. To handle the issue of mass loss, usually the fine grids and small thickness of diffuse-interfaces are used in the phase-field model to improve mass conservation \cite{Aland2011, Ding2014}. Another way to compensate the mass loss is to add the extra mass correction terms into the CH equation \cite{mass2015}. However, this may need additional efforts for correcting the order parameter at each time step, and the energy stability of the numerical methods can be hardly maintained due to the artificial mass correction term.\\
In the present paper, we develop two fully mass conservative, energy stable, staggered grid finite difference methods for the q-NSCH model. The temperal and space discretization for both methods are so designed that the mass for the binary fluid is preserved naturally at the fully discrete level, and the extra artificial mass correction is not required. Moreover, both methods are energy stable at the fully discrete level. Our first method, primitive method, that uses a primitive variable formulation is based on the Vanka-type smoother \cite{Vanka1986}, where the momentum and continuity equations are coupled implicitly and the velocities and pressures are updated simultaneously in a linear sense. Our second method that uses a projection-type formulation was originated from the work by A. Chorin \cite{Projection-1} for solving the Navier-Stokes equations. The key advantage of the projection method is its efficiency such that the computations of the velocity and the pressure fields are decoupled. Our projection method differs from the traditional projection method in that the latter method usually uses the pressure to project the intermediate velocity onto a space of divergence-free velocity field \cite{JieSIAM2010}, whereas we enforce the quasi-incompressible condition instead. To the best of the author's knowledge, for the q-NSCH model, our two finite difference methods which preserve the mass and meanwhile are energy stable at the fully discrete level are new. To solve the schemes efficiently, we design a practical nonlinear multigrid solver - comprised of a standard FAS method for the CH equation, and a method based on the Vanka-type smoothing strategy for the NS equation - for solving these equations. 
 \\
The rest of the paper is organized as follows: in \S 2, we introduce the q-NSCH model and its non-dimensionaliz-ation. For convenience of the numerical design later, we reformulate the system and show the mass conservation and the energy law (thermodynamic consistency) of the reformulated model in \S 3. In \S 4 we present our two numerical methods, namely the primitive method and projection method, at the time-discrete level, and we demonstrate the mass conservation property and energy stability of both methods. In \S 5, we introduce some basic definitions and notations for the finite difference discretization on a staggered grid. In \S 6, we descretize in space and we show that for both methods the mass conservation and energy stability can be achieved at the fully discrete level. In \S 7, we introduce a non-linear multigrid solver for our fully discrete numerical schemes. In \S 8, we test our methods and compare with the existing results. In \S 9 we present a convergence test for both methods. \S 10 is the conclusion. Moreover, the extra notations for the finite difference discretization and some useful propositions are listed in the Appendix A. The multigrid solver is briefly introduced in B.
\section{Quasi-Incompressible NSCH System}
\subsection{Dimensional System Equations}
As derived in \cite{Lowengrub1998}, the q-NSCH system governing a binary incompressible fluid with variable density and viscosity is
\begin{align}
\rho\boldsymbol{u}_{t}+\rho\boldsymbol{u}\cdot \boldsymbol{\nabla}\boldsymbol{u}=&-\boldsymbol{\nabla} p - \eta\epsilon\sigma \boldsymbol{\nabla} \cdot (\rho \boldsymbol{\nabla} c \otimes \boldsymbol{\nabla} c )+\boldsymbol{\nabla} \cdot \big(\mu(c) (\boldsymbol{\nabla} \boldsymbol{u}+(\boldsymbol{\nabla} \boldsymbol{u})^{\rm T})\big)\notag\\
&+\boldsymbol{\nabla}\big(\lambda(c)(\boldsymbol{\nabla} \cdot \boldsymbol{u}) \big)-\rho g \boldsymbol{j},\label{dim-mom}\\
\boldsymbol{\nabla} \cdot \boldsymbol{u}=&\alpha\boldsymbol{\nabla} \cdot (m(c)\boldsymbol{\nabla}\mu_c),\label{dim-mass}\\
\rho c_{t}+ \rho\boldsymbol{u}\cdot \boldsymbol{\nabla}c=&\boldsymbol{\nabla} \cdot(m(c) \boldsymbol{\nabla} \mu_{c}),\label{dim-phase}\\
\mu_{c}=&\frac{\eta\sigma}{\epsilon} f(c)-\frac{\partial \rho}{\partial c}\cfrac{ p}{\rho^2} -\frac{\eta\epsilon\sigma}{\rho}\boldsymbol{\nabla} \cdot (\rho\boldsymbol{\nabla} c ). \label{dim-chem}
\end{align}
Here $\boldsymbol{u}$ is the velocity, $p$ is the pressure, $g$ is the gravity, $c$ is the phase variable (mass concentration), $\mu_{c}$ is the chemical potential, $f(c)=F'(c)$, and $F(c)=c^2(c-1)^2/4$ is the double-well potential, $m(c)=\sqrt{ c^2(1-c)^2}$ is the variable mobility, $\rho=\rho(c)=\rho_{1}\rho_{2}/((\rho_2-\rho_1)c+\rho_1)$ is the variable density for the binary fluid, $\rho_{1}$ and $\rho_{2}$ are the constant densities for the two incompressible fluids, $\alpha=(\rho_{2}-\rho_{1})/\rho_{1}\rho_{2}$ is a constant such that $\rho'=-\alpha \rho^2$, $\mu (c)=\mu_{1}\mu_{2}/((\mu_{2}-\mu_{1})c+\mu_1)$ is the variable viscosity, $\mu_{1}$ and $\mu_{2}$ are the constant viscosities of the two fluids, $\lambda(c)=-\frac{2}{3}\mu(c)$, $\epsilon$ is a small parameter that is related to the thickness of the diffuse-interface, $\sigma$ is the surface tension from the sharp interface model, $\eta$ is a ratio parameter that relates the phase-field model and sharp interface model \cite{Guo2014JFM}. \\
The no-slip boundary condition is imposed for the velocity field
\begin{align}
\boldsymbol{u }\rvert_{\partial \Omega} =\boldsymbol{u}_{\partial\Omega},\label{cons-bc}
\end{align}
and the Neumann boundary conditions are imposed for the phase-field variables, 
\begin{align}
 \boldsymbol{n}\cdot\boldsymbol{\nabla } c \rvert_{\partial \Omega} =  \boldsymbol{n}\cdot\boldsymbol{\nabla } \mu_{c}\rvert_{\partial \Omega} =0,\label{cons-bc2}
\end{align}
where $\boldsymbol{n}$ is the normal vector pointing out of the physical domain $\Omega$.\\
Note that, by multiplying (\ref{dim-mass}) with $1/\alpha$ and substituting into (\ref{dim-phase}), we obtain,
\begin{align}
\alpha\rho c_{t}+\alpha \rho\boldsymbol{u}\cdot \boldsymbol{\nabla}c=\boldsymbol{\nabla} \cdot \boldsymbol{u}\label{cc-11-1}
\end{align}
Multiplying the above equation by $-\alpha \rho$ and using the definition of $\alpha$, we obtain 
\begin{align}
\rho_t+\boldsymbol{\nabla} \cdot (\rho \boldsymbol{u}) = 0,\label{contiti}
\end{align}
which shows that the q-NSCH system (\ref{dim-mom})-(\ref{dim-chem}) satisfies the mass conservation. Note that this equation is also required to reformulate the system equation in \S 3.
\subsection{Non-dimentionalization}
Let $L_{*}$ and $U_{*}$ denote the characteristic scales of length and velocity, we then introduce the dimensionless independent variables: $\hat{x}=\boldsymbol{x}/L_{*}$, $\hat{\boldsymbol{u}} = \boldsymbol{u}/U_{*}$, $\hat{t} = tU_{*}/L_{*}$, and the following natural scaling of the dependent variables: $\hat{p} = \rho_{*}{\mu_{c}}_{*} p_{l}$, $\hat{\rho} = \rho_{l}/\rho_{*}$, $\hat{\mu} = \mu_{l}/\mu_{*}$, $\hat{\mu}_{c} = \mu_{c}/{\mu_{c}}_{*}$, $\hat{\eta} = \eta_{*}/\rho_{l}$, $\hat{\epsilon} = \epsilon/L_*$ and $\bar{m}(c)=m_{l}/m_{*}(c)$, where the subscripts denote characteristic quantities.  Omitting the hat notation, the non-dimensional q-NSCH model is
\begin{align}
\rho\boldsymbol{u}_{t}+\rho\boldsymbol{u}\cdot \boldsymbol{\nabla}\boldsymbol{u}=&-\cfrac{1}{M}\boldsymbol{\nabla} p -\cfrac{\eta\epsilon}{We}\boldsymbol{\nabla} \cdot (\rho \boldsymbol{\nabla} c \otimes \boldsymbol{\nabla} c )+\cfrac{1}{Re}\boldsymbol{\nabla} \cdot\big(\mu(c) \boldsymbol{\nabla} \boldsymbol{u}\big)\notag\\
&+\cfrac{1}{3Re}\boldsymbol{\nabla}\big(\mu(c)(\boldsymbol{\nabla} \cdot \boldsymbol{u}) \big)-\cfrac{\rho}{Fr} \boldsymbol{j},\label{nd-mom-before1}\\
\boldsymbol{\nabla} \cdot \boldsymbol{u}=&\cfrac{\alpha}{Pe}\boldsymbol{\nabla} \cdot\big( m(c)\boldsymbol{\nabla} \mu_{c}\big),\label{nd-mass}\\
\rho c_{t}+ \rho\boldsymbol{u}\cdot \boldsymbol{\nabla}c=&\cfrac{1}{Pe}\boldsymbol{\nabla} \cdot\big( m(c)\boldsymbol{\nabla} \mu_{c}\big),\label{nd-phase}\\
\mu_{c}=&\cfrac{ M\eta}{\epsilon We} f(c)-\cfrac{\partial \rho}{\partial c}\cfrac{ p}{\rho^2} -\cfrac{M\eta\epsilon}{ We}\cfrac{1}{\rho}\boldsymbol{\nabla} \cdot (\rho\boldsymbol{\nabla} c ), \label{nd-chem}
\end{align}
where $M=U_{*}^2/{\mu_{c}}_{*}$ is an analogue of the squared Mach number measuring the relative strength of the kinetic energy to chemical energy \cite{Lowengrub1998}, $We = \rho_{*} U_{*}^2L_{*}/\sigma$ is the Weber number, $Re=\rho_{*}L_{*}U_{*}/\mu_{l}$ is the classical Reynolds number, $Fr=U_{*}^2/gL_{*}$ is the Froude number, $Pe=\rho U_{*} L_{*}/m_{*} {\mu_{c}}_{*}$ is the diffusional Peclet number.\\
Note that the sharp-interface limit analysis is carried out in \cite{Lowengrub1998, Guo2014JFM} to show the convergence of the q-NSCH model. In particular, as the thickness of the diffuse interface approached to 0 ($\epsilon\rightarrow 0$), the q-NSCH reduces to the classical sharp-interface model for binary incompressible fluids. We will show this convergence property through the numerical simulations in \S \ref{sim-sec}.
\section{Reformulation of the system equations}
For convenience of the numerical scheme design later, we follow the same strategy used in \cite{Guo2014} to reformulate the q-NSCH system (\ref{nd-mom-before1})-(\ref{nd-chem}) to obtain
\begin{align}
&\rho\boldsymbol{u}_t+\rho(\boldsymbol{u}\cdot \boldsymbol{\nabla} )\boldsymbol{u}+\cfrac{1}{2}\rho_t\boldsymbol{u}+\cfrac{1}{2}\boldsymbol{\nabla} \cdot (\rho \boldsymbol{u})\boldsymbol{u}=-\cfrac{1}{M}\boldsymbol{\nabla}\bar{p}+\cfrac{1}{M}\rho \bar{\mu}_{c}\boldsymbol{\nabla} c\notag\\
&+\cfrac{1}{Re}\boldsymbol{\nabla} \cdot (\mu(c) \boldsymbol{\nabla} \boldsymbol{u})+\cfrac{1}{3Re}\boldsymbol{\nabla}\big(\mu(c)(\boldsymbol{\nabla} \cdot \boldsymbol{u}) \big)-\cfrac{\rho}{Fr} \boldsymbol{j},\label{nond-mom}\\
&\boldsymbol{\nabla} \cdot \boldsymbol{u}=\cfrac{\alpha}{Pe}\boldsymbol{\nabla} \cdot\big( m(c)\boldsymbol{\nabla}  \bar{\mu}_{c}\big)+\cfrac{\alpha^2 }{Pe}\boldsymbol{\nabla}\cdot\big(  m(c) \boldsymbol{\nabla} \bar{p}\big),\label{nond-mass}\\
&\rho c_t+\rho(\boldsymbol{u}\cdot \boldsymbol{\nabla} )c=\cfrac{1}{Pe}\boldsymbol{\nabla} \cdot\big( m(c)\boldsymbol{\nabla}  \bar{\mu}_{c}\big)+\cfrac{\alpha }{Pe}\boldsymbol{\nabla}\cdot\big(  m(c) \boldsymbol{\nabla} \bar{p}\big),\label{nond-phase}\\
&\rho\bar{\mu}_{c}=\cfrac{M\eta}{\epsilon We}\rho f(c)+\cfrac{M\eta}{\epsilon We}\cfrac{\partial \rho}{\partial c} F(c)+\cfrac{M\eta\epsilon}{We} \cfrac{\partial \rho}{\partial c} \cfrac{1}{2}(\boldsymbol{\nabla} c \cdot \boldsymbol{\nabla} c) - \cfrac{M\eta\epsilon}{We} \boldsymbol{\nabla} \cdot (\rho\boldsymbol{\nabla} c ).\label{nond-chem}
\end{align} 
Here we have defined a new pressure $\bar{p}$ and a new chemical potential $\bar{\mu}_{c}$, such that
\begin{align}
\bar{p}&=p+\cfrac{\eta M}{\epsilon We}\rho F(c) +\cfrac{\eta \epsilon M}{We}\cfrac{\rho}{2} \boldsymbol{\nabla} c\cdot \boldsymbol{\nabla} c,~~~~{\rm and}~~~~\bar{\mu}_{c}=\mu_{c}+\cfrac{\partial \rho}{\partial c} \cfrac{\bar{p}}{\rho^2}=\mu_{c}-\alpha  \bar{p}.\label{new-mu}
\end{align} 
Note that a zero term $\frac{1}{2}\rho_t\boldsymbol{u}+\frac{1}{2}\boldsymbol{\nabla} \cdot (\rho \boldsymbol{u})\boldsymbol{u}$ has been added to the momentum equation, which can be seen as the multiplication of the continuity equation (\ref{contiti}) and the velocity $\boldsymbol{u}$. The reason is that the test functions in deriving the
energy law can be made much simpler (See also Remark 3.3 below) and thus make it easier to achieve the energy stable
finite difference method.\\
We now show that, in the time-continuous and space-continuous level, the reformulated non-dimensional system (\ref{nond-mom})-(\ref{nond-chem}) satisfies the Mass conservation for the binary fluid and also the fluid components. In order to show these properties, we will only consider the homogeneous boundary condition for $\boldsymbol{u}_{\partial\Omega}=0$, such that, the boundary terms that are originated from the integration by parts can be dropped by using the homogeneous boundary conditions.
\begin{thm}
The non-dimensional q-NSCH system (\ref{nond-mom})-(\ref{nond-chem}) preserve the mass of the binary fluid $\rho$ and the fluid components $\rho c$, ${\rm i.e}$, 
\begin{align}
\int_{\Omega}\rho_{t}~{\rm d}\boldsymbol{x}& =0,~~~~{\rm and}~~~~\int_{\Omega}(\rho c)_{t}~{\rm d}\boldsymbol{x} =0.
\end{align}
\end{thm}
\begin{proof} Using the same scheme for (\ref{cc-11-1}), we obtain the continuity equation of the system
\begin{align}
\rho_{t}+\boldsymbol{\nabla}\cdot (\rho\boldsymbol{u})=0.\label{mass-con-conts}
\end{align}
Taking integration of Eq.(\ref{mass-con-conts}) over $\Omega$ with the help of the homogeneous boundary condition of $\boldsymbol{u}$ in Eq.(\ref{cons-bc}), we obtain the mass conservation for the binary fluid $\int_{\Omega}\rho_{t}{\rm d}\boldsymbol{x}=0$. Multiplying Eq.(\ref{mass-con-conts}) by $c$, and adding to Eq.(\ref{nond-phase}), we obtain
\begin{align}
&(\rho c)_t+\boldsymbol{\nabla}\cdot(\rho\boldsymbol{u}  c)=\cfrac{1}{Pe}\boldsymbol{\nabla} \cdot\big( m(c)\boldsymbol{\nabla}  \bar{\mu}_{c}\big)+\cfrac{\alpha }{Pe}\boldsymbol{\nabla}\cdot\big(  m(c) \boldsymbol{\nabla} \bar{p}\big)=\cfrac{1}{Pe}\boldsymbol{\nabla} \cdot\big( m(c)\boldsymbol{\nabla}  \mu_{c}\big).\label{smass-con-conts}
\end{align}
Taking integration of Eq.(\ref{smass-con-conts}) over $\Omega$ with the help of the homogeneous boundary condition (\ref{cons-bc2}), we obtain the mass conservation for the single fluid $\int_{\Omega}(\rho c)_{t}{\rm d}\boldsymbol{x}=0$.
\end{proof}
\begin{thm}
The non-dimensional q-NSCH system (\ref{nond-mom})-(\ref{nond-chem}) is energy stable, namely the system equations satisfy the following energy dissipation law:
\begin{align}
\cfrac{dE}{dt}=& \cfrac{d}{dt}\bigg(\cfrac{1}{2}\rVert\sqrt{\rho}\boldsymbol{u}\rVert^2_{L^{2}}+ \cfrac{\eta\epsilon}{2We}\rVert\sqrt{\rho}~\boldsymbol{\nabla} c\rVert^2_{L^{2}}+\int_{\Omega}\big(\cfrac{\eta}{\epsilon We}\rho F(c) +\cfrac{1}{Fr}\rho y \big){\rm d}\boldsymbol{x}\bigg)\notag\\
=&-\cfrac{1}{Re} \rVert\sqrt{\mu(c)}\boldsymbol{\nabla} \boldsymbol{u}\rVert^2_{L^2} -\cfrac{1}{3Re}\rVert\sqrt{\mu(c)}\boldsymbol{\nabla} \cdot \boldsymbol{u}\rVert^2_{L^2}-\cfrac{1}{MPe} \rVert\sqrt{m(c)}\boldsymbol{\nabla} \mu_{c} \rVert^2_{L^2}\le 0,\label{energy-law-continuous}
\end{align}
where $E$ is the total energy of the binary fluid, and $||\cdot||_{L^2}$ denotes the norm of $L^2(\Omega)$ in Sobolev spaces.
\end{thm}
\begin{pf}
Multiplying Eq.(\ref{nond-mom}) by $\boldsymbol{u}$, using integration-by-parts and dropping boundary terms, we obtain
\begin{align}
\cfrac{d}{dt}\bigg(\cfrac{1}{2}\rVert\sqrt{\rho}\boldsymbol{u}\rVert^2_{L^{2}}+\int_{\Omega}\big(\cfrac{1}{Fr}\rho y \big){\rm d}\boldsymbol{x}\bigg)&=\int_{\Omega}\bigg(\cfrac{1}{M}(\boldsymbol{\nabla} \cdot \boldsymbol{u}) \bar{p} +\cfrac{1}{M}\rho \bar{\mu}_{c}\boldsymbol{u}\cdot \boldsymbol{\nabla} c\bigg){\rm d}\boldsymbol{x}\notag\\
&\hspace{-10mm}-\cfrac{1}{Re} \rVert\sqrt{\mu(c)}\boldsymbol{\nabla} \boldsymbol{u}\rVert^2_{L^2} -\cfrac{1}{3Re}\rVert\sqrt{\mu(c)}\boldsymbol{\nabla} \cdot \boldsymbol{u}\rVert^2_{L^2},\label{nnewvelocity-12}
\end{align}
where we have used the continuity equation (\ref{mass-con-conts}), the homogeneous boundary condition (\ref{cons-bc}), the following identities
\begin{align}
\int_{\Omega}\rho \boldsymbol{u}\cdot\boldsymbol{j} &~{\rm d}\boldsymbol{x}=\int_{\Omega}\rho  \boldsymbol{u}\cdot \boldsymbol{\nabla }y~{\rm d}\boldsymbol{x}=\int_{\Omega}-\boldsymbol{\nabla }\cdot(\rho  \boldsymbol{u}) y~{\rm d}\boldsymbol{x} =\int_{\Omega}\rho_{t}  y~{\rm d}\boldsymbol{x},\label{gravity-e}\\
&\int_{\Omega}\bigg(\rho (\boldsymbol{u}\cdot \boldsymbol{\nabla} )\cfrac{1}{2}(\boldsymbol{u}\cdot \boldsymbol{u})+\cfrac{1}{2}(\boldsymbol{u}\cdot \boldsymbol{u})\boldsymbol{\nabla} \cdot (\rho \boldsymbol{u})\bigg){\rm d}\boldsymbol{x}=0.\label{homo-bc-u-1}
\end{align}
Multiplying Eq.(\ref{nond-mass}) by $\bar{p}/M$ and using integration-by-parts, we obtain
\begin{align}
&\int_{\Omega}0~{\rm d}\boldsymbol{x}=\int_{\Omega}\bigg(-\cfrac{1}{M}(\boldsymbol{\nabla} \cdot \boldsymbol{u})\bar{p}-\cfrac{\alpha}{PeM}m(c)\boldsymbol{\nabla} \bar{\mu}_{c}\cdot \boldsymbol{\nabla} \bar{p}-\cfrac{\alpha^2}{PeM}m(c)\boldsymbol{\nabla} \bar{p}\cdot \boldsymbol{\nabla} \bar{p}\bigg){\rm d}\boldsymbol{x}.\label{nnewmass-2}
\end{align} 
Multiplying Eq.(\ref{nond-phase}) by $\bar{\mu}_{c}/M$ and using integration-by-parts, we obtain
\begin{align}
0&=\int_{\Omega}\bigg(-\cfrac{1}{M}\rho c_{t}\bar{\mu}_{c}-\cfrac{1}{M} \rho \bar{\mu}_{c} \boldsymbol{u} \cdot\boldsymbol{\nabla}  c -\cfrac{1}{PeM}m(c)\boldsymbol{\nabla} \bar{\mu}_{c}\cdot \boldsymbol{\nabla} \bar{\mu}_{c}\notag\\
&-\cfrac{\alpha}{PeM}m(c)\boldsymbol{\nabla} \bar{\mu}_{c} \cdot\boldsymbol{\nabla} \bar{p}\bigg){\rm d}\boldsymbol{x}.\label{nnewqua3-133}
\end{align}
Multiplying Eq.(\ref{nond-chem}) by $c_t/M$ and using integration-by-parts, we obtain
\begin{align}
& \cfrac{d}{dt}\bigg( \cfrac{\eta\epsilon}{2We}\rVert\sqrt{\rho}~\boldsymbol{\nabla} c\rVert^2_{L^{2}}+\int_{\Omega}\big(\cfrac{\eta}{\epsilon We}\rho F(c)\big){\rm d}\boldsymbol{x}\bigg)=\int_{\Omega}\cfrac{1}{M}\rho c_{t}\bar{\mu}_{c}{\rm d}\boldsymbol{x}.\label{nnqua4-122}
\end{align} 
Summing up the four equations, (\ref{nnewvelocity-12}) and (\ref{nnewmass-2})-(\ref{nnqua4-122}), we obtain the energy dissipation law (\ref{energy-law-continuous}) of the continuous system equations.
\end{pf}
\begin{rmk}
Note that the original non-dimensional q-NSCH system (\ref{nd-mom-before1})-(\ref{nd-chem}) also satisfies the energy law (\ref{energy-law-continuous}), which requires, however, much more complicated test functions to derive the energy law. Our reformulation makes the derivation easier with relatively simpler testing functions being used.
\end{rmk}

\section{Time Discrete, Mass Conservative and Energy Stable Schemes}
We now present two time-discrete methods, namely the primitive method (using the primitive variable formulation), and the projection method (using a projection-type formulation). Both methods are mass conservative and energy stable. In the projection method, we show that there is a pressure-Poisson equation naturally occurring in the reformulated system equation that can be used for solving the pressure. This differs from the traditional projection method in that the latter requires constructing an extra pressure-Poisson equation for solving the pressure. Here we present semi-discrete schemes that motivate the fully discrete schemes that we exhibit in later sections.
\subsection{Time-Discrete Primitive Method}
We first present the following numerical method in the primitive variable formulation. Let $\delta t >0$ denote the time step, and assume $\boldsymbol{u}^{n}$, $\bar{p}^{n}$, $c^{n}$, $\bar{\mu}_{c}^{n}$ are the solution at the time $t=n\delta t$, We then find the solutions at time $t=(n+1)\delta t$ are $\boldsymbol{u}^{n+1}$, $\bar{p}^{n+1}$, $c^{n+1}$, $\bar{\mu}_{c}^{n+1}$ that satisfy 
\begin{align}
&\rho^{n} \cfrac{\boldsymbol{u}^{n+1}-\boldsymbol{u}^{n}}{\delta t}+\rho^{n} (\boldsymbol{u}^{n}\cdot \boldsymbol{\nabla} )\boldsymbol{u}^{n+1}+\cfrac{1}{2}\boldsymbol{u}^{n+1}\big(\cfrac{\rho^{n+1}-\rho^{n}}{\delta t}+\boldsymbol{\nabla} \cdot (\rho^{n} \boldsymbol{u}^{n})\big)=-\cfrac{1}{M} \boldsymbol{\nabla} \bar{p}^{n+1} \notag\\
&+\cfrac{1}{M}\rho^{n} {\bar{\mu}_{c}}^{n+1}\boldsymbol{\nabla} c^{n}+\cfrac{1}{Re}\boldsymbol{\nabla} \cdot (\mu(c^{n}) \boldsymbol{\nabla} \boldsymbol{u}^{n+1})+\cfrac{1}{3Re}\boldsymbol{\nabla}\big(\mu(c^{n})(\boldsymbol{\nabla} \cdot \boldsymbol{u}^{n+1}) \big)-\cfrac{\rho^{n}}{Fr} \boldsymbol{j},\label{mm-mom-1}\\
&\boldsymbol{\nabla} \cdot \boldsymbol{u}^{n+1}=\cfrac{\alpha}{Pe}\boldsymbol{\nabla} \cdot\big( m(c^{n})\boldsymbol{\nabla}  {\bar{\mu}_{c}}^{n+1}\big)+\cfrac{\alpha^2 }{Pe}\boldsymbol{\nabla}\cdot\big(  m(c^{n}) \boldsymbol{\nabla} \bar{p}^{n+1}\big),\label{mm-mass2-1}\\
&\rho^{n+1}\cfrac{c^{n+1}-c^{n}}{\delta t}+\rho^{n}\boldsymbol{u}^{n+1}\cdot \boldsymbol{\nabla}  c^{n}=\cfrac{1}{Pe}\boldsymbol{\nabla} \cdot\big( m(c^{n})\boldsymbol{\nabla}  {\bar{\mu}_{c}}^{n+1}\big)\notag\\
&+\cfrac{\alpha }{Pe}\boldsymbol{\nabla}\cdot\big(  m(c^{n}) \boldsymbol{\nabla} \bar{p}^{n+1}\big),\label{mm-phase-1}\\
&\rho^{n+1}{\bar{\mu}_{c}}^{n+1}=\cfrac{M\eta}{\epsilon We}\rho^{n+\frac{1}{2}}g(c^{n+1},c^{n})+\cfrac{M\eta}{\epsilon We} F^{n+\frac{1}{2}}(c)~r(c^{n+1},c^{n})\notag\\
&+\cfrac{\epsilon\eta M}{2 We}(\boldsymbol{\nabla} c\cdot \boldsymbol{\nabla} c)^{n+\frac{1}{2}} r(c^{n+1},c^{n})-\cfrac{\epsilon\eta M}{ We}\boldsymbol{\nabla} \cdot (\rho^{n+\frac{1}{2}}\boldsymbol{\nabla} c^{n+\frac{1}{2}}),\label{mm-chem-1}
\end{align}
where $\rho^{n+\frac{1}{2}}=(\rho^{n+1}+\rho^{n})/2$, $c^{n+\frac{1}{2}}=(c^{n+1}+c^{n})/2$, $F^{n+\frac{1}{2}}(c)=(F(c^{n+1})+F(c^{n}))/2$, $(\boldsymbol{\nabla} c\cdot \boldsymbol{\nabla} c)^{n+\frac{1}{2}} =(\boldsymbol{\nabla} c^{n+1}\cdot \boldsymbol{\nabla} c^{n+1} + \boldsymbol{\nabla} c^{n}\cdot \boldsymbol{\nabla} c^{n})/2$ are the temporal average, and
\begin{align}
g(c^{n+1},c^n)=\cfrac{1}{4}\big(c^{n+1}(c^{n+1}-1)+c^{n}(c^{n}-1)\big)(c^{n+1}+c^{n}-1)
\end{align}
is an approximation to the nonlinear function $F'(c)=f(c)=c(c-1)(c-1/2)$. Here we note the identity,
\begin{align}
&F(c^{n+1})-F(c^{n})=g(c^{n+1},c^n)(c^{n+1}-c^{n}).
\end{align}
\begin{align}
r(c^{n+1},c^{n})=-\alpha \rho(c^{n+1})\rho(c^{n}),\label{def-rhoal}
\end{align}
is an approximation of the nonlinear function $\partial \rho/\partial c=-\alpha \rho^2$, which satisfies the following identity:
\begin{align}
&\rho(c^{n+1})-\rho(c^{n})=r(c^{n+1},c^{n})(c^{n+1}-c^{n}).\label{def-r}
\end{align}
Note that the function $g(c^{n+1},c^{n})$ and $r(c^{n+1},c^{n})$ are critical for the achievement of the mass conservation and energy stability of the numerical schemes. \\
Using the following boundary conditions
\begin{align}
{\boldsymbol{u}^{n+1}|}_{\partial\Omega}=0,~~~~\boldsymbol{n}\cdot \boldsymbol{\nabla}c^{n+1}\rvert_{\partial \Omega}=\boldsymbol{n}\cdot \boldsymbol{\nabla}\mu_{c}^{n+1}\rvert_{\partial \Omega}=0,\label{mm-bc}
\end{align}
we now show that the time-discrete primitive method (\ref{mm-mom-1})-(\ref{mm-chem-1}) satisfies the following properties:
\begin{thm}
The time-discrete primitive method (\ref{mm-mom-1})-(\ref{mm-chem-1}) is mass conservative for the binary fluid and the fluid components, ${\rm i.e}$, 
\begin{align}
\big(\rho^{n+1},1\big)_{L^2}&=\big(\rho^{n},1\big)_{L^2},~~{\rm and}~~\big(\rho^{n+1} c^{n+1},1\big)_{L^2}=\big(\rho^{n}c^{n},1\big)_{L^2},~~\forall n\ge 0.
\end{align}
Here we use $||\cdot||_{L^2}$ to denote the norm of $L^2(\Omega)$ in Sobolev spaces, and $(\cdot ,\cdot)_{L^2}$ denotes the inner product in $L^2(\Omega)$.
\end{thm}\label{pm-dis-mass-theorem}
\begin{pf}Multiplying (\ref{mm-mass2-1}) by $1/\alpha$ and substituting into (\ref{mm-phase-1}), we obtain
\begin{align}
&\rho^{n+1}\cfrac{c^{n+1}-c^{n}}{\delta t}+ \rho^{n}\boldsymbol{u}^{n+1} \cdot\boldsymbol{\nabla}c^{n}=\cfrac{1}{\alpha}\boldsymbol{\nabla} \cdot  \boldsymbol{u}^{n+1}.\label{new-sub}
\end{align}
Multiplying the above by $-\alpha\rho^{n}$, we obtain the continuity equation at the time discrete level
\begin{align}
\cfrac{\rho^{n+1}-\rho^{n}}{\delta t}+ \boldsymbol{\nabla} \cdot  (\rho^n\boldsymbol{u}^{n+1})=0,\label{part-phvar-2a}
\end{align}
where we have used the identity (\ref{def-rhoal}) and (\ref{def-r}). Integrating of Eq.(\ref{part-phvar-2a}) over $\Omega$, thanks to the boundary condition (\ref{mm-bc}), we obtain the mass conservation for the binary fluid $(\rho^{n+1},1)_{L^2}=(\rho^{n},1)_{L^2}$. Multiplying (\ref{part-phvar-2a}) by $c^{n}$, and adding to Eq.(\ref{mm-phase-1}), we obtain
\begin{align}
\cfrac{\rho^{n+1}c^{n+1}-\rho^{n}c^{n}}{\delta t}+\boldsymbol{\nabla}\cdot(\rho^{n}\boldsymbol{u}^{n+1} c^{n})&=\cfrac{1}{Pe}\boldsymbol{\nabla} \cdot\big( m(c^{n})\boldsymbol{\nabla}  {\bar{\mu}_{c}}^{n+1}\big)\notag\\
&+\cfrac{\alpha }{Pe}\boldsymbol{\nabla}\cdot\big(  m(c^{n}) \boldsymbol{\nabla} \bar{p}^{n+1}\big).\label{part-phvar-2bb}
\end{align}
Integrating of Eq.(\ref{part-phvar-2bb}) over $\Omega$, thanks to the definition (\ref{new-mu}) and boundary condition (\ref{mm-bc}), we obtain the mass conservation for the single fluid $(\rho^{n+1}c^{n+1},1)_{L^2}=(\rho^{n}c^{n},1)_{L^2}$.
\end{pf}
\begin{thm}
The time-discrete primitive method (\ref{mm-mom-1})-(\ref{mm-chem-1}) is energy stable, ${\rm i.e}$, 
\begin{align}
E^{n+1}-E^{n}&=\bigg(\cfrac{1}{2} ||\sqrt{\rho^{n+1}}\boldsymbol{u}^{n+1}||^2_{L^2}+\cfrac{\eta\epsilon}{2We}||\sqrt{\rho^{n+1}}\boldsymbol{\nabla} c^{n+1}||^2_{L^2}\notag\\
&+\int_{\Omega}\big(\cfrac{\eta}{\epsilon We}\rho^{n+1}F(c^{n+1})+\cfrac{1}{Fr}\rho^{n+1} y\big){\rm d}\boldsymbol{x}\bigg)\notag\\
&-\bigg(\cfrac{1}{2} ||\sqrt{\rho^{n}}\boldsymbol{u}^{n}||^2_{L^2}+\cfrac{\eta\epsilon}{2We}||\sqrt{\rho^{n}}\boldsymbol{\nabla} c^{n}||^2_{L^2}\notag\\
&+\int_{\Omega}\big(\cfrac{\eta}{\epsilon We}\rho^{n}F(c^{n})+\cfrac{1}{Fr}\rho^{n} y\big){\rm d}\boldsymbol{x}\bigg)\notag\\
&=-\cfrac{\delta t}{Re}||\sqrt{\mu(c^{n})}\boldsymbol{\nabla} \boldsymbol{u}^{n+1}||^2_{L^2}-\cfrac{\delta t}{3Re}||\sqrt{\mu(c^{n})}\boldsymbol{\nabla} \cdot \boldsymbol{u}^{n+1}||^2_{L^2}\notag\\
&-\cfrac{\delta t}{MPe}||\sqrt{m(c^{n})}\boldsymbol{\nabla} \mu_{c}^{n+1}||^2_{L^2}- \cfrac{1}{2}||\sqrt{\rho^{n}}(\boldsymbol{u}^{n+1}-\boldsymbol{u}^{n})||^2_{L^2}\le 0,\label{sdfas-e}
\end{align}
where $E^{n+1}$ is the total energy at the time discrete level.
\end{thm}
\begin{proof}
Multiplying Eq.(\ref{mm-mom-1}) by $\delta t \boldsymbol{u}^{n+1}$, using integration-by-parts and dropping the boundary terms, we obtain
\begin{align}
\cfrac{1}{2} ||\sqrt{\rho^{n+1}}&\boldsymbol{u}^{n+1}||^2_{L^2}- \cfrac{1}{2}||\sqrt{\rho^{n}}\boldsymbol{u}^{n}||^2_{L^2}+\cfrac{1}{Fr}\int_{\Omega}\bigg((\rho^{n+1}-\rho^{n})y\bigg){\rm d}\boldsymbol{x}\notag\\
&=\cfrac{\delta t}{M} \big(\bar{p}^{n+1},(\boldsymbol{\nabla} \cdot \boldsymbol{u}^{n+1})\big)_{L^2}+\cfrac{\delta t}{M}\big(\rho^{n} \boldsymbol{u}^{n+1}\cdot \boldsymbol{\nabla} c^{n},{\bar{\mu}_{c}}^{n+1}\big)_{L^2}\notag\\
&- \cfrac{\delta t}{Re}||\sqrt{\mu(c^{n})}\boldsymbol{\nabla} \boldsymbol{u}^{n+1}||^2_{L^2}-\cfrac{\delta t}{3Re}||\sqrt{\mu(c^{n})}\boldsymbol{\nabla} \cdot \boldsymbol{u}^{n+1}||^2_{L^2}\notag\\
&-\cfrac{1}{2} ||\sqrt{\rho^{n}}(\boldsymbol{u}^{n+1}-\boldsymbol{u}^{n})||^2_{L^2},\label{method-inner-mom2-1}
\end{align}
where we have used the homogeneous boundary condition (\ref{mm-bc}), the identities (\ref{gravity-e}), (\ref{homo-bc-u-1}), (\ref{part-phvar-2a}) and
\begin{align}
\int_{\Omega}\bigg(\rho^{n}\boldsymbol{j}\cdot \boldsymbol{u}^{n+1}\bigg){\rm d }\boldsymbol{x}=\int_{\Omega}\bigg(\rho^{n}\boldsymbol{u}^{n+1}\cdot \boldsymbol{\nabla}y\bigg){\rm d }\boldsymbol{x}&=-\int_{\Omega}\bigg(\boldsymbol{\nabla}\cdot(\rho^{n}\boldsymbol{u}^{n+1}) y\bigg){\rm d }\boldsymbol{x}\notag\\
&=\int_{\Omega}\bigg(\cfrac{\rho^{n+1}-\rho^{n}}{\delta t} y\bigg){\rm d }\boldsymbol{x}.\label{potential}
\end{align}
Note that for all the following derivations, the boundary terms originated from the integration-by-parts can be dropped by using the homogeneous boundary condition (\ref{mm-bc}).\\
Multiplying Eq.(\ref{mm-mass2-1}) by $\delta t \bar{p}^{n+1}/ M$ and using integration-by-parts, we obtain
\begin{align}
0=&-\cfrac{\delta t}{M}\big((\boldsymbol{\nabla} \cdot \boldsymbol{u}^{n+1}),\bar{p}^{n+1}\big)_{L^2}-\cfrac{\delta t\alpha}{M Pe}\big( m(c^{n})\boldsymbol{\nabla} {\bar{\mu}_{c}}^{n+1}, \boldsymbol{\nabla} \bar{p}^{n+1}\big)_{L^2}\notag\\
&-\cfrac{\delta t\alpha^2}{M Pe} ||\sqrt{m(c^{n})}\boldsymbol{\nabla} \bar{p}^{n+1}||^2_{L^2}.\label{method-inner-mass-1-1}
\end{align}
Multiplying Eq.(\ref{mm-phase-1}) by $\delta t{\bar{\mu}_{c}}^{n+1}/M$ and using integration-by-parts, we obtain
\begin{align}
0=&-\cfrac{1}{M}\big(\rho^{n+1}(c^{n+1}-c^{n}),{\bar{\mu}_{c}}^{n+1}\big)_{L^2}-\cfrac{\delta t}{M}\big(\rho^{n}\boldsymbol{u}^{n+1}\cdot\boldsymbol{\nabla} c^{n}, {\bar{\mu}_{c}}^{n+1}\big)_{L^2}\notag\\
& -\cfrac{\delta t }{MPe}||\sqrt{ m(c^{n})}\boldsymbol{\nabla} {\bar{\mu}_{c}}^{n+1} ||^2_{L^{2}} -\cfrac{\delta t \alpha}{MPe}\big(m(c^{n})\boldsymbol{\nabla} {\bar{\mu}_{c}}^{n+1}, \boldsymbol{\nabla} \bar{p}^{n+1}\big)_{L^2}.\label{method-inner-phvar-1}
\end{align}
Multiplying (\ref{mm-chem-1}) by $(c^{n+1}-c^{n})/ M$ and using integration-by-parts, we obtain
\begin{align}
&\cfrac{\eta\epsilon}{2We}(||\sqrt{\rho^{n+1}}\boldsymbol{\nabla} c^{n+1}||^2_{L^2}-||\sqrt{\rho^{n}}\boldsymbol{\nabla} c^{n}||^2_{L^2})+\cfrac{\eta}{\epsilon We}\big(\int_{\Omega}\rho^{n+1}F(c^{n+1}){\rm d}\boldsymbol{x}\notag\\
&-\int_{\Omega}\rho^{n}F(c^{n}){\rm d} \boldsymbol{x}\big)=\cfrac{1}{M}\big(\rho^{n+1}{\bar{\mu}_{c}}^{n+1},c^{n+1}-c^{n}\big)_{L^2}.\label{method-inner-chem-1-1}
\end{align}
Summing up the four relations, (\ref{method-inner-mom2-1}) and (\ref{method-inner-mass-1-1})-(\ref{method-inner-chem-1-1}), we obtain the energy stability (\ref{sdfas-e}) for time-discrete primitive method.
\end{proof}
\subsection{Time-Discrete Projection Method}
To design an efficient projection-type methods, we follow the projection formulation to decouple the computation of the velocity $\boldsymbol{u}$ and pressure $p$. In particular, an intermediate velocity that does not satisfy the quasi-incompressibility constraint (\ref{nond-mass}) is computed at each time step, the pressure is then used to correct the intermediate velocity to get the next updated velocity that satisfies the quasi-incompressible constraint. Our projection method differs from traditional projection methods in that traditional methods usually use the pressure to project the intermediate velocity onto a space of divergence-free velocity fields (See \cite{JieSIAM2010} as example). Moreover, in our projection method, a pressure-Poisson equation (\ref{pj-mass}) naturally occurs in the reformulated system equations, and an extra pressure-Poisson equation is not required, which, however, is usually compulsory in traditional projection methods (This can be done applying the divergence operator to Eq.(\ref{pj-pj}), see \cite{Jie-proj} as a review). Our projection method for the q-NSCH system (\ref{nond-mom})-(\ref{nond-chem}) is the following: given $\tilde{\boldsymbol{u}}^{n}$, $\boldsymbol{u}^{n}$, $\bar{p}^{n}$, $c^{n}$, $\bar{\mu}_{c}^{n}$, find the solution $\tilde{\boldsymbol{u}}^{n+1}$, $\boldsymbol{u}^{n+1}$, $\bar{p}^{n+1}$, $c^{n+1}$, $\bar{\mu}_{c}^{n+1}$ satisfying 
\begin{align}
&\rho^{n} \cfrac{\tilde{\boldsymbol{u}}^{n+1}-\boldsymbol{u}^{n}}{\delta t}+\rho^{n} (\boldsymbol{u}^{n}\cdot \boldsymbol{\nabla} )\tilde{\boldsymbol{u}}^{n+1}+\cfrac{1}{2}\tilde{\boldsymbol{u}}^{n+1}\big(\cfrac{\rho^{n+1}-\rho^{n}}{\delta t}+\boldsymbol{\nabla} \cdot (\rho^{n} \boldsymbol{u}^{n})\big)\notag\\
&=\cfrac{1}{Re}\boldsymbol{\nabla} \cdot (\mu(c^{n}) \boldsymbol{\nabla}\tilde{ \boldsymbol{u}}^{n+1})+\cfrac{1}{3Re}\boldsymbol{\nabla}\big(\mu(c^{n})(\boldsymbol{\nabla} \cdot \tilde{\boldsymbol{u}}^{n+1}) \big),\label{pj-mom}\\
&\rho^{n+1} \cfrac{\boldsymbol{u}^{n+1}-\tilde{\boldsymbol{u}}^{n+1}}{\delta t}=-\cfrac{1}{M} \boldsymbol{\nabla} \bar{p}^{n+1}+\cfrac{1}{M}\rho^{n+1}{\bar{\mu}_{c}}^{n+1}  \boldsymbol{\nabla}c^{n+1}-\cfrac{\rho^{n+1}}{Fr} \boldsymbol{j}\label{pj-pj}\\
&\boldsymbol{\nabla} \cdot \boldsymbol{u}^{n+1}=\cfrac{\alpha}{Pe}\boldsymbol{\nabla} \cdot\big( m(c^{n})\boldsymbol{\nabla}  {\bar{\mu}_{c}}^{n+1}\big)+\cfrac{\alpha^2 }{Pe}\boldsymbol{\nabla}\cdot\big(  m(c^{n}) \boldsymbol{\nabla} \bar{p}^{n+1}\big),\label{pj-mass}\\
&\rho^{n}\cfrac{c^{n+1}-c^{n}}{\delta t}+\rho^{n+1}\boldsymbol{u}^{n+1} \cdot\boldsymbol{\nabla} c^{n+1}=\cfrac{1}{Pe}\boldsymbol{\nabla} \cdot\big( m(c^{n})\boldsymbol{\nabla}  {\bar{\mu}_{c}}^{n+1}\big)\notag\\
&\hspace{52mm}+\cfrac{\alpha }{Pe}\boldsymbol{\nabla}\cdot\big(  m(c^{n}) \boldsymbol{\nabla} \bar{p}^{n+1}\big),\label{pj-phase}\\
&\rho^{n}{\bar{\mu}_{c}}^{n+1}=\cfrac{M\eta}{\epsilon We}\rho^{n+\frac{1}{2}}g(c^{n+1},c^{n})+\cfrac{M\eta}{\epsilon We} F^{n+\frac{1}{2}}(c)~r(c^{n+1},c^{n})\notag\\
&+\cfrac{\epsilon\eta M}{2 We}(\boldsymbol{\nabla} c\cdot \boldsymbol{\nabla} c)^{n+\frac{1}{2}}~r(c^{n+1},c^{n})-\cfrac{\epsilon\eta M}{ We}\boldsymbol{\nabla} \cdot (\rho_{n+\frac{1}{2}}\boldsymbol{\nabla} c_{n+\frac{1}{2}}),\label{pj-chem}
\end{align}
with the following boundary conditions
\begin{align}
&{\tilde{\boldsymbol{u}}^{n+1}|}_{\partial\Omega}=0,~~~~\boldsymbol{n}\cdot{\boldsymbol{u}|^{n+1}}_{\partial \Omega}=0,~~~~\boldsymbol{n}\cdot \boldsymbol{\nabla}c^{n+1}\rvert_{\partial \Omega}=\boldsymbol{n}\cdot \boldsymbol{\nabla}\mu_{c}^{n+1}\rvert_{\partial \Omega}=0\label{pj-bc}
\end{align}
Here the intermediate velocity $\tilde{\boldsymbol{u}}$ is solved first in (\ref{pj-mom}), then the pressure $\bar{p}$ is solved in (\ref{pj-mass}) and is used to correct $\tilde{\boldsymbol{u}}$ to obtain the velocity $\boldsymbol{u}$ that satisfies the quasi-incompressible constraint through the projection equation (\ref{pj-pj}).
\begin{thm}
The time-discrete projection scheme (\ref{pj-mom})-(\ref{pj-chem}) is mass conservative for the binary fluid and single fluid, ${\rm i.e}$, 
\begin{align}
\big(\rho^{n+1},1\big)_{L^2}&=\big(\rho^{n},1\big)_{L^2},~~{\rm and}~~\big(\rho^{n+1} c^{n+1},1\big)_{L^2}=\big(\rho^{n}c^{n},1\big)_{L^2},~~\forall n\ge 0.
\end{align}
\end{thm}
\begin{proof}
Using the same strategy that used for Theorem {\bf{\ref{pm-dis-mass-theorem}}}, we obtain the mass conservation equation from our time-discrete projection method:
\begin{align}
\cfrac{\rho^{n+1}-\rho^{n}}{\delta t}+ \boldsymbol{\nabla} \cdot  (\rho^{n+1}\boldsymbol{u}^{n+1})=0,\label{part-phvar-2a1}
\end{align}
where we have used the identity (\ref{def-r}). Integrating of Eq.(\ref{part-phvar-2a1}) over $\Omega$, thanks to the homogeneous boundary conditions (\ref{pj-bc}), we obtain the mass conservation for the binary fluid $(\rho^{n+1},1)_{L^2}=(\rho^{n},1)_{L^2}$. Again, multiplying Eq.(\ref{part-phvar-2a1}) by $c^{n+1}$, and adding to Eq.(\ref{pj-phase}) we obtain
\begin{align}
&\cfrac{\rho^{n+1}c^{n+1}-\rho^{n}c^{n}}{\delta t}+\boldsymbol{\nabla}\cdot(\rho^{n+1}\boldsymbol{u}^{n+1} c^{n+1})=\cfrac{1}{Pe}\boldsymbol{\nabla} \cdot\big( m(c^{n})\boldsymbol{\nabla}  {\bar{\mu}_{c}}^{n+1}\big)\notag\\
&+\cfrac{\alpha }{Pe}\boldsymbol{\nabla}\cdot\big(  m(c^{n}) \boldsymbol{\nabla} \bar{p}^{n+1}\big).\label{part-phvar-2bc}
\end{align}
Integrating of Eq.(\ref{part-phvar-2bc}) over $\Omega$, thanks to the boundary condition (\ref{pj-bc}), we obtain the mass conservation for fluid components $(\rho^{n+1}c^{n+1},1)_{L^2}=(\rho^{n}c^{n},1)_{L^2}$.
\end{proof}
\begin{thm}\label{energy-pj}
The time-discrete projection scheme (\ref{pj-mom})-(\ref{pj-chem}) is energy stable, ${\rm i.e}$, 
\begin{align}
&E^{n+1}-E^{n}=\notag\\
&\bigg(\cfrac{1}{2} ||\sqrt{\rho^{n+1}}\boldsymbol{u}^{n+1}||^2_{L^2}+\cfrac{\eta\epsilon}{2We}||\sqrt{\rho^{n+1}}\boldsymbol{\nabla} c^{n+1}||^2_{L^2}+\int_{\Omega}\big(\cfrac{\eta}{\epsilon We}\rho^{n+1}F(c^{n+1})+\cfrac{1}{Fr}\rho^{n+1}  y\big){\rm d}\boldsymbol{x}\bigg)\notag\\
&-\bigg(\cfrac{1}{2} ||\sqrt{\rho^{n}}\boldsymbol{u}^{n}||^2_{L^2}+\cfrac{\eta\epsilon}{2We}||\sqrt{\rho^{n}}\boldsymbol{\nabla} c^{n}||^2_{L^2}+\int_{\Omega}\big(\cfrac{\eta}{\epsilon We}\rho^{n}F(c^{n})+\cfrac{1}{Fr}\rho^{n}  y\big){\rm d}\boldsymbol{x}\bigg)\notag\\
&=-\cfrac{\delta t}{Re}||\sqrt{\mu(c^{n})}\boldsymbol{\nabla} \tilde{\boldsymbol{u}}^{n+1}||^2_{L^2}-\cfrac{\delta t}{3Re}||\sqrt{\mu(c^{n})}\boldsymbol{\nabla} \cdot \tilde{\boldsymbol{u}}^{n+1}||^2_{L^2}\notag\\
&-\cfrac{\delta t}{MPe}||\sqrt{m(c^{n})}\boldsymbol{\nabla} \mu_{c}^{n+1}||^2_{L^2}- \cfrac{1}{2}||\sqrt{\rho^{n}}(\tilde{\boldsymbol{u}}^{n+1}-\boldsymbol{u}^{n})||^2_{L^2}\notag\\
&-\cfrac{1}{2} ||\sqrt{\rho^{n+1}}(\boldsymbol{u}^{n+1}-\tilde{\boldsymbol{u}}^{n+1})||^2_{L^2}\le 0,\label{adfsafds-dasd}
\end{align}
where $E^{n+1}$ is the total energy at time discrete level.
\end{thm}
\begin{rmk}
Here we omit the details of the proof, as the derivations here are similar with the proof for the primitive method in many aspects. The primary differences are that in the projection method, we have one more projection equation (\ref{pj-pj}), and the mass conservation equation (\ref{part-phvar-2a1}) here is slightly different to that of the primitive methods, moreover the test functions are different as well. In particular, to show the energy stability of the projection method, we multiply Eq.(\ref{pj-mom}) by $\delta t\tilde{\boldsymbol{u}}^{n+1}$, Eq.(\ref{pj-pj}) by $\delta t\boldsymbol{u}^{n+1}$, Eq.(\ref{pj-mass}) by $\delta t\bar{p}^{n+1}/ M$, Eq.(\ref{pj-phase}) by ${\delta t\bar{\mu}_{c}}^{n+1}/M$ and Eq.(\ref{pj-chem}) by $(c^{n+1}-c^{n})/M$. After using the integration-by-parts with the homogeneous boundary conditions, we sum up the resulted relations to obtain the energy stability for the projection method (\ref{adfsafds-dasd}). 
\end{rmk}
\begin{rmk}
Both numerical methods are highly coupled and non-linear, however we never observe a problem with existence and uniqueness of the solution in all our extensive numerical experiments. Here we refer to \cite{Han20177} for some results and analysis methods about this issue.
\end{rmk}
\section{Finite Difference Discretization on Staggered Grid}
Before we present our fully discrete finite difference schemes, we first show some basic definitions and notations for the finite difference discretization on a staggered grid. Here we use the notation and results for cell-centered functions from \cite{SteveMultigrid-2010, SteveMultigrid2013, SteveMultigrid-2009}. Let $\Omega = (0,L_x )\times(0,L_y )$, with $L_x = m_{1} \cdot h$ and $L_y = m_{1} \cdot h$, where $m_{1}$ and $m_{2}$ are positive integers and $h > 0$ is the spatial step size. For simplicity we assume that $L_{x}=L_{y}$. Consider the following four sets
\begin{align}
E_{m_{1}} &=\{ x_{i+\frac{1}{2}} | i = 0,\cdots,{m_{1}}  \},~~~~E_{\overline{m}_{1}} =\{ x_{i+\frac{1}{2}} | i = -1,\cdots,{m_{1}+1}  \},\\
C_{m_{1}}&=\{  x _{i} | i = 1, \cdots,m_{1} \},~~~~~~~~C_{\overline{m}_{1}}=\{  x_{i} | i = 0, \cdots,m_{1}+1\},
\end{align}
where $x_{i+\frac{1}{2}}=i\cdot h$ and $x_{i}=(i-\frac{1}{2})\cdot h$. Here $E_{m_{1}}$ and $E_{\overline{m}_{1}}$ are called the uniform partition of $[0,L_{x}]$ of size $m_1$, and its elements are called edge-centered points. The two points belonging to $E_{\overline{m}_1}$\big\backslash $E_{m_1}$ are called ghost points. The elements of $C_{m_1}$ and $C_{\overline{m}_1}$ are called cell-centered points. Again, the two points belonging to $C_{\overline{m}_1}$\big\backslash $C_{m_1}$ are called ghost points. Analogously, the sets $E_{m_{2}}$ and $E_{\overline{m}_{2}}$ contain the edge-centered points, and $C_{m_2}$ and $C_{\overline{m}_{2}}$ contain the cell-centered points of the interval $[0,L_y]$.\\
We then define the following function spaces
\begin{align}
&\mathcal{C}_{m_{1}\times m_{2}} = \{ \phi: C_{m_{1}}\times C_{m_{2}}\rightarrow  \boldsymbol{R}\},\hspace{2 mm} \mathcal{V}^{vc}_{m_{1}\times m_{2}} = \{f: E_{m_{1}}\times E_{m_{2}}\rightarrow  \boldsymbol{R}\},\\
&\mathcal{E}^{ew}_{m_{1}\times m_{2}} = \{u: E_{m_{1}}\times C_{m_{2}}\rightarrow  \boldsymbol{R}\},\hspace{2 mm} \mathcal{E}^{ns}_{m_{1}\times m_{2}} = \{ v: C_{m_{1}}\times E_{m_{2}}\rightarrow  \boldsymbol{R}\},
\end{align}
for cell-centered functions, vertex-centered functions, east-west edge-centered functions and north-south edge-centered functions respectively. Due to the different locations of the functions, we define several average and difference operators as follows:
\begin{align}
{\rm edge~to~center~average~and~difference}:&~a_{x},~a_{y},~d_x,~d_{y};\notag\\
{\rm center~to~edge~average~and~difference}:&~A_{x},~A_{y},~D_x,~D_{y};\notag\\
{\rm vertex~to~edge~average~and~difference}:&~\mathfrak{A}_{x},~\mathfrak{A}_{y},~\mathfrak{D}_{x},~\mathfrak{D}_{y};\notag\\
{\rm edge~to~vertex~average~and~difference}:&~\mathcal{A}_{x},~\mathcal{A}_{y},~\mathcal{D}_{x},~\mathcal{D}_{y};\notag\\
{\rm center~to~vertex~average}:&~\mathcal{A}.\notag
\end{align}
We also define an average operator $\boldsymbol{A}=(\begin{smallmatrix} 
A_{x} & 0 \\
0 & A_{y}
\end{smallmatrix})$ and the following divergence operator:
\begin{align}
&\boldsymbol{\nabla}_{\hspace{-1mm}d}=(d_{x},d_{y}),~~\boldsymbol{\nabla}_{\hspace{-1mm}D}=(D_{x},D_{y}),~~\boldsymbol{\nabla}_{\hspace{-1mm}(d,\mathcal{D})}=(d_{x},\mathcal{D}_{y}),\notag\\
&\boldsymbol{\nabla}_{\hspace{-1mm}(\mathcal{D},d)}=(\mathcal{D}_{x},d_{y}),~~\boldsymbol{\nabla}_{\hspace{-1mm}(D,\mathfrak{D})}=(D_{x},\mathfrak{D}_{y}),~~~~\boldsymbol{\nabla}_{\hspace{-1mm}(\mathfrak{D},D)}=(\mathfrak{D}_{x},D_{y}).
\end{align}
We refer the reader to $\ref{app-A1}$ and $\ref{app-A2}$ for a description of our notations for the above spaces and operators. Moreover $(\cdot ,\cdot)_{2}$, and $[\cdot,\cdot ]_{ew}$, $[\cdot,\cdot ]_{ns}$, and $\langle\cdot ,\cdot \rangle_{vc}$ denote the fully discrete inner product of the cell-centered, edge-centered and vertex-centered variables respectively which are defined in $\ref{app-A3}$. \\
Note that in this paper, the cell-centered functions are the phase variable $c$, chemical potential $\mu_{c}$, $\bar{\mu}_{c}$, and pressure $p$ and $\bar{p}$, the east-west edge-centered function is the x-component of the velocity, $u$ and $\tilde{u}$(for projection method), and the north-south edge-centered function is the y-component of the velocity, $v$ and $\tilde{v}$ (for the projection method). 
\section{Fully Discrete Mass Conservative, Energy Stable Schemes}
In this section we describe and analyze our two staggered grid finite difference schemes for q-NSCH model. We show that the property of mass conservation and energy stability can be achieved at the fully discrete level for both schemes.
\subsection{Fully Discrete Primitive Method}
The fully-discrete scheme for the primitive method (\ref{mm-mom-1})-(\ref{mm-chem-1}) is the following: Let $\delta t>0$ represent the time step, and the grid functions $c^{n}, {\bar{\mu}}_{c}^{n}, \bar{p}^{n} \in \mathcal {C}_{\overline{m}_{1}\times \overline{m}_{2}}$, $u^{n}\in  \mathcal {E}^{ew}_{m_{1}\times m_{2}}$ and $v^{n}\in \mathcal {E}^{ns}_{m_{1}\times m_{2}}$, and $\boldsymbol{u}^{n}=(u^{n},v^{n})$ be the solution at time $t=n\delta t$, find $c^{n+1}, {\bar{\mu}}^{n+1}_{c}, \bar{p}^{n+1} \in  \mathcal {C}_{\overline{m}_{1}\times \overline{m}_{2}}$, $u^{n+1}\in  \mathcal {E}^{ew}_{m_{1}\times m_{2}}$, $v^{n+1}\in \mathcal {E}^{ns}_{m_{1}\times m_{2}}$, and $\boldsymbol{u}^{n+1}=(u^{n+1},v^{n+1})$ at $t=(n+1)\delta t$ such that:
\begin{align}
&\boldsymbol{A}\rho^{n} \cfrac{\boldsymbol{u}^{n+1}-\boldsymbol{u}^{n}}{\delta t}+\rho^{n}\boldsymbol{u}^{n}\cdot \boldsymbol{\nabla}\boldsymbol{u}^{n+1}+\cfrac{\boldsymbol{A}\rho^{n+1}-\boldsymbol{A}\rho^{n}}{2\delta t}\boldsymbol{u}^{n+1}+\cfrac{1}{2}\boldsymbol{\nabla}\cdot (\rho ^{n}\boldsymbol{u} ^{n})\boldsymbol{u}^{n+1}\notag\\
&=-\cfrac{1}{M} \boldsymbol{\nabla}_{\hspace{-1mm}D}~\bar{p}^{n+1} +\cfrac{1}{M}\rho^{n} {\bar{\mu}_{c}}^{n+1}\boldsymbol{\nabla}~c^{n}+\cfrac{1}{Re}\boldsymbol{\nabla}_{\hspace{-1mm}(D,\mathfrak{D})}\cdot (\mu(c^{n}) \boldsymbol{\nabla}_{\hspace{-1mm}(d,\mathcal{D})} \boldsymbol{u}^{n+1})\notag\\
&+\cfrac{1}{3Re}\boldsymbol{\nabla}_{\hspace{-1mm}D}\big(\mu(c^{n})\boldsymbol{\nabla}_{\hspace{-1mm}d}\cdot  \boldsymbol{u}^{n+1} \big)-\cfrac{1}{Fr}\boldsymbol{A}\rho^{n} g \boldsymbol{j},\label{pm-num-mx}\\
&\boldsymbol{\nabla}_{\hspace{-1mm}d}\cdot \boldsymbol{u}^{n+1}=\cfrac{\alpha}{Pe} \boldsymbol{\nabla}_{\hspace{-1mm}d} \cdot   \big(\boldsymbol{ A}m(c^{n}) \boldsymbol{\nabla}_{\hspace{-1mm}D} {\bar{\mu}_{c}}^{n+1}\big)+\cfrac{\alpha^2}{Pe} \boldsymbol{\nabla}_{\hspace{-1mm}d} \cdot   \big(\boldsymbol{ A}m(c^{n}) \boldsymbol{\nabla}_{\hspace{-1mm}D} {\bar{p}}^{n+1}\big),\label{pm-num-mass}\\
&\rho^{n+1}\cfrac{c^{n+1}-c^{n}}{\delta t}+\rho^{n}\boldsymbol{u}^{n+1}\cdot \boldsymbol{\nabla}c^n=\cfrac{1}{Pe} \boldsymbol{\nabla}_{\hspace{-1mm}d} \cdot   \big(\boldsymbol{ A}m(c^{n}) \boldsymbol{\nabla}_{\hspace{-1mm}D} {\bar{\mu}_{c}}^{n+1}\big)\notag\\
&+\cfrac{\alpha}{Pe} \boldsymbol{\nabla}_{\hspace{-1mm}d} \cdot   \big(\boldsymbol{ A}m(c^{n}) \boldsymbol{\nabla}_{\hspace{-1mm}D} {\bar{p}}^{n+1}\big),\label{pm-num-phase}\\
&\rho^{n+1}{\bar{\mu}_{c}}^{n+1}=\cfrac{M\eta}{\epsilon We}\rho^{n+\frac{1}{2}}g(c^{n+1},c^{n})+\cfrac{M\eta}{\epsilon We} F^{n+\frac{1}{2}}(c)~r(c^{n+1},c^{n})\notag\\
&+\cfrac{\epsilon\eta M}{2 We} \lVert\boldsymbol{\nabla }_{D}c\rVert_{2}^{n+\frac{1}{2}}~r(c^{n+1},c^{n})-\cfrac{\epsilon\eta M}{ We}\boldsymbol{\nabla}_{\hspace{-1mm}d}\cdot  (\boldsymbol{A}\rho^{n+\frac{1}{2}}\boldsymbol{\nabla}_{\hspace{-1mm}D}~c^{n+\frac{1}{2}}).\label{pm-num-chem}
\end{align}
In Eq.(\ref{pm-num-mx}), we let
\begin{align}
\rho^{n}\boldsymbol{u}^{n}\cdot \boldsymbol{\nabla}\boldsymbol{u}^{n+1}&=\begin{pmatrix}
A_{x}\big(\rho^{n} a_{x}u^{n}d_{x}u^{n+1}\big)+ \mathfrak{A}_{y}\big(\mathcal{A}\rho^{n}\mathcal{A}_{x}v^{n} \mathcal{D}_{y}u^{n+1}\big)\\
 \mathfrak{A}_{x} (\mathcal{A}\rho^{n} \mathcal{A}_{y}u^{n}\mathcal{D}_{x} v^{n+1}\big)+A_{y}\big(\rho^{n}a_{y}v^{n}d_{y} v^{n+1}\big)
\end{pmatrix},\label{e1e1-1}\\
\cfrac{1}{2}\boldsymbol{\nabla}\cdot (\rho^{n} \boldsymbol{u}^{n})\boldsymbol{u}^{n+1}&=\begin{pmatrix}
\big(D_{x} (\mathcal{A}\rho^{n} \mathcal{A}_{y}u^{n})+\mathfrak{D}_{y} (\rho^{n} a_{y}v^{n})\big)u^{n+1}\\
\big(\mathfrak{D}_{x} (\mathcal{A}\rho^{n} \mathcal{A}_{y}u^{n})+D_{y} (\rho^{n} a_{y}v^{n})\big)v^{n+1}
\end{pmatrix},\label{e1e1-2}\\
\rho^{n} {\bar{\mu}_{c}}^{n+1}\boldsymbol{\nabla}c^{n}&=\begin{pmatrix}
A_{x}(\rho^{n}_{ew} {\bar{\mu}_{c}}^{n+1})D_{x} c^{n}\\
A_{y}(\rho^{n}_{ns} {\bar{\mu}_{c}}^{n+1})D_{y} c^{n}
\end{pmatrix}.\label{e1e1}
\end{align}
In Eq.(\ref{pm-num-phase}) we let
\begin{align}
\rho^{n}\boldsymbol{u}^{n+1}\cdot \boldsymbol{\nabla}c^n=a_{x} (\rho^{n}_{ew}D_{x}c^{n} u^{n+1})+a_{y} (\rho^{n}_{ns}D_{y}c^{n} v^{n+1}).\label{e1e2}
\end{align}
Note that the special discretization for the advection terms $a_{x}(\rho^{n}_{ew}u^{n+1}D_x c^{n})$ and $a_{y}(\rho^{n}_{ns}v^{n+1}D_y c^{n})$ in (\ref{e1e2}), and the surface tension terms $A_{x}(\rho^{n}_{ew} {\bar{\mu}_{c}}^{n+1})D_{x} c^{n}$  and $A_{y}(\rho^{n}_{ns} {\bar{\mu}_{c}}^{n+1})D_{y} c^{n}$ in (\ref{e1e1}),  are introduced in $\ref{mass-der-rho-dis}$ and $\ref{surface-ten}$ respectively. These discretizations are critical for deriving the fully mass conservation and energy stability of our primitive method. Moreover, $\lVert\boldsymbol{\nabla }_{D}c^{n+\frac{1}{2}}\rVert_{2}=\lVert\boldsymbol{\nabla }_{D}c^{n+1}\rVert_{2}+\lVert\boldsymbol{\nabla }_{D}c^{n}\rVert_{2}$ is the temperal average of the norm $\lVert\boldsymbol{\nabla }_{D}c\rVert_{2}$, which is described in \ref{app-norm}. The following expression for $m(c^{n+1})$ is used for the computations:
\begin{align}
m(c^{n+1})=\sqrt{(c^{n+1})^2(1-{c^{n+1}})^2+\epsilon}.
\end{align}
We assume the cell-centered functions satisfy the following homogeneous Neumann boundary conditions
\begin{align}
\boldsymbol{n}\cdot \boldsymbol{\nabla}_{D}~c^{n+1}\rvert_{\partial \Omega}=\boldsymbol{n}\cdot \boldsymbol{\nabla}_{D}~\mu_{c}^{n+1}\rvert_{\partial \Omega}=0,\label{pm-dis-bc1}
\end{align}
and the velocity $\boldsymbol{u}^{n+1}=(u^{n+1},v^{n+1})$ satisfies the no-slip boundary condition
\begin{align}
u^{n+1}\rvert_{\partial \Omega}=v^{n+1}\rvert_{\partial \Omega}=0.\label{pm-dis-bc2}
\end{align}
A detailed description of the discrete boundary conditions is provided in $\ref{app-bc}$.
\begin{thm}\label{them-mass-conservation}
The fully discrete primitive scheme (\ref{pm-num-mx})-(\ref{pm-num-chem}) is mass conservative for the binary fluid, ${\rm i.e}$, 
\begin{align}
\big(\rho^{n+1}, 1\big)_2&=\big(\rho^{n}, 1\big)_2,~~\hspace{0.5mm}~~\forall n\ge 0.
\end{align}
\end{thm}
\begin{proof}
Multiply Eq.(\ref{pm-num-mass}) by $1/\alpha$ and substituting into Eq.(\ref{pm-num-phase}), we obtain
\begin{align}
&\rho^{n+1}\cfrac{c^{n+1}-c^{n}}{\delta t}+a_{x} (\rho^{n}_{ew}D_{x}c^{n} u^{n+1})+a_{y} (\rho^{n}_{ns}D_{y}c^{n} v^{n+1})= \cfrac{1}{\alpha}d_{x} u^{n+1}+\cfrac{1}{\alpha}d_{y} v^{n+1}.
\end{align}
Multiplying the above equation by $-\delta t~\alpha \rho^{n}$, we obtain the continuity equation at the fully discrete level
\begin{align}
\big((\rho^{n+1}-\rho^{n}),1\big)_2=&-\delta t~\big(a_{x} (u^{n+1} D_{x}\rho^{n} ),1\big)_2-\delta t~\big(a_{y} (v^{n+1}D_{y}\rho^{n} ),1\big)_2\notag\\
& -\delta t~ \big(d_{x} u^{n+1}~\rho^{n},1\big)_2-\delta t~\big( d_{y} v^{n+1}~\rho^{n},1\big)_2,\label{pm-mass-con-dis0}
\end{align}
where we have used the relation (\ref{def-r}) and the following identity 
\begin{align}
&\delta t~\bigg(\big(a_{x} (\rho^{n}_{ew}D_{x}c^{n} u^{n+1})+a_{y} (\rho^{n}_{ns}D_{y}c^{n} v^{n+1})\big), -\alpha\rho^{n}\bigg)_{2}\notag\\
&=\delta t~\big(a_{x} (u^{n+1} D_{x}\rho^{n} ),1\big)_2+\delta t~\big(a_{y} (v^{n+1}D_{y}\rho^{n} ),1\big)_{2}.\label{afdsf-mass}
\end{align}
Note that the definition of $\rho_{ew}$, $\rho_{ns}$ and a detailed derivation of (\ref{afdsf-mass}) are given in $\ref{mass-der-rho-dis}$. Applying summation-by-parts to Eq.(\ref{pm-mass-con-dis0}) and utilizing the homogeneous boundary conditions (\ref{pm-dis-bc2}), we obtain the fully discrete mass conservation for the binary fluid:
\begin{align}
\big((\rho^{n+1}-\rho^{n}),1\big)_2=&\delta t~\big(d_{x}(A_{x}\rho^{n}u^{n+1})+d_{y}(A_{y}\rho^{n}v^{n+1}), 1\big)_2=0.\label{pm-mass-con-dis}
\end{align}
\end{proof}
\begin{thm}\label{pm-en-stable-theo}
The fully discrete primitive scheme (\ref{pm-num-mx})-(\ref{pm-num-chem}) is energy stable at the fully discrete level, ${\rm i.e}$, 
\begin{align}
E^{n+1}_{h}-E^{n}_{h}&=\bigg(\cfrac{1}{2}|| \sqrt{\rho^{n+1}}\boldsymbol{ u}^{n+1}||_{2}^2+\cfrac{\epsilon\eta }{2 We}\rVert \sqrt{\rho^{n+1}}\boldsymbol{\nabla}_{D} c^{n+1}\rVert_2^2 +\cfrac{\eta h^2}{\epsilon We}\big(\rho^{n+1} F(c^{n+1}),1\big)_{2}\notag\\
&+\cfrac{h^2}{Fr\delta t}\big(\rho^{n+1}y, 1\big)_{2}\bigg) \notag\\
&-\bigg(\cfrac{1}{2}|| \sqrt{\rho^{n}} \boldsymbol{u}^{n}||_{2}^2+ \cfrac{\epsilon\eta  }{2 We}\rVert \sqrt{\rho^{n}}\boldsymbol{\nabla}_{D} c^{n}\rVert_2^2)+\cfrac{\eta h^2}{\epsilon We}\big(\rho^{n} F(c^{n}), 1\big)_{2}+\cfrac{h^2}{Fr}\big( \rho^{n} y, 1\big)_{2}\bigg)\notag\\
&=-\cfrac{\delta t}{Re}||\sqrt{\mu(c^{n})}\boldsymbol{\nabla}_{d} \boldsymbol{u}^{n+1}||^2_2-\cfrac{\delta t}{3Re}||\sqrt{\mu(c^{n})}\boldsymbol{\nabla}_{d}\cdot \boldsymbol{u}^{n+1}||^2_2\notag\\
&-\cfrac{\delta t}{Pe}||\sqrt{ m(c^{n})}\boldsymbol{\nabla}_{D} {\mu_{c}}^{n+1}||_{2}^2-\cfrac{1}{2}|| \sqrt{\rho^{n}} (\boldsymbol{u}^{n+1}-\boldsymbol{u}^{n})||_{2}^2\le 0.\label{asdfas-asdfsa-1}
\end{align}
where $E^{n+1}_{h}$ is the total energy of the system at the fully discrete level. Here all the norms are defined by Eqs.(\ref{norm-rho-vel})-(\ref{norm-grad-dot-vel}). Here $\rVert \cdot\rVert_{2}$ is the fully discrete norm that is defined in $\ref{app-norm}$.
\end{thm}
\begin{proof}
Multiplying Eq.(\ref{pm-num-mx}) by $\delta t \boldsymbol{u}^{n+1}=\delta t(u^{n+1},v^{n+1})$ in x and y direction respectively, using summation-by-parts equations and dropping the boundary terms, we obtain
\begin{align}
\cfrac{1}{2} \rVert \sqrt{\rho^{n+1}} \boldsymbol{u}^{n+1} \rVert^{2}_{2}&-\cfrac{1}{2} \rVert \sqrt{\rho^{n}} \boldsymbol{u}^{n} \rVert^{2}_{2}+\cfrac{h^2}{Fr}\big((\rho^{n+1}-\rho^{n}) y ,1\big)_{2}\notag\\
=&-\cfrac{1}{2}\rVert \sqrt{\rho^{n}}(\boldsymbol{u}^{n+1}-\boldsymbol{u}^{n}) \rVert ^{2}_{2}+\cfrac{h^2\delta t}{M}\big(\boldsymbol{\nabla}_{d}\cdot \boldsymbol{u}^{n+1},\bar{p}^{n+1}\big)_{2}\notag\\
&-\cfrac{\delta t}{Re}\rVert \sqrt{\mu(c^{n})}\boldsymbol{\nabla}_{d}\boldsymbol{u}^{n+1}\rVert^{2}_{2}-\cfrac{\delta t}{3Re}\rVert\sqrt{\mu(c^{n})}\boldsymbol{\nabla}_{d}\cdot \boldsymbol{u}^{n+1}\rVert^{2}_{2}\notag\\
&+\cfrac{h^2\delta t}{M}\big(A_{x}(\rho^{n}_{ew}u^{n+1})D_x c^{n},\bar{\mu}^{n+1}_{c}\big)_{2}\notag\\
&+\cfrac{h^2\delta t}{M}\big(A_{y}(\rho^{n}_{ns}v^{n+1})D_y c^{n}, \bar{\mu}^{n+1}_{c}\big)_{2},\label{pm-num-mx-int}
\end{align}
where a special discretization for surface tension terms $A_{x}(\rho^{n}_{ew}{\bar{\mu}_{c}}^{n+1})D_x c^{n}$ and $A_{y}(\rho^{n}_{ns}{\bar{\mu}_{c}}^{n+1})D_y c^{n}$, and the corresponding derivations are introduced in $\ref{surface-ten}$. The various summation-by-parts equations we have used here are Eqs.(\ref{sum-bp-2ew-2c})-(\ref{sum-bp-3yc-3c}). Note that in all the derivations throughout this theorem, the boundary terms that originated from summation-by-parts can be eliminated by utilizing the homogeneous boundary conditions ({\ref{pm-dis-bc1}}) and ({\ref{pm-dis-bc2}}). We have also used the fully discrete mass conservation (\ref{pm-mass-con-dis}) and the following identity
\begin{align}
-\cfrac{h^2}{Fr}\big(A_{y}\rho^{n+1},v^{n+1}\big)_{2}=&-\cfrac{h^2}{Fr}\big(A_{x}\rho^{n+1}u^{n+1},0\big)_{2}-\cfrac{h^2}{Fr}\big(A_{y}\rho^{n+1}v^{n+1},D_{y}y\big)_{2}\notag\\
=&-\cfrac{h^2}{Fr}\big(A_{x}\rho^{n+1}u^{n+1},D_{x}y\big)_{2}-\cfrac{h^2}{Fr}\big(A_{y}\rho^{n+1}v^{n+1},D_{y}y\big)_{2}\notag\\
=&~~~~\cfrac{h^2}{Fr}\bigg(\big(d_{x}(A_{x}\rho^{n+1}u^{n+1})+d_{y}(A_{y}\rho^{n+1}v^{n+1})\big),y\bigg)_{2}\notag\\
=&-\cfrac{h^2}{Fr \delta t}\big((\rho^{n+1}-\rho^{n}) ,y\big)_{2}.\label{potential-discrete}
\end{align}
Multiplying Eq.(\ref{pm-num-mass}) by $\delta t  {\bar{p}}^{n+1}$ and using the summation-by-parts Eqs.(\ref{sum-bp-3xc-3c}) and (\ref{sum-bp-3yc-3c}), we obtain
\begin{align}
h^2\delta t\big(\boldsymbol{\nabla}_{d}\cdot  \boldsymbol{u}^{n+1},\bar{p}^{n+1}\big)_{2} =&-\cfrac{\alpha^2 \delta t}{Pe}\rVert \sqrt{ m(c^{n})} \boldsymbol{\nabla}_{D}\bar{p}\rVert_{2}^{2}\notag\\
&-\cfrac{h^2\alpha \delta t}{Pe} \big( m(c^{n})\boldsymbol{\nabla}_{D} {\bar{\mu}_{c}}^{n+1}, \boldsymbol{\nabla}_{D} {\bar{p}}^{n+1}\big).\label{pm-num-mass-int}
\end{align}
Multiplying Eq.(\ref{pm-num-phase}) by $\delta t{\bar{\mu}_{c}}^{n+1}$ and using the summation-by-parts Eqs.(\ref{sum-bp-3xc-3c}) and (\ref{sum-bp-3yc-3c}), we obtain
\begin{align}
&h^2(\rho^{n+1}(c^{n+1}-c^{n}), {\bar{\mu}_{c}}^{n+1})_{2}+h^2\delta t\big(a_{x} (\rho^{n}_{ew}D_{x}c^{n} u^{n+1})),{\bar{\mu}_{c}}^{n+1}\big)_{2}\notag\\
&+h^2\delta t\big(a_{y} (\rho^{n}_{ns} D_y c^{n}v^{n+1}), {\bar{\mu}_{c}}^{n+1}\big)_{2}\notag\\
&=-\cfrac{h^2\alpha \delta t}{Pe}\big( m(c^{n})\boldsymbol{\nabla}_{D} {\bar{p}}^{n+1}, \boldsymbol{\nabla}_{D} {\bar{\mu}_{c}}^{n+1}\big)-\cfrac{\delta t}{Pe}\rVert \sqrt{m(c^{n})} \boldsymbol{\nabla }_{D} {\bar{\mu}_{c}}^{n+1}\rVert^{2}_{2}.\label{pm-num-phase-int}
\end{align}
Again the special discretization for advection terms $a_{x} (\rho^{n}_{ew}D_{x}c^{n} u^{n+1}))$ and $a_{y} (\rho^{n}_{ns} D_y c^{n}v^{n+1})$, and the corresponding derivations are introduced in $\ref{mass-der-rho-dis}$.
Multiplying Eq.(\ref{pm-num-chem}) by $({c}^{n+1}-{c}^{n})/M$ and using the norm definition (\ref{norm-cgradc}), we obtain
\begin{align}
\cfrac{h^2}{M}\big(\rho^{n+1}(c^{n+1}-c^{n}),  {\bar{\mu}_{c}}^{n+1}\big)_{2}&=\cfrac{\eta h^2 }{\epsilon We}\big((\rho^{n+1} F(c^{n+1}), 1)_{2}-(\rho^{n} F(c^{n}), 1)_{2}\big)\notag\\
&+\cfrac{\epsilon\eta }{2 We}(\rVert \sqrt{\rho^{n+1}}\boldsymbol{\nabla}_{D} c^{n+1}\rVert_2^2 - \rVert \sqrt{\rho^{n}}\boldsymbol{\nabla}_{D} c^{n}\rVert_2^2).\label{pm-num-chem-int}
\end{align}
Summing up the four relations (\ref{pm-num-mx-int}), (\ref{pm-num-mass-int})-(\ref{pm-num-chem-int}), we obtain the energy stability (\ref{asdfas-asdfsa-1}) of the primitive method at the fully discrete level.  Here all the norms are defined by Eqs.(\ref{norm-rho-vel})-(\ref{norm-grad-dot-vel}).
\end{proof}

\subsection{Fully Discrete Projection Method}
The fully-discrete scheme for projection method (\ref{pj-mom})-(\ref{pj-chem}) is the following:
given $c^{n}$, $\bar{\mu}^{n}_c$, $\bar{p}^{n} \in \mathcal {C}_{\overline{m}_{1}\times \overline{m}_{2}}$, $u^{n}\in  \mathcal {E}^{ew}_{m_{1}\times m_{2}}$, $v^{n}\in \mathcal {E}^{ns}_{m_{1}\times m_{2}}$, and $\boldsymbol{u}^{n}=(u^{n},v^{n})$ at time $t=n\delta t$, find grid functions $c^{n+1},~\bar{\mu}^{n+1}_c,~\bar{p}^{n+1}\in  \mathcal {C}_{\overline{m}_{1}\times \overline{m}_{2}}$, $u^{n+1}$, $\tilde{u}^{n+1}\in  \mathcal {E}^{ew}_{m_{1}\times m_{2}}$, $v^{n+1}$, $\tilde{v}^{n+1}\in \mathcal {E}^{ns}_{m_{1}\times m_{2}}$, and $\boldsymbol{u}^{n+1}=(u^{n+1},v^{n+1})$ at time $t=(n+1)\delta t$:
\begin{align}
&\boldsymbol{A}\rho^{n} \cfrac{\tilde{\boldsymbol{u}}^{n+1}-\boldsymbol{u}^{n}}{\delta t}+\rho^{n} \boldsymbol{u}^{n}\boldsymbol{\nabla}\tilde{\boldsymbol{u}}^{n+1}+\cfrac{1}{2}\tilde{\boldsymbol{u}}^{n+1}\big(\cfrac{\boldsymbol{A}\rho^{n+1}-\boldsymbol{A}\rho^{n}}{\delta t}+\boldsymbol{\nabla}\cdot (\rho^{n}\boldsymbol{u}^{n})\big)\notag\\
&=\cfrac{1}{Re}\boldsymbol{\nabla}_{\hspace{-1mm}(D,\mathfrak{D})}\cdot (\mu(c^{n}) \boldsymbol{\nabla}_{\hspace{-1mm}(d,\mathcal{D})} \tilde{\boldsymbol{u}}^{n+1})+\cfrac{1}{3Re}\boldsymbol{\nabla}_{\hspace{-1mm}D}\big(\mu(c^{n})\boldsymbol{\nabla}_{\hspace{-1mm}d}\cdot  \tilde{\boldsymbol{u}}^{n+1} \big)-\cfrac{1}{Fr}\boldsymbol{A}\rho^{n+1} g \boldsymbol{j},\label{pj-num-mx}\\
&\boldsymbol{A}\rho^{n+1}\cfrac{\boldsymbol{u}^{n+1}-\tilde{\boldsymbol{u}}^{n+1}}{\delta t}=-\cfrac{1}{M}\boldsymbol{\nabla}_{D}\bar{p}^{n+1} +\cfrac{1}{M}\rho^{n+1}{\bar{\mu}_{c}}^{n+1}\boldsymbol{\nabla} c^{n+1},\label{pj-num-pjx}\\
&\boldsymbol{\nabla}_{\hspace{-1mm}d}\cdot \boldsymbol{u}^{n+1}=\cfrac{\alpha}{Pe} \boldsymbol{\nabla}_{\hspace{-1mm}d} \cdot   \big(\boldsymbol{ A}m(c^{n}) \boldsymbol{\nabla}_{\hspace{-1mm}D} {\bar{\mu}_{c}}^{n+1}\big)+\cfrac{\alpha^2}{Pe} \boldsymbol{\nabla}_{\hspace{-1mm}d} \cdot   \big(\boldsymbol{ A}m(c^{n}) \boldsymbol{\nabla}_{\hspace{-1mm}D} {\bar{p}}^{n+1}\big),\label{pj-num-mass}\\
&\rho^{n}\cfrac{c^{n+1}-c^{n}}{\delta t}+\rho^{n+1}\boldsymbol{u}^{n+1}\cdot \boldsymbol{\nabla}c^{n+1}=\cfrac{1}{Pe} \boldsymbol{\nabla}_{\hspace{-1mm}d} \cdot   \big(\boldsymbol{ A}m(c^{n}) \boldsymbol{\nabla}_{\hspace{-1mm}D} {\bar{\mu}_{c}}^{n+1}\big)\notag\\
&+\cfrac{\alpha}{Pe} \boldsymbol{\nabla}_{\hspace{-1mm}d} \cdot   \big(\boldsymbol{ A}m(c^{n}) \boldsymbol{\nabla}_{\hspace{-1mm}D} {\bar{p}}^{n+1}\big),\label{pj-num-phase}\\
&\rho^{n}{\bar{\mu}_{c}}^{n+1}=\cfrac{M\eta}{\epsilon We}\rho^{n+\frac{1}{2}}g(c^{n+1},c^{n})+\cfrac{M\eta}{\epsilon We} F^{n+\frac{1}{2}}(c)~r(c^{n+1},c^{n})\notag\\
&+\cfrac{\epsilon\eta M}{2 We} \lVert\boldsymbol{\nabla }_{D}c\rVert_{2}^{n+\frac{1}{2}} r(c^{n+1},c^{n})-\cfrac{\epsilon\eta M}{ We}\boldsymbol{\nabla}_{\hspace{-1mm}d}\cdot  (\boldsymbol{A}\rho^{n+\frac{1}{2}}\boldsymbol{\nabla}_{\hspace{-1mm}D}~c^{n+\frac{1}{2}}).\label{pj-num-chem}
\end{align}
Note that the terms $\rho^{n}\boldsymbol{u}^{n}\cdot \boldsymbol{\nabla}\tilde{\boldsymbol{u}}^{n+1}$ and $\boldsymbol{\nabla}\cdot (\rho^{n} \boldsymbol{u}^{n})\boldsymbol{u}^{n+1}/2$ in Eq.(\ref{pj-num-mx}), the term $\rho^{n+1} {\bar{\mu}_{c}}^{n+1}\boldsymbol{\nabla}c^{n+1}$ in Eq.(\ref{pj-num-pjx}), and the term $\rho^{n+1}\boldsymbol{u}^{n+1}\cdot \boldsymbol{\nabla}c^{n+1}$ in Eq.(\ref{pj-num-phase}) are defined analogously as in Eqs.(\ref{e1e1-1})-(\ref{e1e2}), where the main difference is that in the current method, the above three terms have different upper subscript representing the solution at the different time step. Again, $\lVert\boldsymbol{\nabla }_{D}c\rVert_{2}^{n+\frac{1}{2}}=\lVert\boldsymbol{\nabla }_{D}c^{n+1}\rVert_{2}+\lVert\boldsymbol{\nabla }_{D}c^{n}\rVert_{2}$ is the temperal average. Once again, the reader is referred to the Appendix and \cite{SteveMultigrid-2009,SteveMultigrid-2010,SteveMultigrid2013} for a description of the finite difference notation used here. We assume the cell-centered functions satisfy the following Neumann boundary conditions
\begin{align}
\boldsymbol{n}\cdot \boldsymbol{\nabla}_{D}c^{n+1}\rvert_{\partial \Omega}=\boldsymbol{n}\cdot \boldsymbol{\nabla}_{D}\mu_{c}^{n+1}\rvert_{\partial \Omega}=0,\label{pj-dis-bc1}
\end{align}
and the intermediate velocity $\tilde{\boldsymbol{u}}^{n+1}=(\tilde{u}^{n+1},\tilde{v}^{n+1})$ satisfies the no-slip boundary condition
\begin{align}
\tilde{u}^{n+1}\rvert_{\partial\Omega}=\tilde{v}^{n+1}\rvert_{\partial\Omega}=0,\label{pj-dis-bc2}
\end{align}
and the velocity $\boldsymbol{u}^{n+1}=(u^{n+1},v^{n+1})$ satisfies the following boundary condition 
\begin{align}
\boldsymbol{n}\cdot \boldsymbol{u}^{n+1}|_{\partial\Omega}=0.\label{pj-dis-bc3}
\end{align}
A detailed description for the boundary condition is provided in $\ref{app-bc}$.
\begin{thm}
The scheme (\ref{pj-num-mx})-(\ref{pj-num-chem}) is mass conservative for the two-phase fluid, ${\rm i.e}$, 
\begin{align}
\big(\rho^{n+1},  1\big)&=\big(\rho^{n},  1\big),~~\hspace{0.5mm}~~\forall n\ge 0.
\end{align}
\end{thm}
\begin{proof}
Multiplying Eq.(\ref{pj-num-mass}) by $1/\alpha$ and substituting into (\ref{pj-num-phase}), we obtain
\begin{align}
&\rho^{n}\cfrac{c^{n+1}-c^{n}}{\delta t}+a_{x} (\rho^{n+1}_{ew}D_{x}c^{n+1} u^{n+1})+a_{y} (\rho^{n+1}_{ns}D_{y}c^{n+1} v^{n+1})\notag\\
&= \cfrac{1}{\alpha}d_{x} u^{n+1}+\cfrac{1}{\alpha}d_{y} v^{n+1}.
\end{align}
Again the definition of $\rho_{ew}$, $\rho_{ns}$ and a detailed derivation of (\ref{afdsf-mass}) are given in $\ref{mass-der-rho-dis}$. Multiplying the above equation by $-\alpha \rho^{n}$ and using the same treatment in Eqs.(\ref{afdsf-mass}) and (\ref{pm-mass-con-dis}), we obtain the mass conservation of the binary fluids from our projection method:
\begin{align}
\cfrac{1}{\delta t}\big((\rho^{n+1}-\rho^{n}),1\big)_{2}=&\big(d_{x}(A_{x}\rho^{n+1}u^{n+1})_{2}+d_{y}(A_{y}\rho^{n+1}v^{n+1}) ,1\big)_{2}.\label{pj-dis-mass-con}
\end{align}
Using summation by parts, we obtain
\begin{align}
\big((\rho^{n+1}-\rho^{n}), 1\big)_{2}&=0,
\end{align}
where we have shown that our projection method preserves the mass of the binary fluid at the fully discrete level. \end{proof}
\begin{thm}\label{pj-en-stable-theo}
The fully discrete projection scheme (\ref{pj-num-mx})-(\ref{pj-num-chem}) is energy stable, ${\rm i.e}$, 
\begin{align}
&E^{n+1}_{h}-E^{n}_{h}=\notag\\
&\bigg(\cfrac{1}{2}|| \sqrt{\rho^{n+1}}\boldsymbol{ u}^{n+1}||_{2}^2+\cfrac{\epsilon\eta }{2 We}\rVert \sqrt{\rho^{n+1}}\boldsymbol{\nabla}_{D} c^{n+1}\rVert_2^2+\cfrac{\eta h^2}{\epsilon We}\big(\rho^{n+1} F(c^{n+1}), 1\big)_{2}+\cfrac{h^2}{Fr}\big(\rho^{n+1} y , 1\big)_{2}\bigg)\notag\\
&-\bigg(\cfrac{1}{2}|| \sqrt{\rho^{n}} \boldsymbol{u}^{n}||_{2}^2+\cfrac{\epsilon\eta }{2 We} \rVert \sqrt{\rho^{n}}\boldsymbol{\nabla}_{D} c^{n}\rVert_2^2)+\cfrac{\eta h^2}{\epsilon We}\big(\rho^{n} F(c^{n}),  1\big)_{2}+\cfrac{h^2}{Fr}\big(\rho^{n} y , 1\big)_{2}\bigg)\notag\\
&=-\cfrac{\delta t}{Re}||\sqrt{\mu(c^{n})}\boldsymbol{\nabla}_{d} \tilde{\boldsymbol{u}}^{n+1}||^2_2-\cfrac{\delta t}{3Re}||\sqrt{\mu(c^{n})}\boldsymbol{\nabla}_{d}\cdot\tilde{ \boldsymbol{u}}^{n+1}||^2_2-\cfrac{\delta t}{Pe}||\sqrt{ m(c^{n})}\boldsymbol{\nabla}_{D} {\mu_{c}}^{n+1}||_{2}^2\notag\\
&-\cfrac{1}{2}|| \sqrt{\rho^{n}} (\tilde{\boldsymbol{u}}^{n+1}-\boldsymbol{u}^{n})||_{2}^2-\cfrac{1}{2}|| \sqrt{\rho^{n+1}} (\boldsymbol{u}^{n+1}-\tilde{\boldsymbol{u}}^{n+1})||_{2}^2\le 0.\label{pj-sdafa}
\end{align}
where $E^{n+1}_{h}$ is the total energy of the system at the fully discrete level.
\end{thm}
\begin{rmk}
Here we omit the details of the proof, as the derivations here are similar with the proof for the primitive method in many aspects. The primary differences is that in the projection method has one more projection equation (\ref{pj-num-pjx}), and the test functions are different. In particular, to show the energy stability of the projection method, we multiply Eq.(\ref{pj-num-mx}) by $\delta t\tilde{\boldsymbol{u}}^{n+1}$, Eq.(\ref{pj-num-pjx}) by $\delta t\boldsymbol{u}^{n+1}$, Eq.(\ref{pj-num-mass}) by $\delta t\bar{p}^{n+1}/ M$, Eq.(\ref{pj-num-phase}) by ${\delta t\bar{\mu}_{c}}^{n+1}/M$ and Eq.(\ref{pj-num-chem}) by $(c^{n+1}-c^{n})/M$. After using the summation-by-parts with the homogeneous boundary conditions, we sum up the resulted relations to obtain the energy stability for the projection method (\ref{pj-sdafa}). 
\end{rmk}
\section{Multigrid Solver}
In this paper, we present an efficient nonlinear FAS multigrid solver for our schemes. The solver is motivated by that described for the Cahn-Hilliard-Brinkman scheme in an existing paper \cite{SteveMultigrid2013}, where a finite difference method in primitive variable formulations is used. The primary difference is that in the present paper, the multigrid solver for our primitive method is designed for a much more complicated and highly non-linear problem comprised of a full Navier-Stokes equation and Cahn-Hilliard equation with variable density. 
\subsection{Primitive Methods}
In the multigrid solver for our the primitive method (\ref{pm-num-mx})-(\ref{pm-num-chem}), the smoothing operators for the Cahn-Hilliard equation and Navier-Stoke equation are decoupled. Specifically, for each grid cell $(i, j)$, we perform the following steps:\\
\begin{enumerate}[Step 1.]
\item Update $c^{k+1}_{i,j}$, $\bar{\mu}^{k+1}_{c~i,j}$ using a non-linear Gauss-Seidel method on the CH equations (\ref{pm-num-phase}) and (\ref{pm-num-chem}).
\item Update the five variables $u^{k+1}_{i\pm\frac{1}{2},j}$, $v^{k+1}_{i,j\pm\frac{1}{2}}$, and $\bar{p}^{k+1}_{i,j}$ using a Vanka-type smoothing strategy \cite{Oosterlee2008, Multigrid, Vanka1986, Wesseling2004} on the NS equations (\ref{pm-num-mx})-(\ref{pm-num-mass}) with the updated values for $c^{k+1}_{i,j}$, $\bar{\mu}^{k+1}_{c~i,j}$.
\end{enumerate}
Here $k$ stands for the iteration step at the current time step. Note that $u$ and $v$ are edge-centered variables and this contributes to the complication of the method. Here we omit the details for the relaxation, and we refer to the \cite{SteveMultigrid-2010} as a description for the Vanka-type smoother for the fluid equation in primitive variable formulation. Note that the smoother operator for CH equation and the Vanka-type smoother for NS equation is performed in the RedBlack order.
\subsection{Projection Methods}
For the projection method (\ref{pj-num-mx})-(\ref{pj-num-chem}), we first perform a relaxation on the Cahn-Hilliard equations (\ref{pj-num-phase}) and (\ref{pj-num-chem}), which is the same as the primitive method; and then relax the flow equation (\ref{pj-num-mx}) to obtain the intermediate velocity by using a Vanka-type smoother. Note that the smoother used here differs from that of the primitive method in that in the present case the pressure is not updated together with the four intermediate velocity variables. We next relax the mass conservation equation (\ref{pj-num-mass}) to obtain the pressure, and finally we update the velocity through the projection equation (\ref{pj-num-pjx}). \\
Specifically, in the proposed smoother, for each grid cell $(i, j)$, we perform the following steps:
\begin{enumerate}[Step 1.]
\item Update $c^{k+1}_{i,j}$, $\bar{\mu}^{k+1}_{c~i,j}$ by using a nonlinear Gauss-Seidel method on CH equations (\ref{pj-num-phase}) and (\ref{pj-num-chem}).
\item Update the four intermediate velocity variables $\tilde{u}^{k+1}_{i\pm\frac{1}{2},j}$, $\tilde{v}^{k+1}_{i,j\pm\frac{1}{2}}$ using a Vanka-type smoothing strategy on the fluid equation (\ref{pj-num-mx}), with the updated $c^{k+1}_{i,j}$, $\bar{\mu}^{k+1}_{c~i,j}$.
\item Update the pressure $\bar{p}^{k+1}_{i,j}$ by using a nonlinear Gauss-Seidel method on the the mass conservation equation (\ref{pj-num-mass}) with the updated $c^{k+1}_{i,j}$, $\bar{\mu}^{k+1}_{c~i,j}$ and $\tilde{u}^{k+1}_{i\pm\frac{1}{2},j}$, $\tilde{v}^{k+1}_{i,j\pm\frac{1}{2}}$.
\item Update the four velocity variables $u^{k+1}_{i\pm\frac{1}{2},j}$ and $v^{k+1}_{i,j\pm\frac{1}{2}}$ through the projection equation (\ref{pj-num-pjx}) with the updated $c^{k+1}_{i,j}$, $\bar{\mu}^{k+1}_{c~i,j}$, $\tilde{u}^{k+1}_{i\pm\frac{1}{2},j}$, $\tilde{v}^{k+1}_{i,j\pm\frac{1}{2}}$ and $\bar{p}^{k+1}_{i,j}$.
\end{enumerate}
Here we omit the full details of the implementation and smoothing strategy and refer the reader to \cite{SteveMultigrid2013} for the remaining details of the solver. Moreover, for both methods, we use a standard FAS V-cycle approach that can be found in \cite{SteveMultigrid2013}. 
\section{Numerical example}\label{sim-sec}
In this section we investigate the performance of our numerical schemes by solving several test problems. For the advantage of using the finite difference method, we will focus only on rectangular domains. Due to the page limit, only two figures are shown in the \ref{app-B} to illustrate the mass conservation of our methods for Case 1 in Example 1.
\subsection{Capillary Wave}\label{sec-CW}
The first test is the damping of a sinusoidal, capillary wave, which takes into account the surface tension, gravity, and two phase flows with variable density. In \cite{Prosperetti1981}, an analytical solution was found in the case of the small-amplitude waves on an interface between incompressible viscous fluids in an infinite domain. To simulate this problem, we choose a computational domain, $\Omega = \{(x,y) : 0\leqslant x\leqslant 1, 0\leqslant y\leqslant 1\}$. We assume that the equilibrium position of the interface coincides with $x$ axis, and the capillary wave-length equals to the length of the domain in $x$-direction. We further assume zero initial velocity, and that the initial profile of the interface given by
\begin{align}
c(y,0) = \cfrac{1}{2}\bigg(1-{\rm tanh}\big(\cfrac{y-\tilde{y}}{2\sqrt{2 \epsilon}}  \big)\bigg)
\end{align}
with the perturbation $\tilde{y}(x) =0.5- H_{0}~{\rm{cos}} kx$, $k={2 \pi}/{\lambda_{w}}= 2 \pi$ and the initial amplitude of the perturbation wave $H_{0} = 0.01$. We set the gravity $g=1$ and surface tension $\sigma=1$. Moreover, the ratio parameter $\eta$ that relates the sharp interface model and phase field model is determined through the following equation \cite{Guo2014JFM}:
\begin{align}
\eta=\cfrac{(\rho_{2}-\rho_{1})^3}{2\sqrt{2}\rho_{1}\rho_{2}(\rho_{2}^2-\rho_{1}^{2}-2\rho_{1}\rho_{2}{\rm ln}\frac{\rho_{2}}{\rho_{1}})}.
\end{align}
To test our schemes, two cases with different density and viscosity ratios are considered. In Case 1 and 2, we choose the following values for kinematic viscosities and densities for the two fluids respectively:
\begin{align}
\nu&=\cfrac{\mu_{1}}{\rho_{1}}=\cfrac{\mu_{2}}{\rho_{2}}=0.01,~\rho_{1}=1,~\rho_{2}=10,\notag\\
\nu&=\cfrac{\mu_{1}}{\rho_{1}}=\cfrac{\mu_{2}}{\rho_{2}}=0.01,~\rho_{1}=1,~\rho_{2}=1000.
\end{align}
The other non-dimensional parameters are set as
\begin{align}
Re=100,~We=1,~Fr=1,~M=\epsilon,~Pe=1/\epsilon,
\end{align}
which is corresponding to the asymptotic analysis of the q-NSCH model \cite{Lowengrub1998}. Periodic conditions are imposed on the left and right boundary for the velocity $\boldsymbol{u}$, phase-field function $c$ and chemical potential $\mu_{c}$. At the upper and lower boundaries, we impose the no-slip boundary condition for the velocity, and no-flux boundary conditions for the phase-field functions $c$, $\mu_{c}$. The time step is set as $\delta t = 10^{-3}$. For each case, we use two values of $\epsilon = 0.005,~{\rm and}~0.0025$ with the corresponding grid size $[256 \times 256]$ and $[512 \times 512]$ respectively. Due to the sharp interface analysis \cite{Guo2014JFM}, the numerical results of this phase-field model approaches to that of the sharp interface model as the value of $\epsilon$ decreases. Both schemes are computed and the numerical results are compared with the analytical solution. Figure \ref{CAP-Den1to10-intf} shows the capillary wave amplitude for Case 1 with density ratio $1:10$, where for both schemes, the numerical results all agree well with the analytical solution. Moreover, as $\epsilon$ decreases, the numerical results converge to the analytical solutions. In \ref{app-B}, Figures \ref{CAP-Den1to10-SMass} and \ref{CAP-Den1to10-TMass} show the time evolution of the mass of single component $\rho c$ and binary fluids $\rho$, where it can be observed that both methods preserve the mass well. In particular, the mass can be conserved up to $10^{-10}$ with the primitive method, which performs slightly better than the projection method that preserves the mass up to $10^{-9}$. The similar results can be observed in Figure \ref{CAP-Den1to1000-intf} for the Case 2 with density ratio $1:1000$. Moreover, it has been confirmed that the mass of the single component and the binary fluids are preserved up to $10^{-9}$ by primitive method and up to $10^{-8}$ by the projection method for Case 2. In Figure \ref{CAP-Den1to10-En}, we show the energy dissipation for both cases with both methods. As predicted by the Theorem \ref{pm-en-stable-theo} and \ref{pj-en-stable-theo}, the energy decreases for both methods, exhibiting a similar dissipation way.
\subsection{Rising Droplets}\label{sec-RS}
As a second test, we simulate the dynamics of rising droplets. The test setup is taken from in \cite{Mark-JFM-1997}. In particular, the computational domain $\Omega=[0,1]\times [0,2]$ is filled with the heavier fluid ($c=0$) and a initially circular shaped lighter fluid ($c=1$) is placed inside. The initial drop has a radius of 0.25 and is centered at $[0.5,0.5]$. This leads to the initial profile of the interface given by
\begin{align}
c(r,0) = \cfrac{1}{2}\bigg(1-{\rm tanh}\big(\cfrac{r-R_{0}}{2\sqrt{2 }\epsilon}  \big)\bigg)
\end{align}
with $r=\sqrt{(x-0.5)^2+(y-0.5)^2}$ and $R_{0}=0.25$. The parameters of the outer fluid are $\rho_{2}=1000$, $\mu_{2}=10$, and $\rho_{1}=100$, $\mu_{1}=1$ for the drop fluid. The gravity is $\boldsymbol{g}=(0, 0.98)$, and the surface tension is $\sigma=24.5$, which lead to the values of following non-dimensional parameters
\begin{align}
Re=100,~We=1,~Fr=0.98,~M=\epsilon,~Pe=1/\epsilon.
\end{align}
As in the previous example, periodic conditions are imposed on the left and right boundary for the velocity $\boldsymbol{u}$, pressure $\bar{p}$, phase-field function $c$ and chemical potential $\mu_c$, and the no-slip boundary condition for the velocity, and no-flux boundary conditions for the phase-field functions $c$, $\mu_{c}$ are imposed at the upper and lower boundaries. The time step is $\delta t = 10^{-3}$ for the two methods. Moreover, to test the convergence of the diffuse-interface, we use two values of $\epsilon = 0.005~{\rm and}~0.0025$, which corresponds to the grid size $[256 \times 256]$ and $[512 \times 512]$ respectively.\\
Because the droplet is lighter than the surrounding fluid, the droplet rises. For a rigorous estimate of the accuracy of the simulation, we calculated the rising velocity that is determined by:
\begin{align}
V_{c}=\cfrac{\int_{\Omega_{}}v c~{\rm d}\boldsymbol{x}}{\int_{\Omega}c~{ \rm d} \boldsymbol{x}}
\end{align}
where $v$ is the second (vertical) component of the velocity $\boldsymbol{u}$. Moreover, to show the quasi-incompressibility of the q-NSCH model, we calculate the divergence of velocity $\boldsymbol{\nabla}_{d} \cdot \boldsymbol{u}$ at the fully discrete level.\\
Snapshots of the deformed droplet interfaces and the $\boldsymbol{\nabla} \cdot \boldsymbol{u}$ (quasi-incompressibility) are presented in Figure \ref{RS-Den1to10-quasi}, where we observe that the drop deforms slowly, resulting in a mushroom shape. Recall that the divergence-free condition does not hold for quasi-incompressible fluids with different densities because the fluids may mix slightly across the interface. The two incompressible fluids can be compressible across the interface where the two components are mixed. It can be observed that the fluid is incompressible ($\boldsymbol{\nabla} \cdot \boldsymbol{u}=0$) almost everywhere except along the moving interface. Near the interface, waves of expansion ($\boldsymbol{\nabla} \cdot \boldsymbol{u} >0$ ) and compression ($\boldsymbol{\nabla} \cdot \boldsymbol{u}<0$ ) are observed. Figure \ref{RS-Den1to10-Intf} shows the droplet shapes at the final time ($t=3$), where we observe that the droplet shapes differ clearly for different values of $\epsilon$ but seem to converge so that there is no big difference for the finest values $\epsilon=0.0025$ and the result obtained from the \cite{Mark-JFM-1997} by using a sharp interface model. Figure \ref{RS-Den1to10-Vel} plots the rising velocity of our numerical relusts and the result obtained from \cite{Mark-JFM-1997}, where the agreement improves as $\epsilon\rightarrow 0$. In Figure \ref{RS-Den1to10-En}, we show the energy dissipation of the binary fluid system obtained from both methods by using different values of $\epsilon$. Note that, to show the energy dissipation converges as $\epsilon$ is decreased, we also compute the example by using a even smaller $\epsilon=0.00125$. As expected, the energy decreases and yields very similar way for both methods. it can be observed that the  It has been confirmed that the mass of the single component and the binary fluids are preserved up to $10^{-10}$ by primitive method and up to $10^{-9}$ by the projection method. 
\subsection{Rayleigh-Taylor Instability}\label{sec-RT}
Our last test is the Rayleigh-Taylor instability which would occur for any perturbation along the interface between a heavy fluid ($c=0$) on top of a light fluid ($c=1$), and is characterised by the density difference between the two fluids. The instability is characterized by the non-dimensional parameter Atwood ratio, that $At = (\rho_{A}-\rho_{B})/(\rho_{A}+\rho_{B})$. The initial growth and long-time evolution of Rayleigh-Taylor instability was investigated by Tryggvason \cite{Tryggvason1988} for inviscid incompressible flows with zero surface tension, at $At = 0.50$. Guermond et al \cite{Guermond1967} studied this instability at the same value of At but accounted for viscous effects. Ding \cite{Ding2007} studied this instability problem by using the a different phase-field model where the velocity satisfies the divergence free constraint. We validate our code here by investigating the same problem as Guermond et al \cite{Guermond1967}, i.e., at $At = 0.50$ and $Re (=\rho_{l} L^{3/2}g^{1/2}/\mu_{l}) = 3000$, with the initial interface being located in a rectangular domain $[0, L] \times [0, 4L]$ at $\tilde{y}(x) = 2L + 0.1 L \rm{cos} (2\pi x/L)$, which represents a planar interface superimposed by a perturbation of wave number $k = 1$ and amplitude $0.1L$. Here we set $L=1$, and we take
\begin{align}
c(y,0) = \cfrac{1}{2}\bigg(1-{\rm tanh}\big(\cfrac{y-\tilde{y}}{2\sqrt{2 \epsilon}}  \big)\bigg)
\end{align}
Here we set the gravitational acceleration $g=1$ and surface tension $\sigma=0$. In the present case of zero surface tension, the Cahn-Hilliard equation simply amounts to interface tracking only. The non-dimensional parameters are set as
\begin{align}
Re=3000,~We=200,~Fr=1,~M=\epsilon,~Pe=1/\epsilon.
\end{align}
Periodic conditions are imposed on the left and right boundary for $\boldsymbol{u}$, $c$ and $\mu_{c}$. At the upper and lower boundaries, we impose the no-slip boundary condition for $\boldsymbol{u}$, and no-flux boundary conditions for $c$ and $\mu_{c}$. We use two values of $\epsilon = 0.005,~{\rm and}~0.0025$ which corresponds to the grid size $[128 \times 512]$ and $[256 \times 1024]$. We set time-step $\delta t=10^{-3}$. Results are presented in Figure \ref{RT-position} in terms of the y-coordinate of the top of the rising fluid and the bottom of the falling fluid, together with the corresponding previous results of Tryggvason \cite{Tryggvason1988}, Guermond et al \cite{Guermond1967} (sharp interface models) and Ding \cite{Ding2007} (another phase-field model). For both methods, good agreement is observed with these results. As the value of $\epsilon$ decreases, our numerical resutls converge to the other numerical results. The evolution of the interface of our numerical results (projection with $\epsilon=0.0025$) and the results in \cite{Ding2007} are shown in Figure \ref{RT-2}, in which the rolling-up of the falling fluid can be clearly seen. At the early time, two counter-rotating vortices are formed along the sides of the falling filament and grow with time. To a certain extent, the two vortices are shed and a pair of secondary vortices occurs at the tails of the roll-ups. Our results agree with those obtained in \cite{Ding2007}. However, comparing to their results, the small structures around the vortices are preserved and can be observed more clearly due to the mass conservative property of our numerical methods. The time evolution of the energy is shown in Figure \ref{RT-En}, which decreases as expected. It has been confirmed that the mass of the single component and the binary fluids are preserved up to $10^{-12}$ by primitive method and up to $10^{-10}$ by the projection method. 
\section{Convergence test}\label{accuracy-sec}
To show our two methods are both first order accurate in time and second order accurate in space, we carry out a convergence test by considering the Cauchy sequence. We compute our q-NSCH system with the following function as the initial condition for the phase variable $c$, 
	\begin{align}
c(x,y)=\cos( 2\pi x )+\cos( 2\pi y ),
	\label{eqn-3d-exact-solution}
	\end{align}
and zero for all the other variables, including $\mu$, $p$ and $\boldsymbol{u}$.	 The computational domain is $[0,1]\times [0,1]$, and the homogeneous Neumann boundary conditions for $c$ $\mu$ and $p$, and $\boldsymbol{u}$ are applied on the boundary. We refine the mesh and time step according to the schedule
	\begin{align*}
m=16 \qquad &{\rm and} \qquad \Delta t= 1/16,
	\\
m=32 \qquad &\mbox{and} \qquad \Delta t = 1/64,
	\\
m=64 \qquad &\mbox{and} \qquad \Delta t = 1/256,
	\\
m=128  \qquad &\mbox{and} \qquad \Delta t = 1/1024.
	\\
m=256  \qquad &\mbox{and} \qquad \Delta t = 1/4096.
	\end{align*}
Here $m$ is the grid points in both $x$ direction and $y$ direction, $\Delta t$ is the time step. To compare solutions on different grid resultions, we push the solution at the coarse grid up to the next fine grid to calculate the difference in $L_2$ norm, and then obtain the convergence rate which are shown in Table \ref{tab01}. The second order convergence rate are achieved for both methods, which indicates that both of our methods are second order accurate in space and first order accurate in time. Moreover, we also present the average computational cost for each time step of both methods with different grid resolutions in Table \ref{tab02}, which indicates that the Projection method is much more efficient than the Primitive method. All the tests are carried out on a desktop with 4.0GHz AMD(R) FX(TM)-8350 processor.
	\begin{table}
	\begin{center}
	\begin{tabular}{|c||c|c||c|c|}
\hline	
&\multicolumn{2}{c||}{Primitive}&\multicolumn{2}{c|}{Projection}
	\\
	\hline	
$m$ & error&rate& error&rate
	\\
	\hline
16~{\&}~32 &6.063$\times 10^{-3}$ & ---&6.409$\times 10^{-3}$ &---
	\\
	\hline

32~{\&}~64&2.310$\times 10^{-3}$&1.40&2.603$\times 10^{-3}$ &1.30 	\\
	\hline

64~{\&}~128&6.812$\times 10^{-4}$&1.76&7.814$\times 10^{-4}$ &1.73
	\\
	\hline
128~{\&}~256& 1.622$\times 10^{-4}$&2.06&1.917$\times 10^{-4}$ &2.03
	\\
	\hline
	\end{tabular}
	\end{center}

\caption{Convergence test for the primitive method and projection method. }
	\label{tab01}
	\end{table} 

	\begin{table}
	\begin{center}
	\begin{tabular}{|c||c|c||}
\hline	
&\multicolumn{2}{c||}{Ave. CPU time per time step (sec)} 
	\\
	\hline	
&{~~~~~Primitive~~~~~}&{Projection}
	\\
	\hline
16 &2.913$\times 10^{-2}$ & 1.075$\times 10^{-2}$
	\\
	\hline
32 &1.659$\times 10^{-1}$ & 3.816$\times 10^{-2}$
	\\
	\hline
64&2.659$\times 10^{-1}$&1.149$\times 10^{-1}$ 
\\
	\hline
128&1.027$\times 10^{0}$&4.170$\times 10^{-1}$ 
	\\
	\hline
256& 2.622$\times 10^{0}$&1.246$\times 10^{0}$ 
	\\
	\hline
	\end{tabular}
	\end{center}

\caption{Average computational cost for the Primitive method and Projection method.}
	\label{tab02}
	\end{table}

\section{Conclusion}
In this paper we presented and analyzed two fully mass conservative and energy stabel finite difference schemes for the q-NSCH sytem governing the binary incompressible fluid flows with variable density and viscosity. At the continuous level, we reformulated the system equations and show that the system conserves both the component and binary fluid mass and the energy is non-increasing due to the energy dissipation law that underlies the model. Based on the reformulated system, we introduced two time-discrete and fully discrete numerical methods using primitive variable and projection-type formulation. Both schemes are fully mass conserving and the extra mass correction is not necessary. Moreover, the fully discrete energy stability are achieved for both methods, where the enregy functionals are always non-increasing. In particular, our projection method differs from the traditional projection methods in two ways: \\
(a) due to the quasi-incompressibility of the q-NSCH model, the pressure here is used to correct the intermediate velocity to get the updated velocity that satisfies the quasi-incompressible constraint, whereas in most of the existing projection methods the intermediate velocity is projected onto a space of divergence-free velocity field due to the corresponding model constraints (See \cite{JieSIAM2010} as example);\\
(b) in our projection method, a pressure-Poisson equation (continuity equation (\ref{pj-mass})) naturally occurs in the reformulated system equations, whereas, in most of the projection methods, the extra construction of the pressure-Poisson equation is required (this is usually done by applying the divergence operator to the decoupled equation (\ref{pj-pj}));
We also present an efficient nonlinear FAS multigrid method for each method, which is motivated by a Vanka-type smoothing strategy for the Cahn-Hilliard-Brinkman equation \cite{SteveMultigrid2013}.\\
Three numerical examples are investigated numerically, including the Capillary Wave, Rayleigh-Taylor instability and rising droplets. Quantitative comparisons are made with the existing analytical solution or the previous numerical results to validate the accuracy of our numerical schemes. Moreover it has been confirmed that the mass of the single fluid and the binary fluids are preserved up to $10^{-8}$, and the energy is always non-increasing for all the examples. The convergence property of the q-NSCH is investigated as well. In particular, our numerical results converge to the analytical or numerical solutions of the other sharp interface models as the thickness of the diffuse-interface decreases. Moreover, in the rising droplet example, we show that the quasi-incompressibility ($\boldsymbol{\nabla} \cdot \boldsymbol{u}\ne 0$) is captured smoothly together with the move interface which indicates that the quasi-incompressibility does not give any problems to our schemes.
\section*{Acknowledgement}
ZG gratefully acknowledge partial support from the $150^{\rm th}$ Anniversary Postdoctoral Mobility Grant of London Mathematical Society (PMG14-15 09). JL and ZG gratefully acknowledge partial support from National Science Foundation Grants NSF-DMS-1719960, NSF-DMS-1522775 and the National Institute of Health grant P50GM76516 for the Center of Excellence in Systems Biology at the University of California, Irvine. PL is partially supported by the National Natural Science Foundation of China (No. 91430106) and the Fundamental Research Funds for the Central Universities (No. 06500073). SMW gratefully acknowledges support from a grant from the National Science Foundation (NSF-DMS 1418692) and partial support from the Mathematics Department at the University of California, Irvine. Part of this work was completed during a sabbatical visit (03/2015-06/2015) spent at that institution.
\newpage
\begin{figure}
        \begin{subfigure}{0.5\textwidth}
                \centering
                \includegraphics[width=\textwidth]{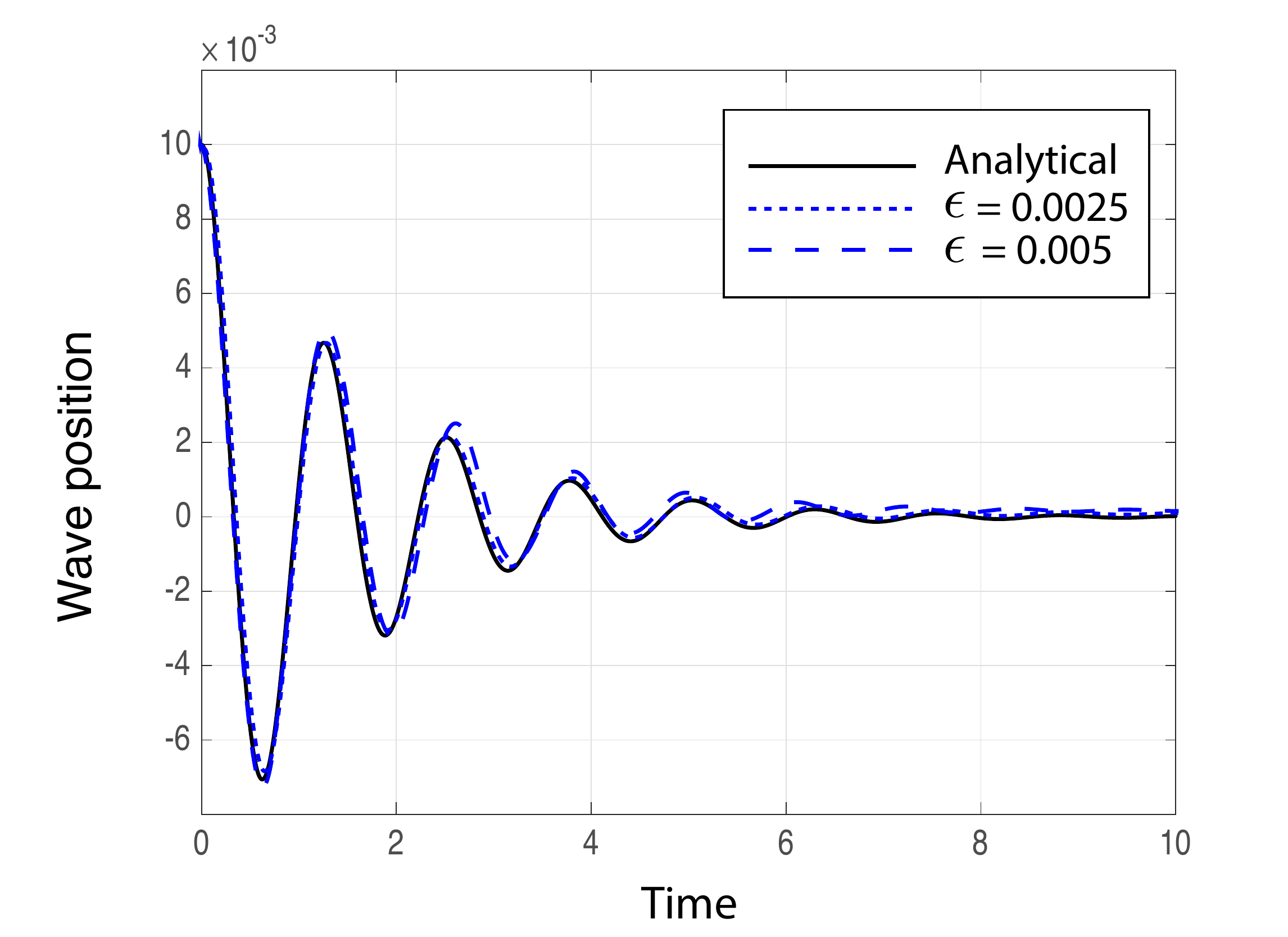}
                               \caption{Primitive method}
        \end{subfigure}%
        \begin{subfigure}{0.5\textwidth}
                \centering
                \includegraphics[width=\textwidth]{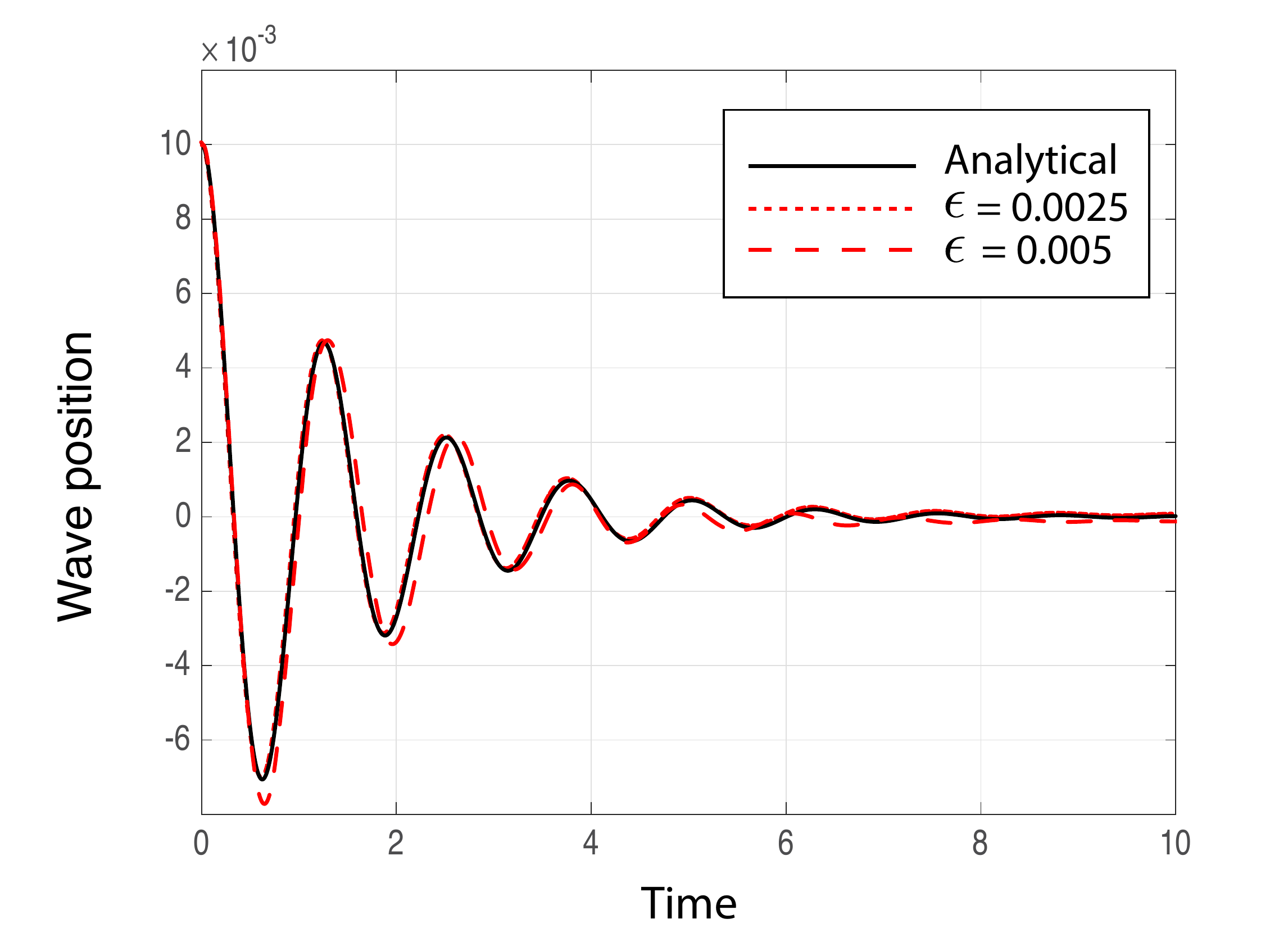}
                               \caption{Projection method}
        \end{subfigure}%
               \caption{Time evolution of capillary wave amplitude $H(t)$ in Case 1 with density ratio $1:10$ in \S \ref{sec-CW}. (a): results from Primitive methods; (b) results from Projection method.}\label{CAP-Den1to10-intf}
\end{figure}
\begin{figure}
        \begin{subfigure}{0.5\textwidth}
                \centering
                \includegraphics[width=\textwidth]{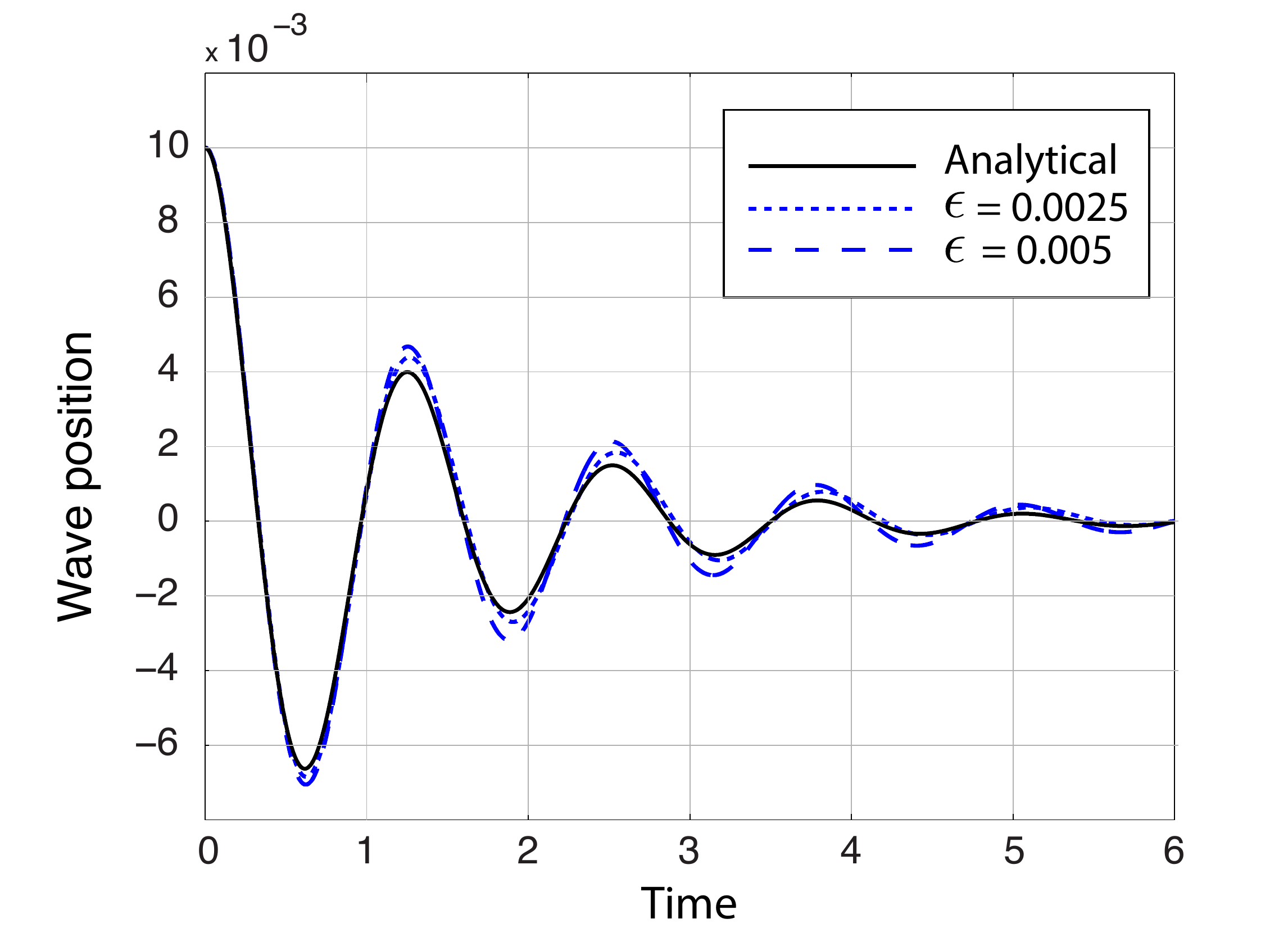}
                               \caption{Primitive method}
        \end{subfigure}%
        \begin{subfigure}{0.5\textwidth}
                \centering
                \includegraphics[width=\textwidth]{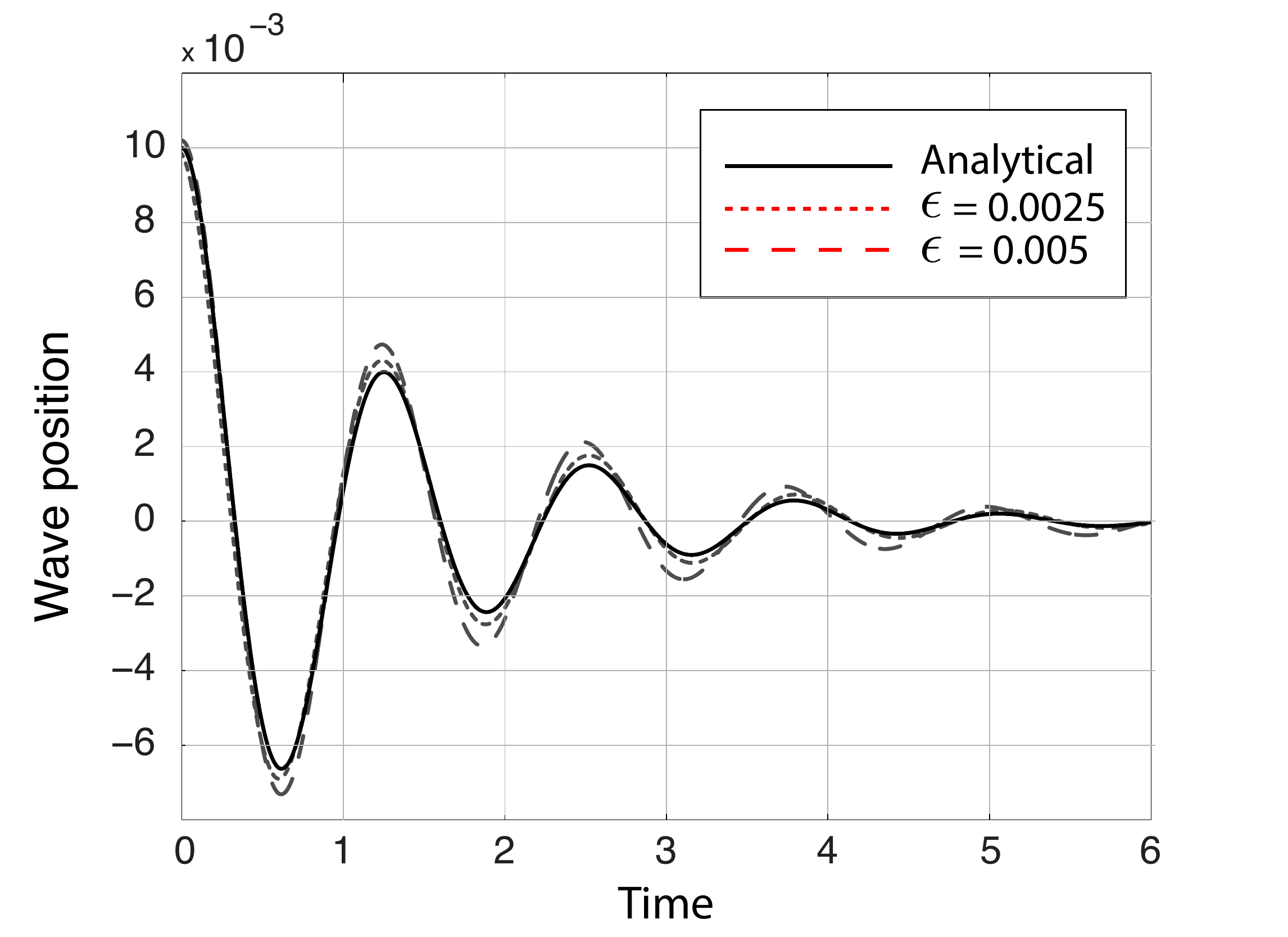}
                               \caption{Projection method}
        \end{subfigure}%
               \caption{Time evolution of capillary wave amplitude $H(t)$ in Case 2 with density ratio $1:1000$ in \S\ref{sec-CW}. (a): results from Primitive methods; (b) results from Projection method.}\label{CAP-Den1to1000-intf}
\end{figure}
\begin{figure}
        \begin{subfigure}{0.5\textwidth}
                \centering
                \includegraphics[width=\textwidth]{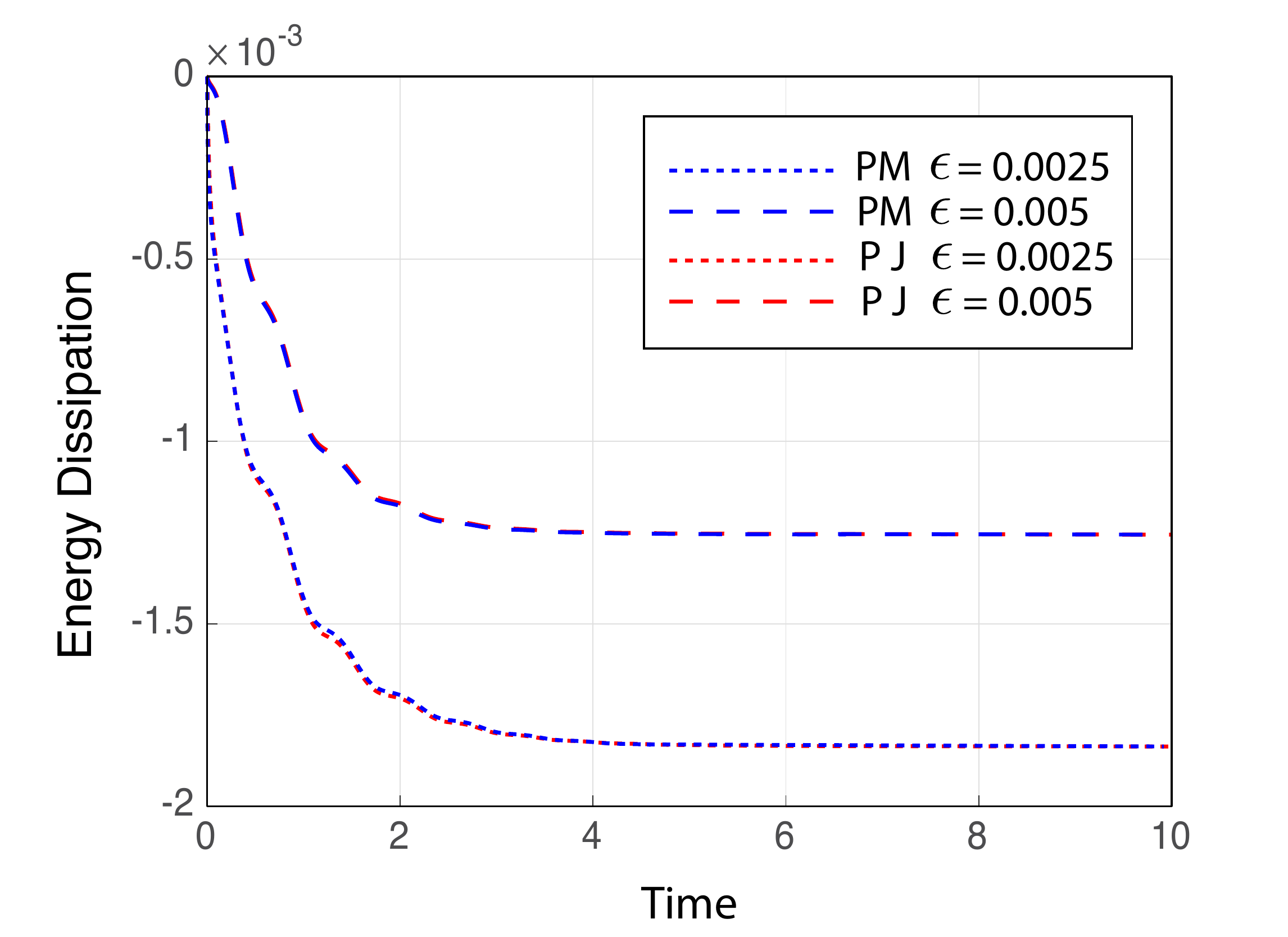}
                \caption{Case 1: Density ratio 1:10}
         \end{subfigure}
        \begin{subfigure}{0.5\textwidth}
                \centering
                \includegraphics[width=\textwidth]{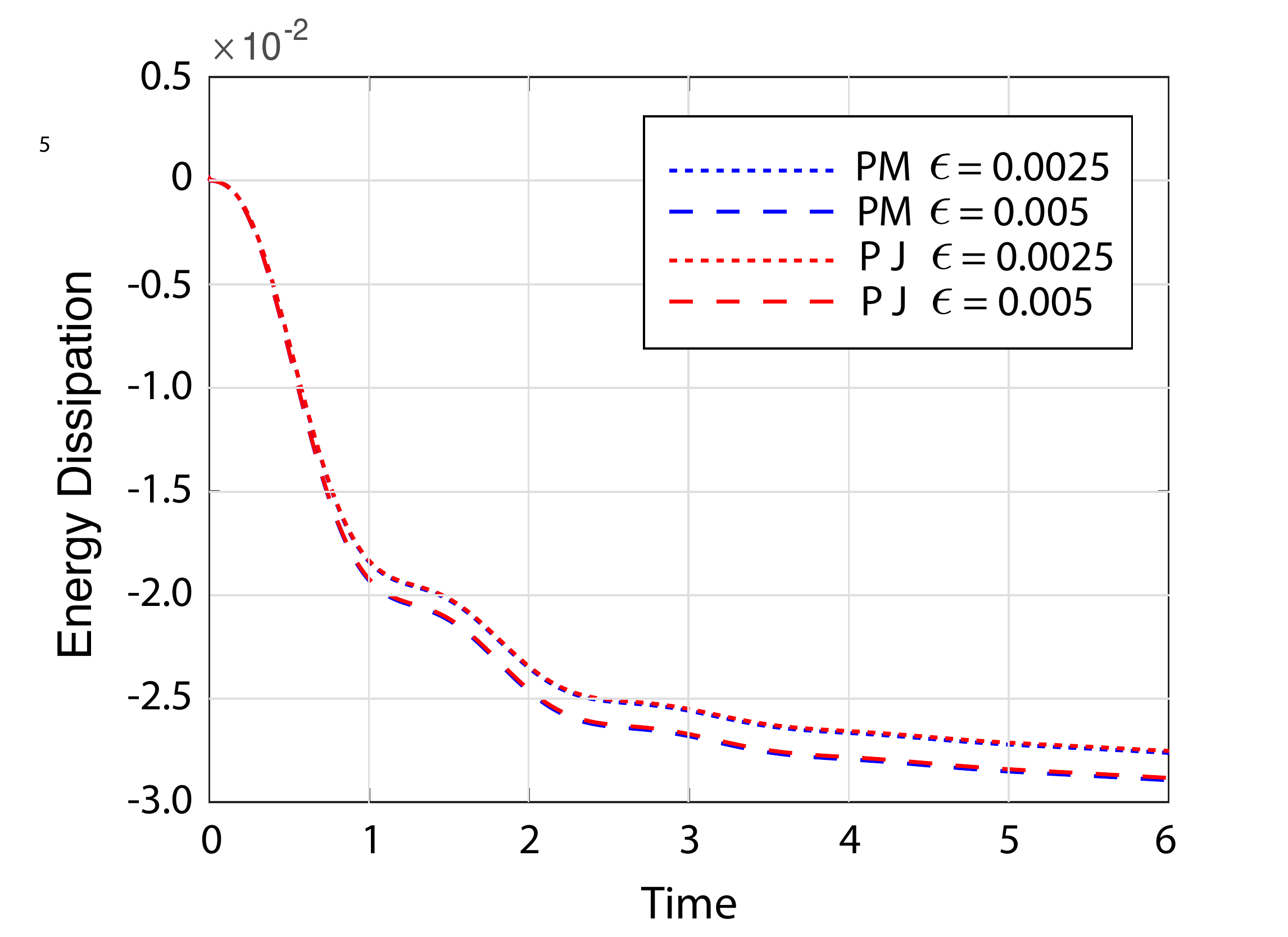}
                \caption{Case 2: Density Ratio 1:1000}
                         \end{subfigure}
               \caption{Time evolution of the energy difference $E^{n}-E^{0}$ of the binary fluid system for (a) Case 1 and (b) Case 2 with density ratio $1:10$ and $1:1000$ respectively in \S \ref{sec-CW}. The blue (red) dotted lines denotes the solution from Primitive (Projection) method with different values of $\epsilon$. Both methods give very similar results.}\label{CAP-Den1to10-En}
\end{figure}
\begin{figure}
        \begin{subfigure}{0.33\textwidth}
                \centering
                \includegraphics[width=\textwidth]{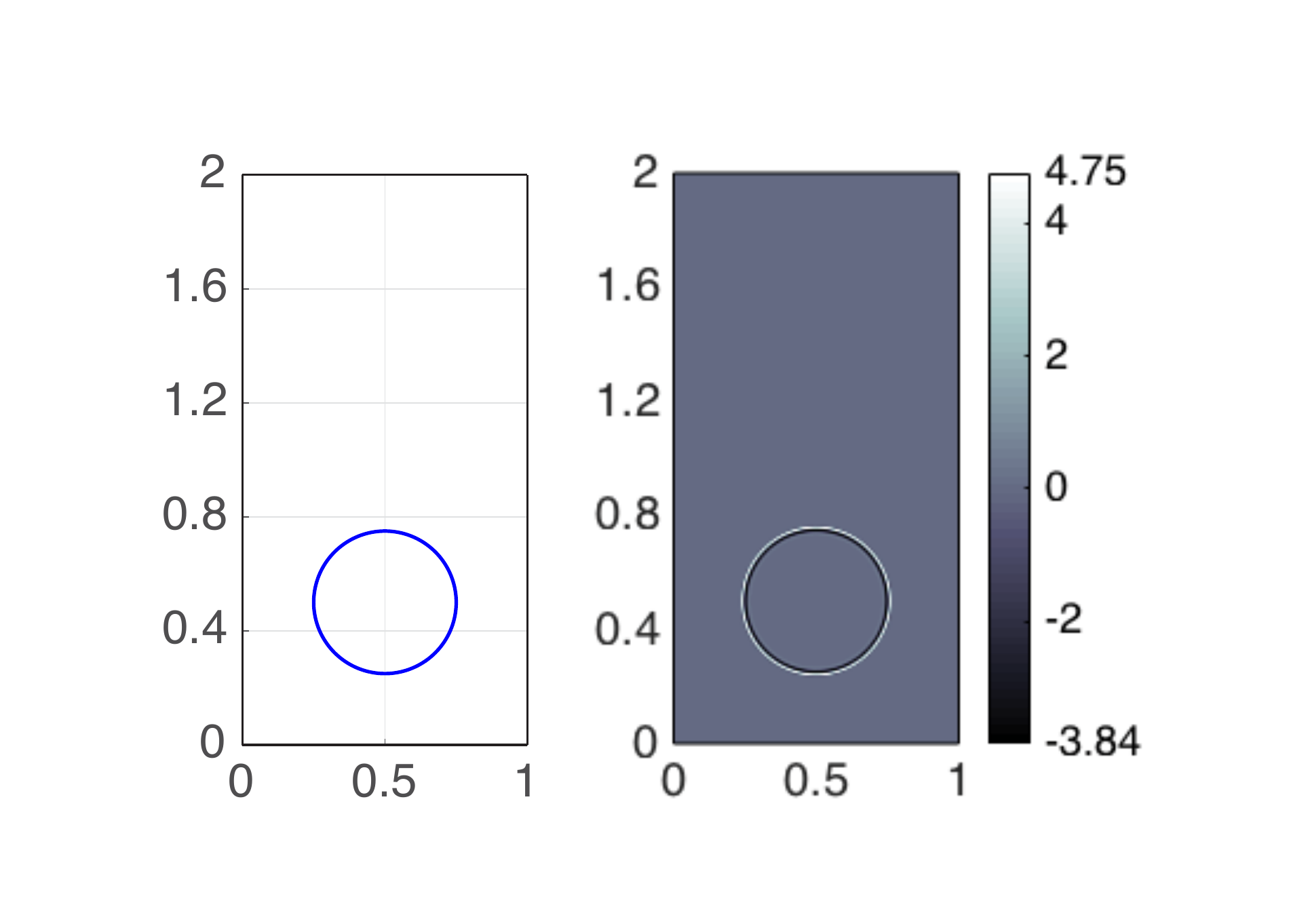}
                \caption*{t=0}
        \end{subfigure}%
        \begin{subfigure}{0.33\textwidth}
                \centering
                \includegraphics[width=\textwidth]{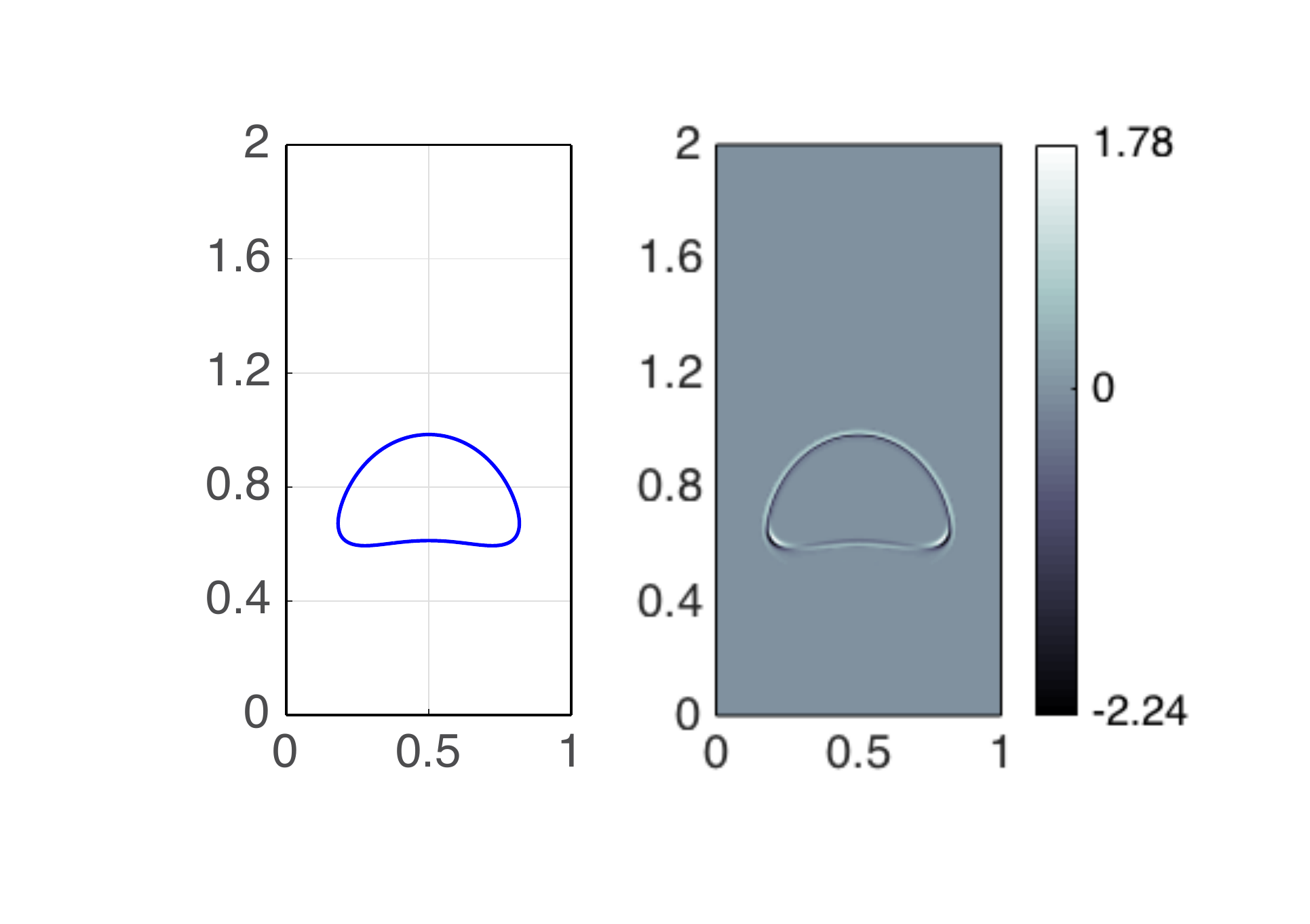}
                \caption*{t=1.5}
        \end{subfigure}%
        \begin{subfigure}{0.33\textwidth}
                \centering
                \includegraphics[width=\textwidth]{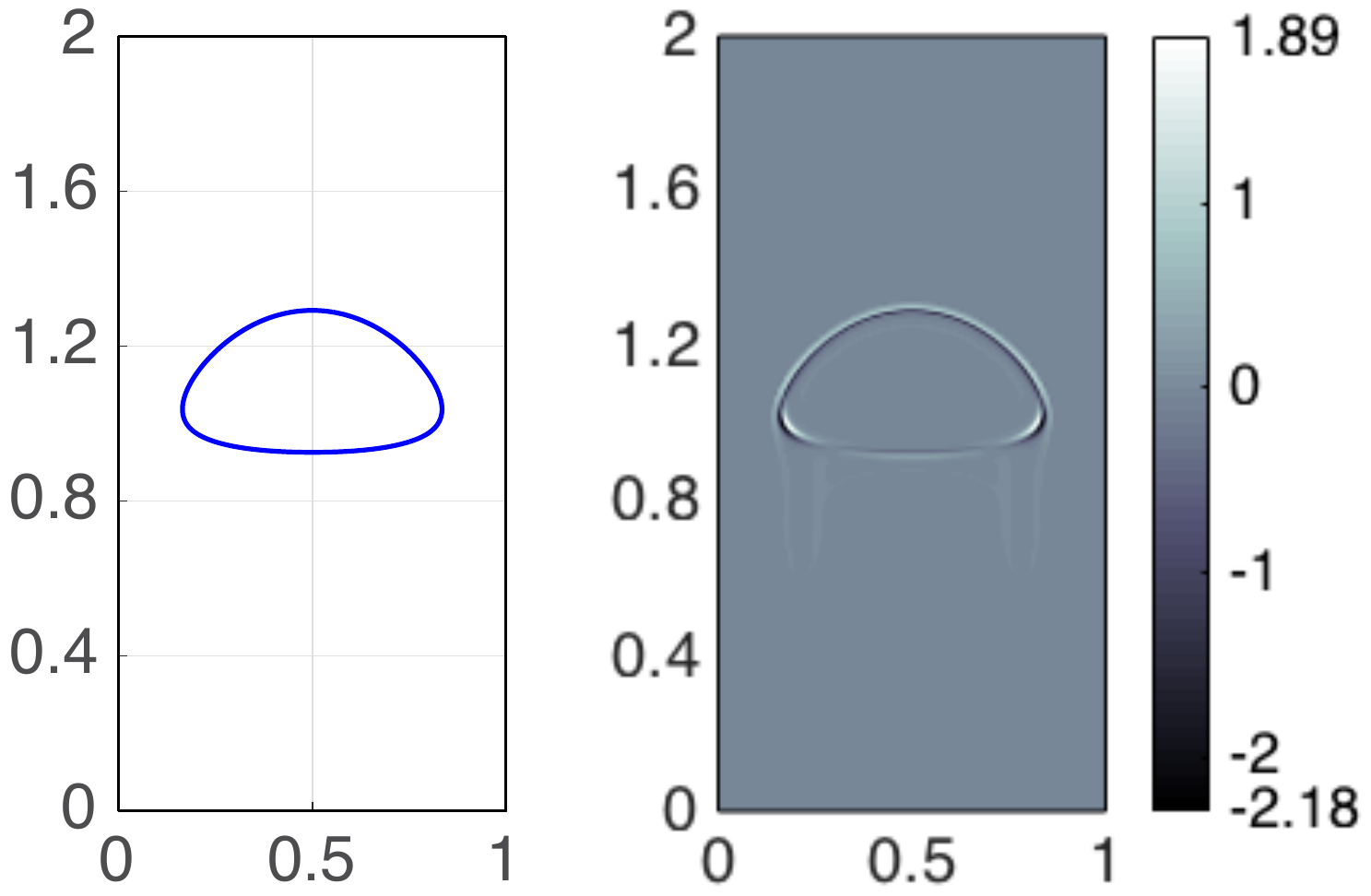}
                \caption*{t=3}
        \end{subfigure}
        \caption*{(a) Primitive method}
                \begin{subfigure}{0.33\textwidth}
                \centering
                \includegraphics[width=\textwidth]{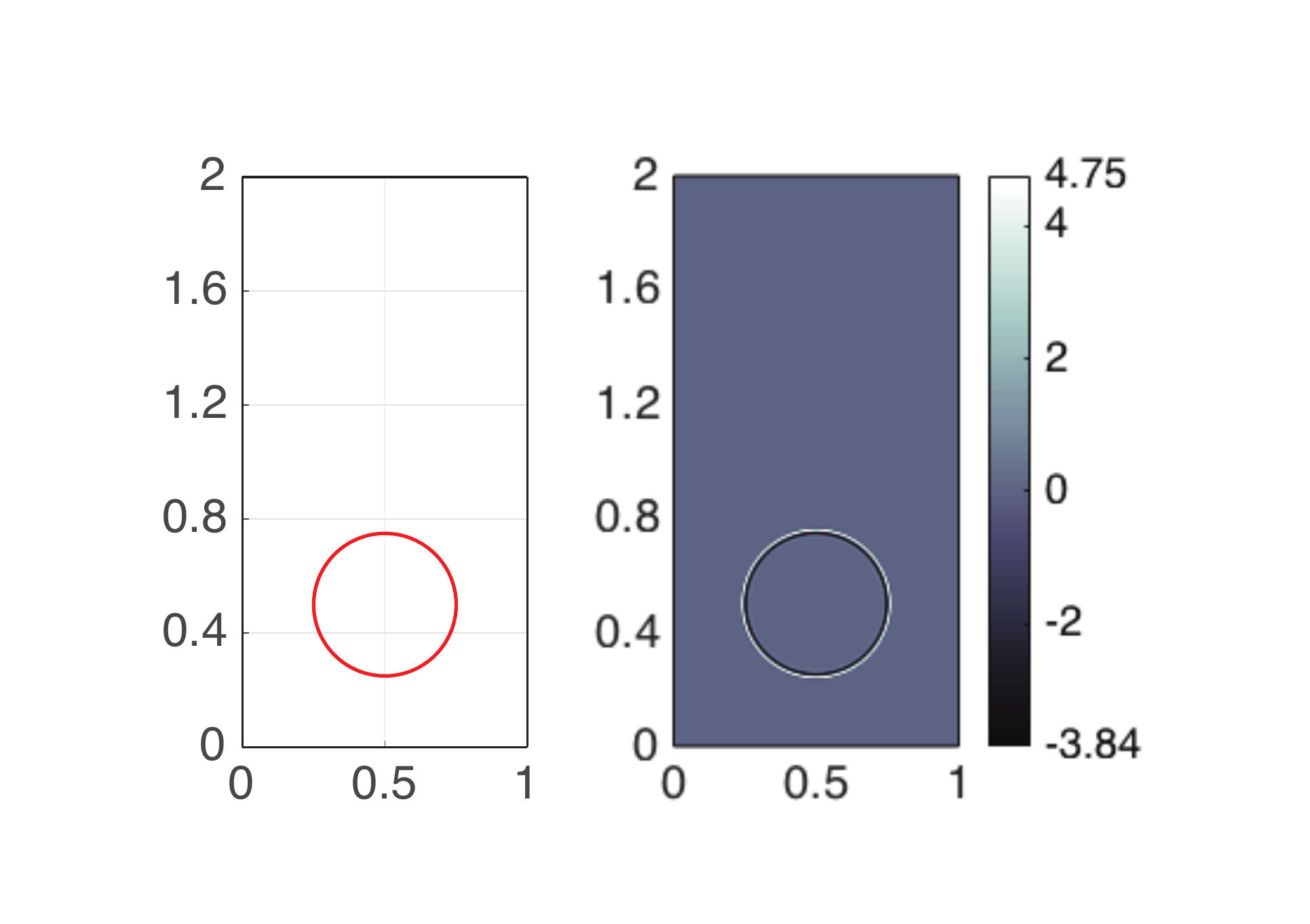}
                \caption*{t=0}
        \end{subfigure}%
        \begin{subfigure}{0.33\textwidth}
                \centering
                \includegraphics[width=\textwidth]{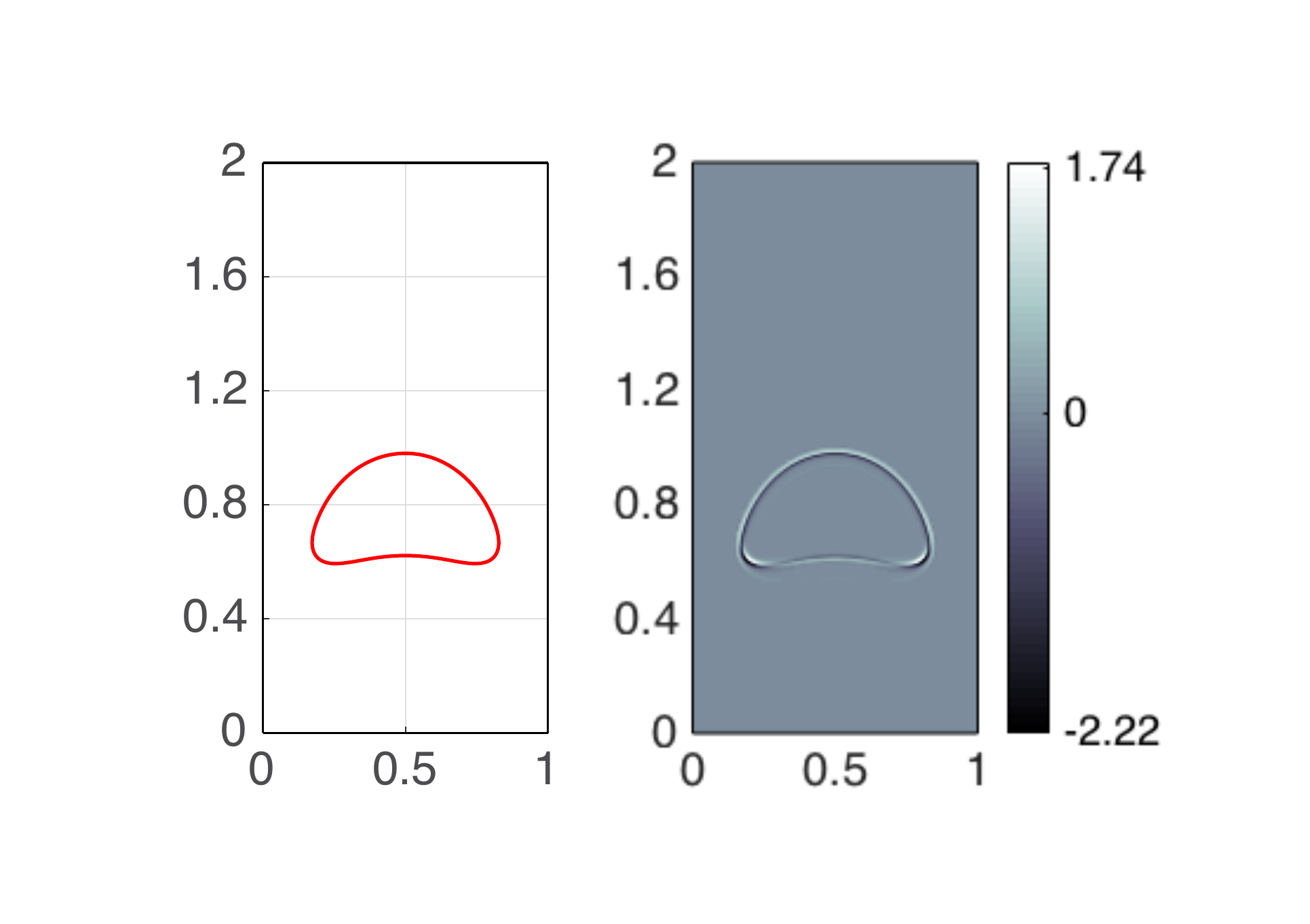}
                \caption*{t=1.5}
        \end{subfigure}
        \begin{subfigure}{0.33\textwidth}
                \centering
                \includegraphics[width=\textwidth]{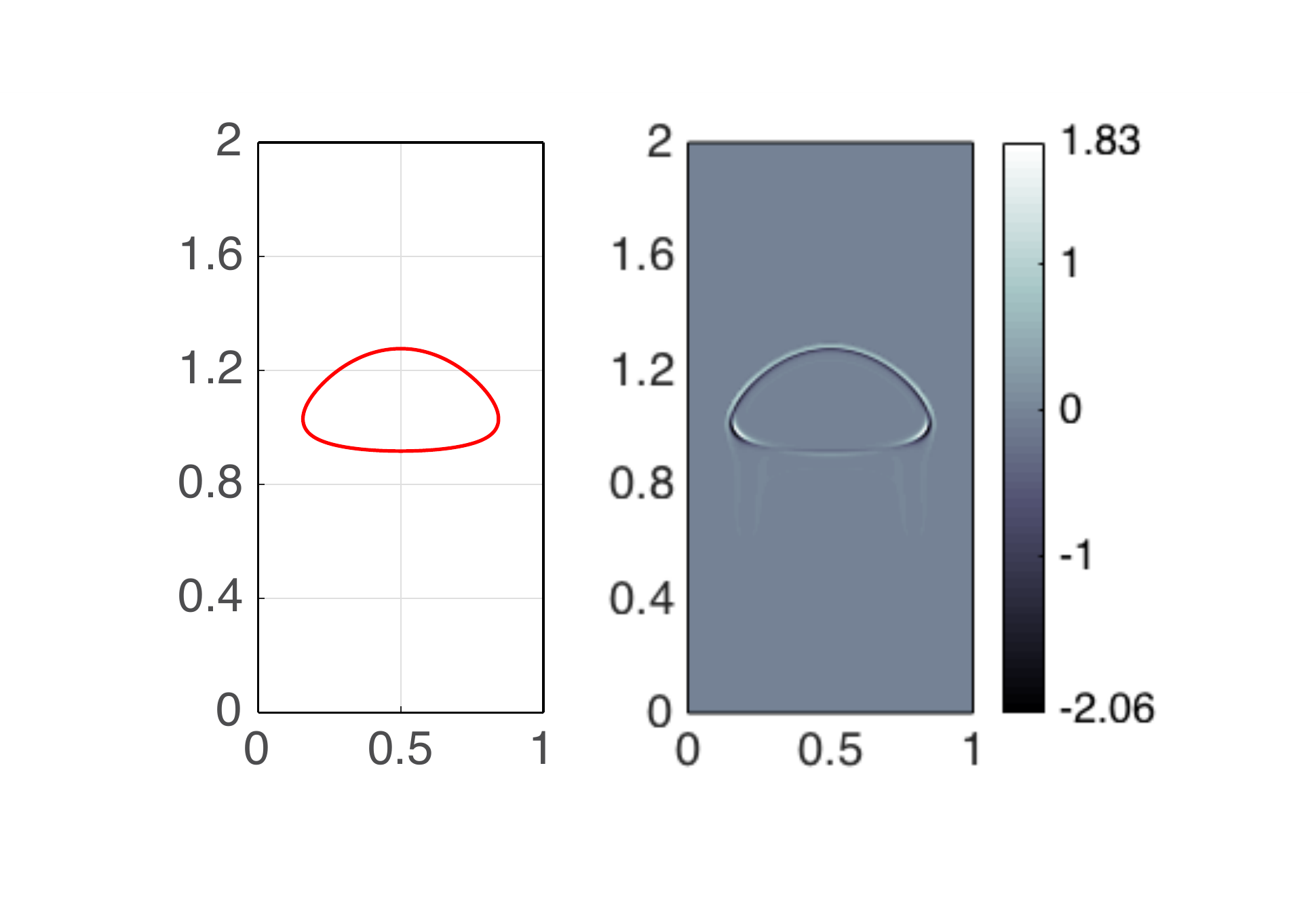}
                \caption*{t=3}
        \end{subfigure}
         \caption*{(b) Projection method}
               \caption{Deformed droplet interfaces (left) and $\boldsymbol{\nabla} \cdot \boldsymbol{u}$ (right) at different times, $t=0,~1.5,~3$ for \S \ref{sec-RS}. (a): results from Primitive methods; (b) results from Projection method.}\label{RS-Den1to10-quasi}
\end{figure}
\begin{figure}
        \begin{subfigure}{0.5\textwidth}
                \centering
                \includegraphics[width=\textwidth]{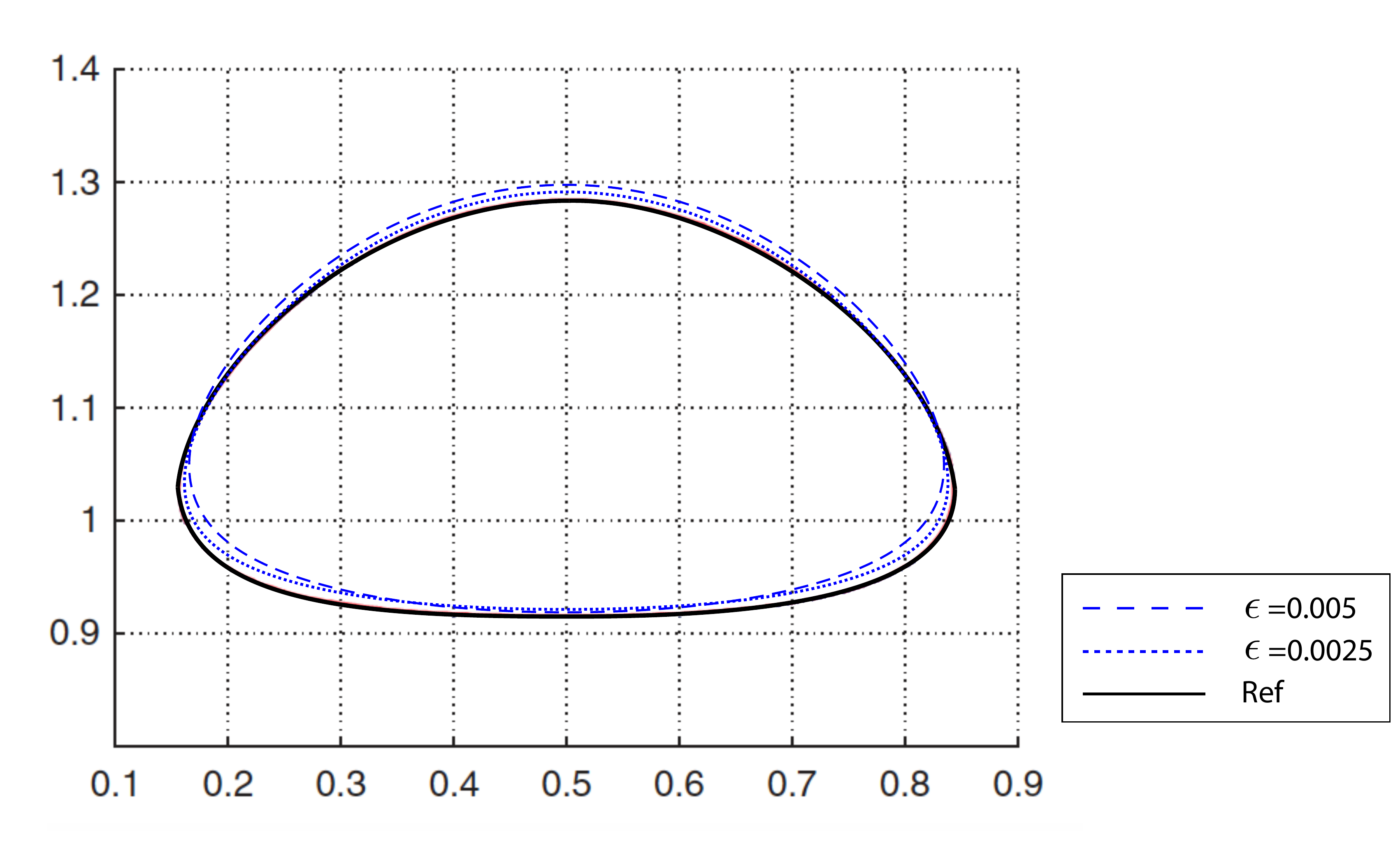}
                \caption{Primitive method}
        \end{subfigure}%
        \begin{subfigure}{0.5\textwidth}
                \centering
                \includegraphics[width=\textwidth]{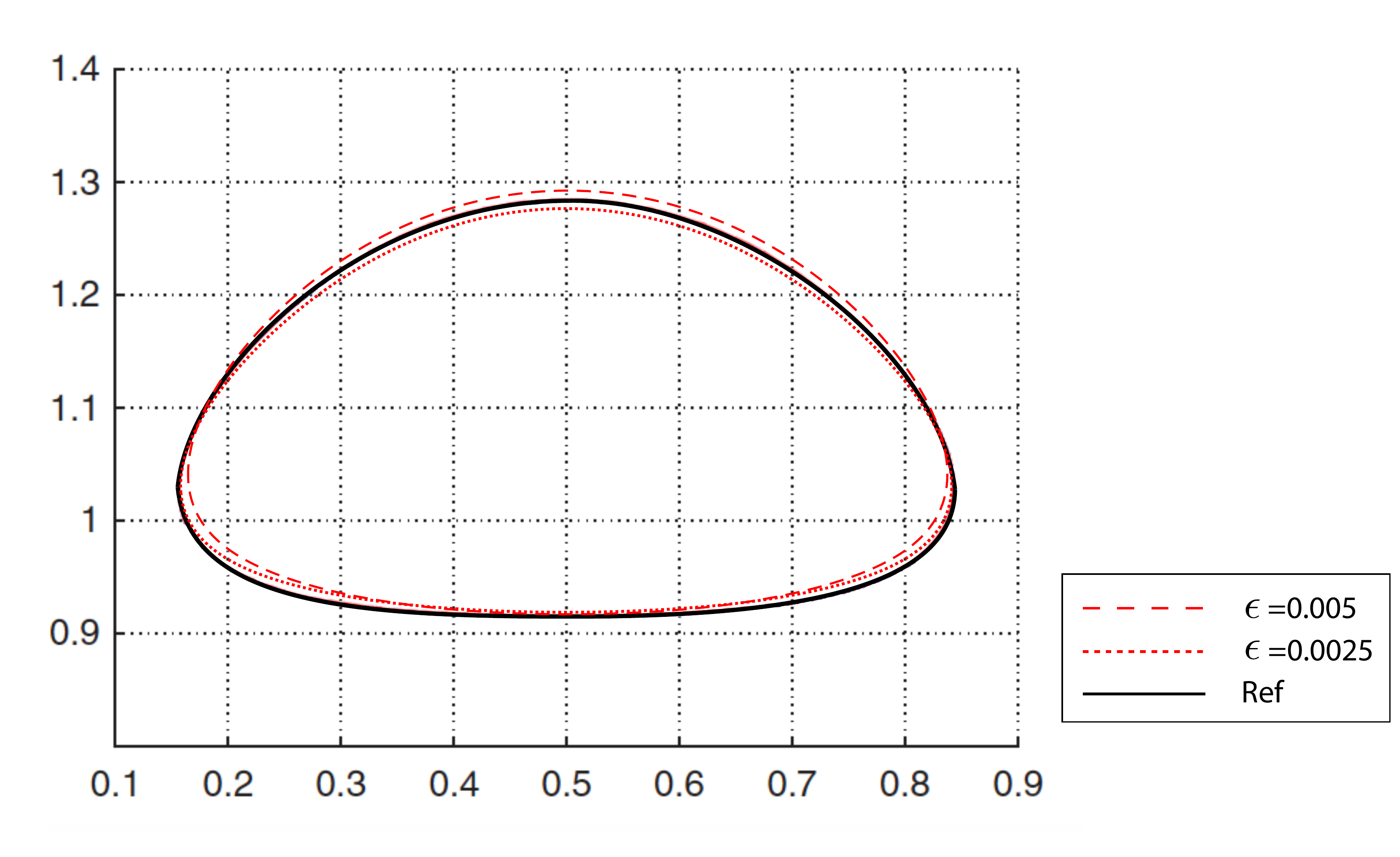}
                \caption{Projection method}
        \end{subfigure}
               \caption{Comparison between droplet shapes at $t=3$ in \S \ref{sec-RS}.  (a) The comparison between the results from our primitive method (blue dotted lines) and the result (black solid line) from Ref \cite{Mark-JFM-1997} by using a sharp interface model. (b) The comparison between the results from our projection method (red dotted lines) and the result (black solid line) from Ref \cite{Mark-JFM-1997}.}\label{RS-Den1to10-Intf}
\end{figure}
\begin{figure}
        \begin{subfigure}{0.5\textwidth}
                \centering
                \includegraphics[width=\textwidth]{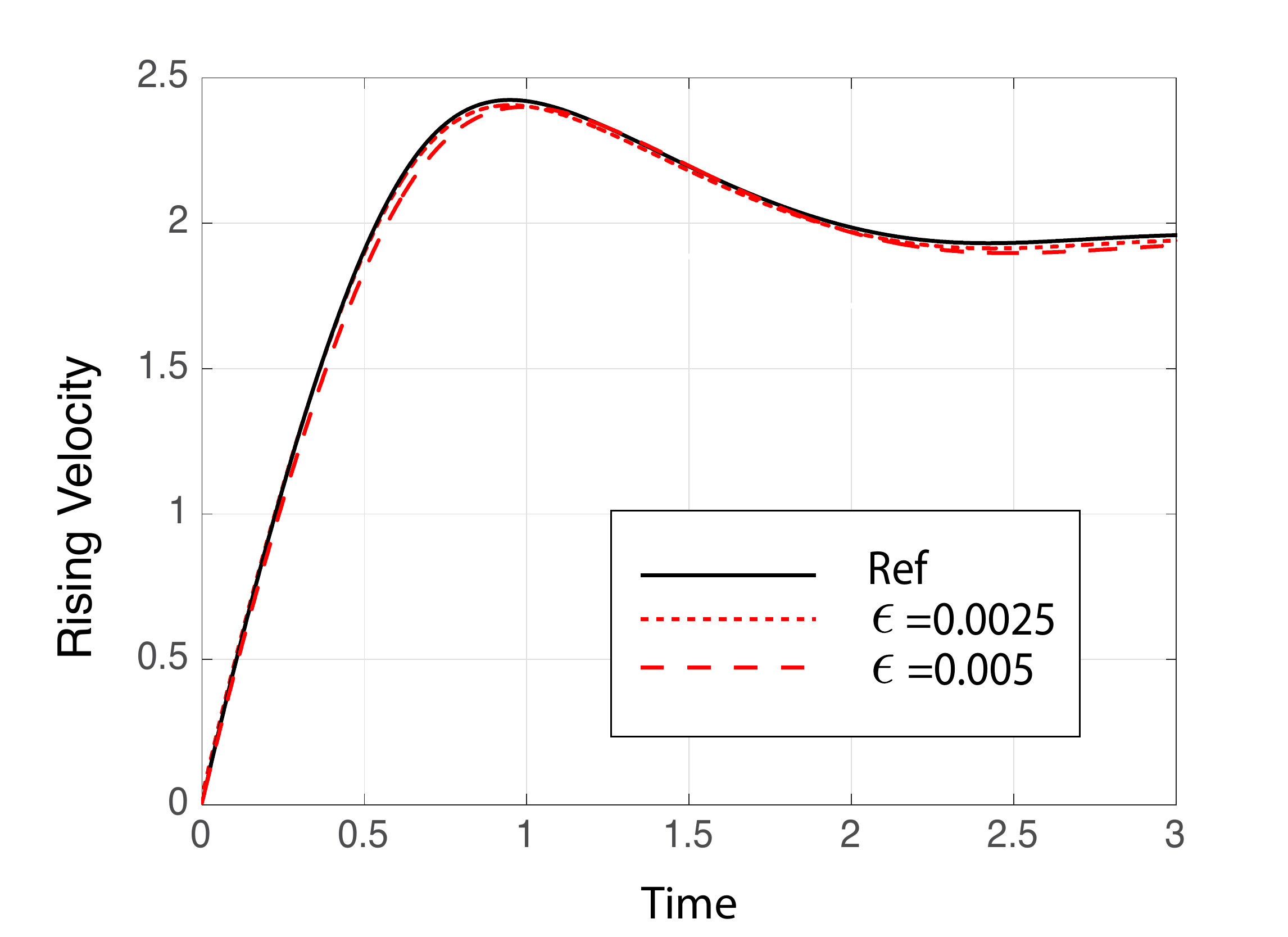}
               \caption{Primitive method}
        \end{subfigure}%
        \begin{subfigure}{0.5\textwidth}
                \centering
                \includegraphics[width=\textwidth]{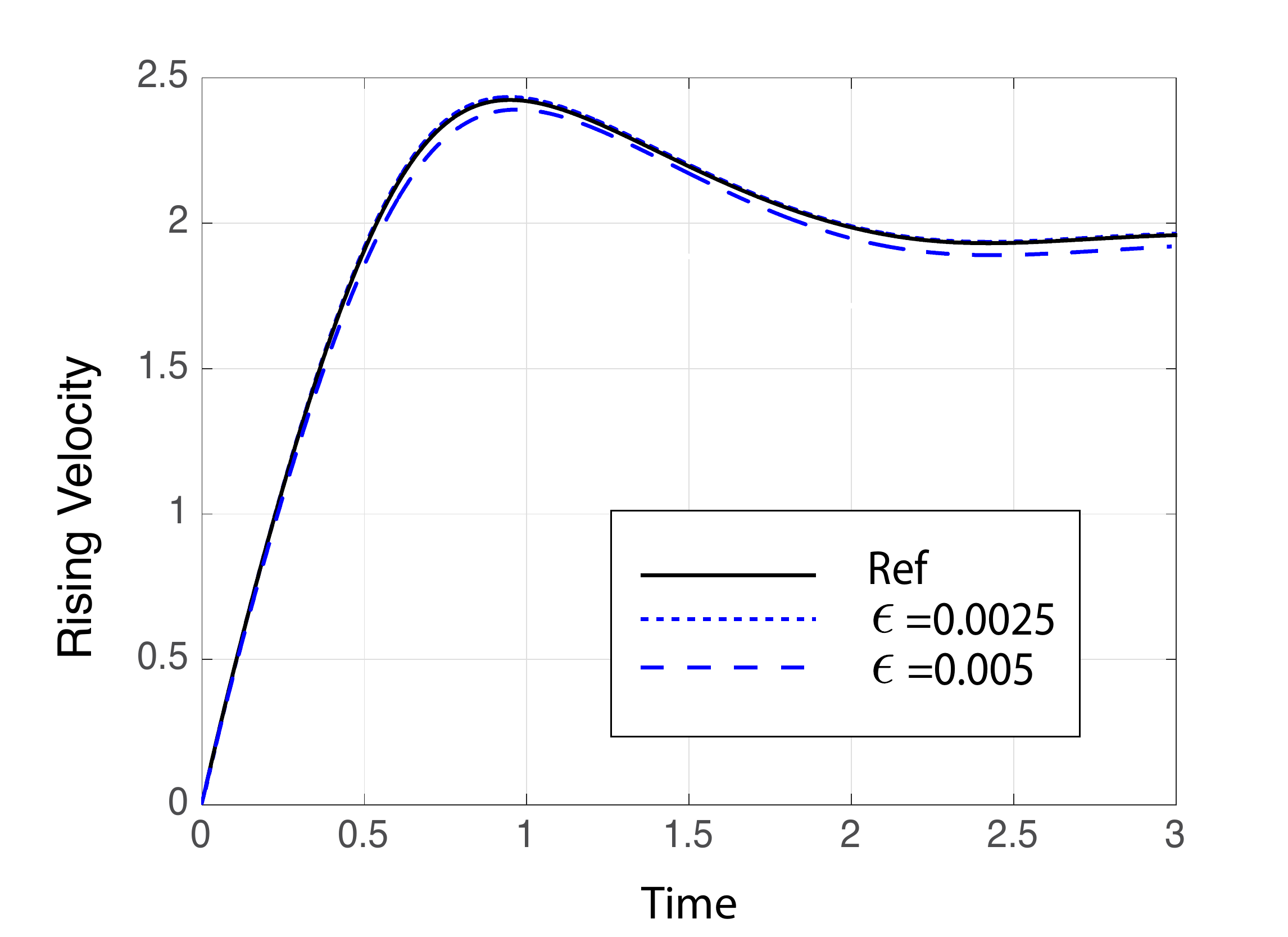}
               \caption{Projection method}
        \end{subfigure}%
                       \caption{Comparisions of benchmark quantity: rising velocity of a droplet in \S \ref{sec-RS}. (a) Primitive method vs Ref \cite{Mark-JFM-1997} by using a sharp interface model; (b) Projection method vs Ref \cite{Mark-JFM-1997}.}\label{RS-Den1to10-Vel}
\end{figure}
\begin{figure}
                \centering
                \includegraphics[width=0.5\textwidth]{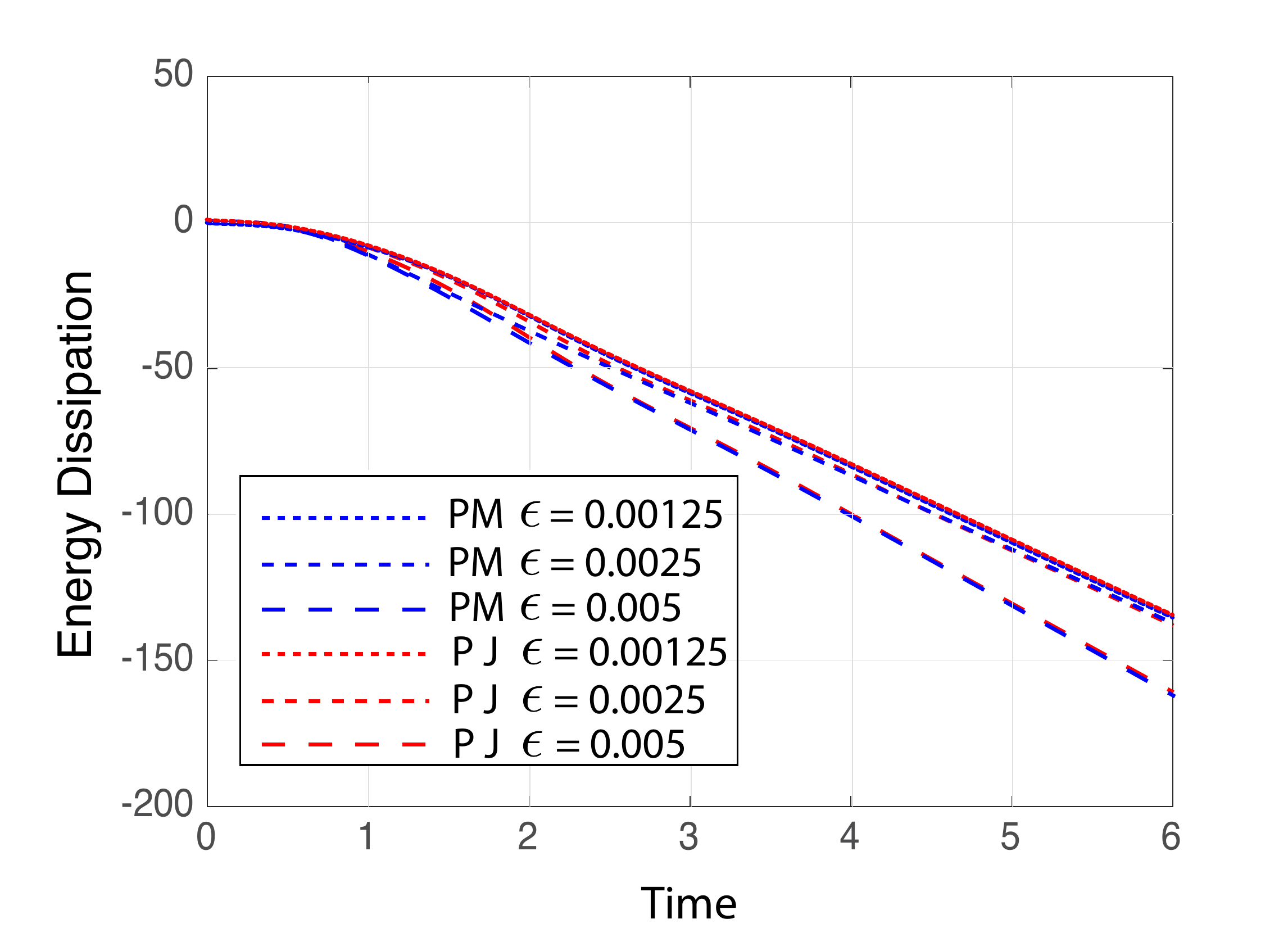}
               \caption{Time evolution of the energy dissipation $E^{n}-E^{0}$ of the binary fluid system in \S \ref{sec-RS}. The blue (red) dotted lines denotes the solution from Primitive (Projection) method with different values of $\epsilon$.}\label{RS-Den1to10-En}
\end{figure}
 \begin{figure}
 \vspace{-10mm}
        \begin{subfigure}{0.5\textwidth}
                \centering
                \includegraphics[width=\textwidth]{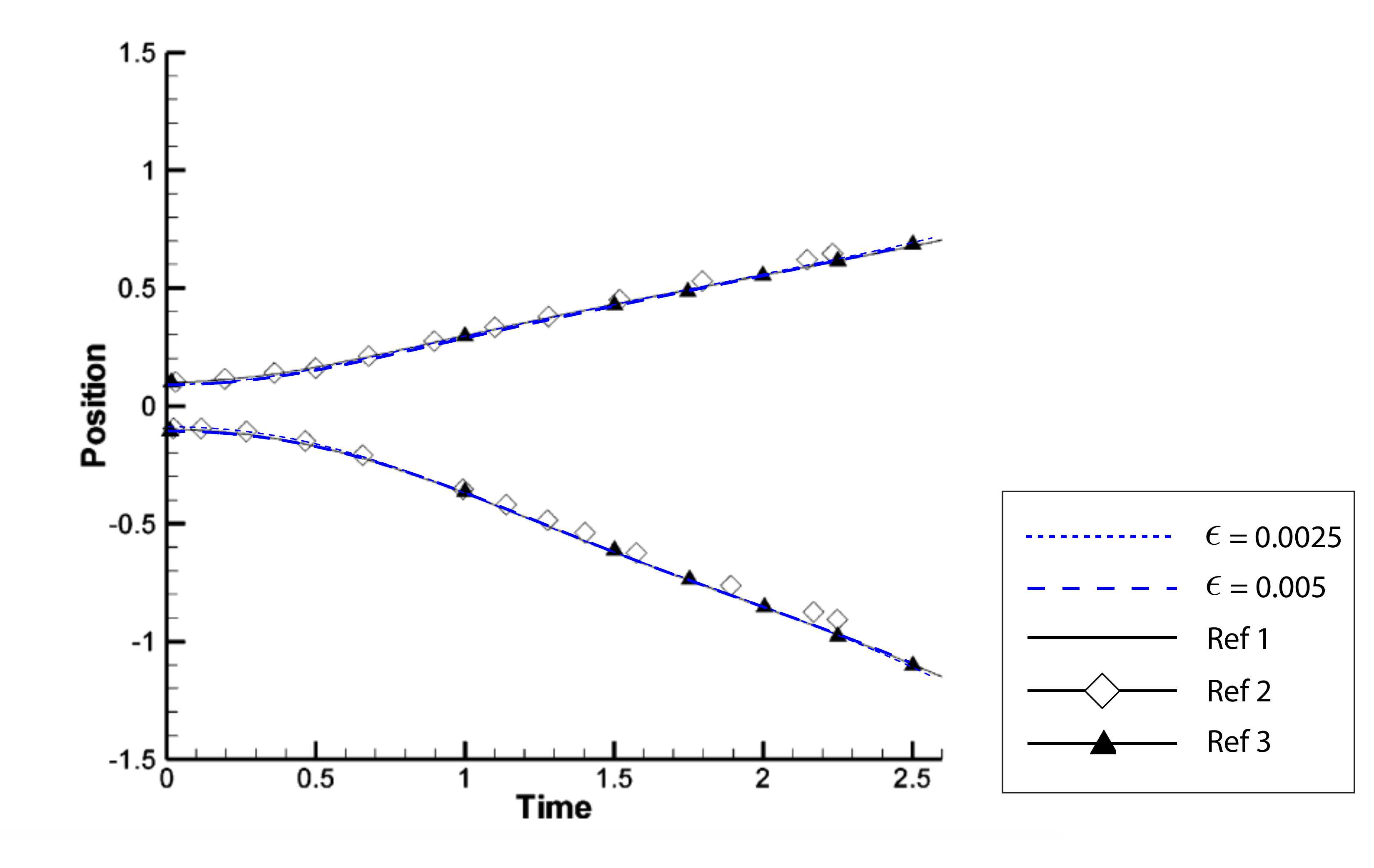}
                               \caption{Primitive method}
        \end{subfigure}%
        \begin{subfigure}{0.5\textwidth}
                \centering
                \includegraphics[width=\textwidth]{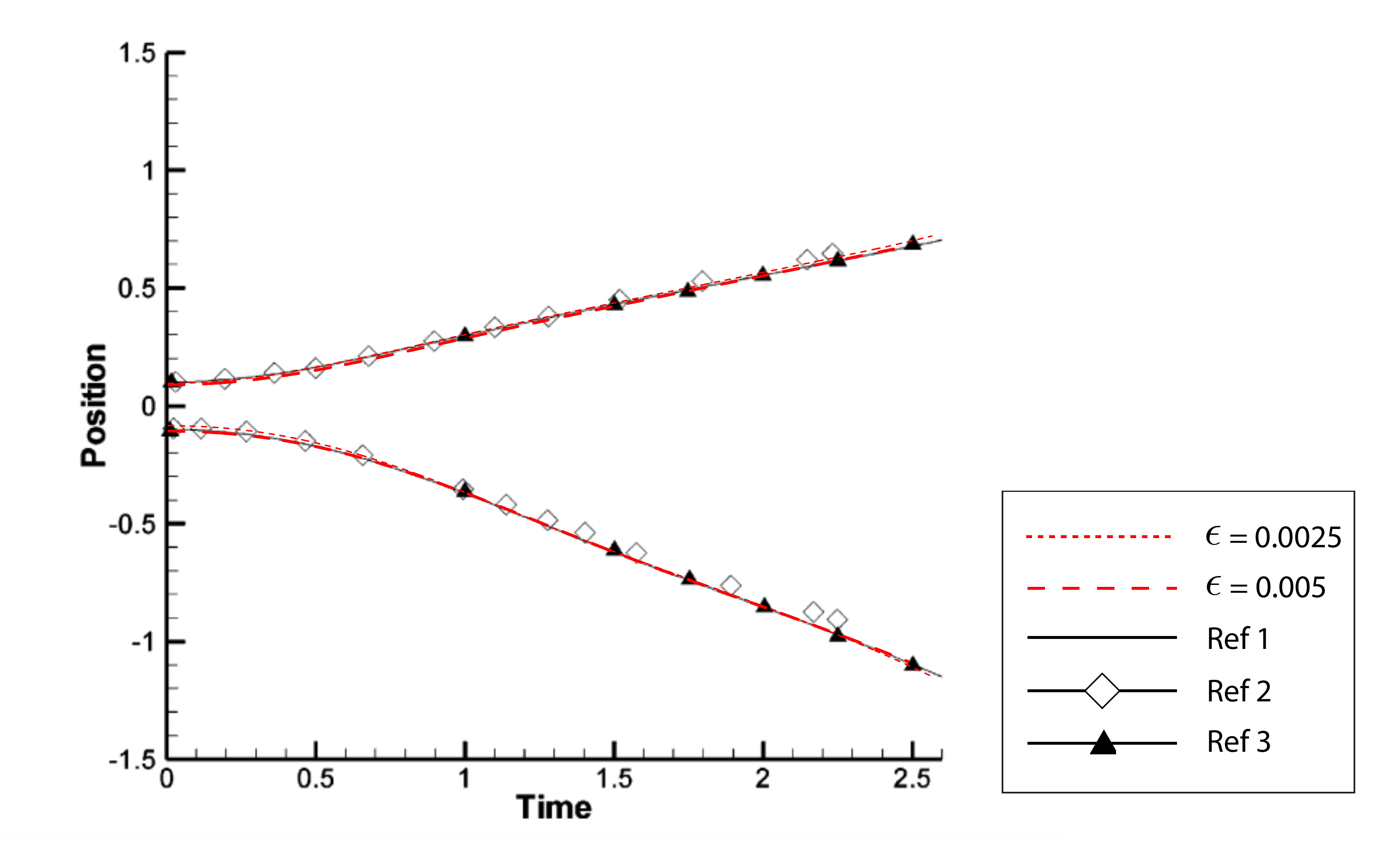}
                               \caption{Projection method}
        \end{subfigure}%
                      \caption{Comparison between the numerical results. The y-coordinate of the tip of the falling and rising fluid versus time: the open diamonds represents the solution of Tryggvason \cite{Tryggvason1988}, the filled triangles that of Guermond et al. \cite{Guermond1967}, the black solid line represents that of Ding \cite{Ding2007}, and (a) the blue doted lines denote the solution from Primitive method; (b) the red doted lines denote the solution from Projection method.} \label{RT-position}
\end{figure}
\begin{figure}
\vspace{0mm}
\centering
        \begin{subfigure}{0.95\textwidth}
                \centering
                 \includegraphics[width=\textwidth]{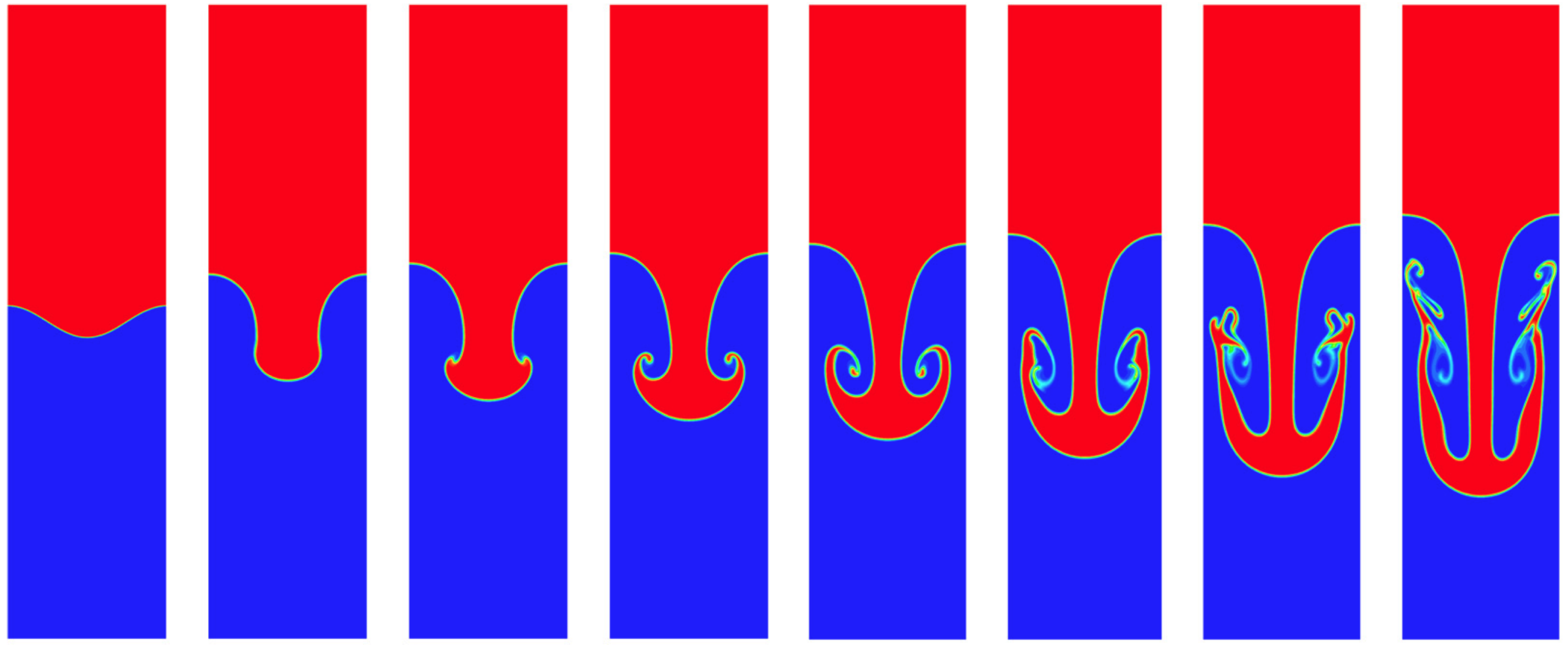}
\caption{Existing numerical results obtained from \cite{Ding2007}}
        \end{subfigure}
        \begin{subfigure}{0.95\textwidth}
                \centering
                 \includegraphics[width=\textwidth]{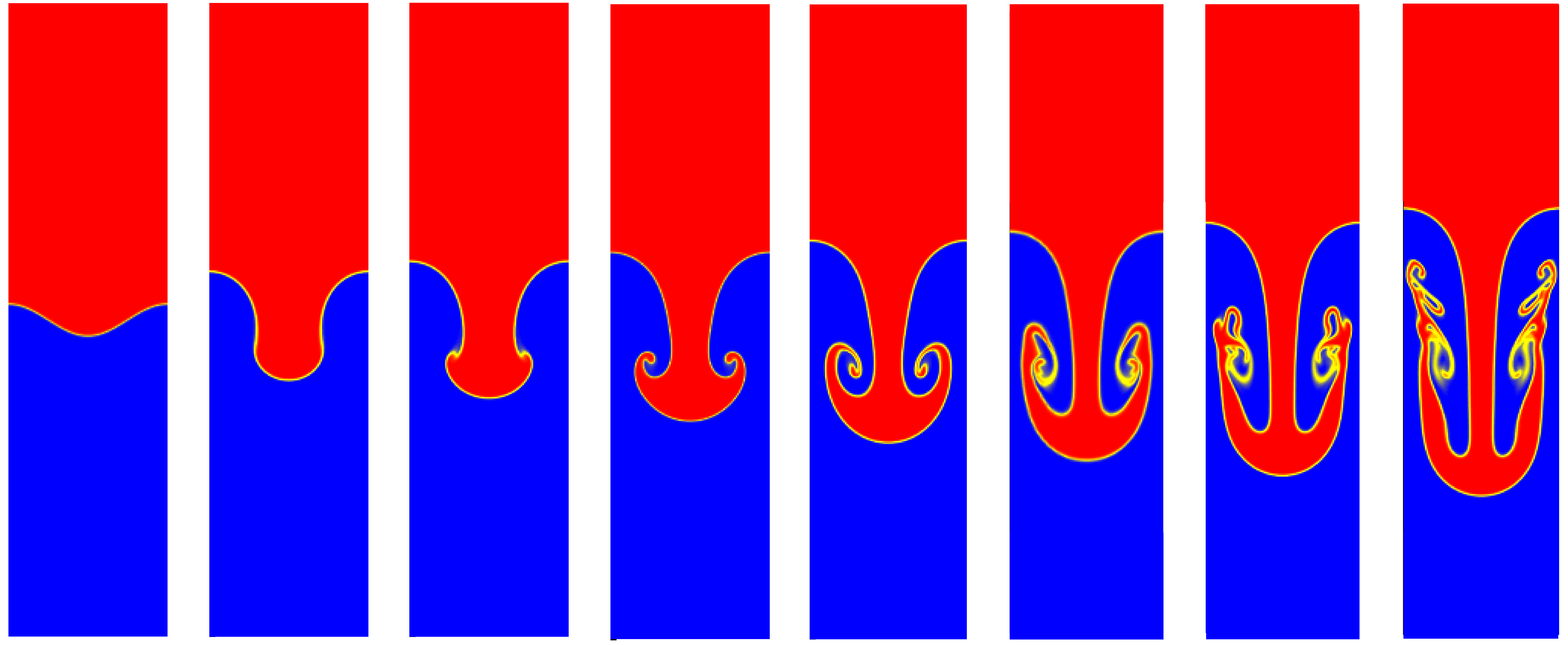}
             \caption{Numerical results obtained from projection method} 
        \end{subfigure}
              \caption{Rayleigh-Taylor instability simulation at different times ($t = 0,1,1.25,1.5,1.75, 2, 2.25, 2.5$), with density ratio 1:3 in \S \ref{sec-RT}. The numerical results from the projection method with $\epsilon=0.0025$ are shown at the bottom being compared with the results in \cite{Ding2007} at the top by using a different diffuse-interface model.}\label{RT-2}
\end{figure}
\begin{figure}
                \centering
                \includegraphics[width=0.5\textwidth]{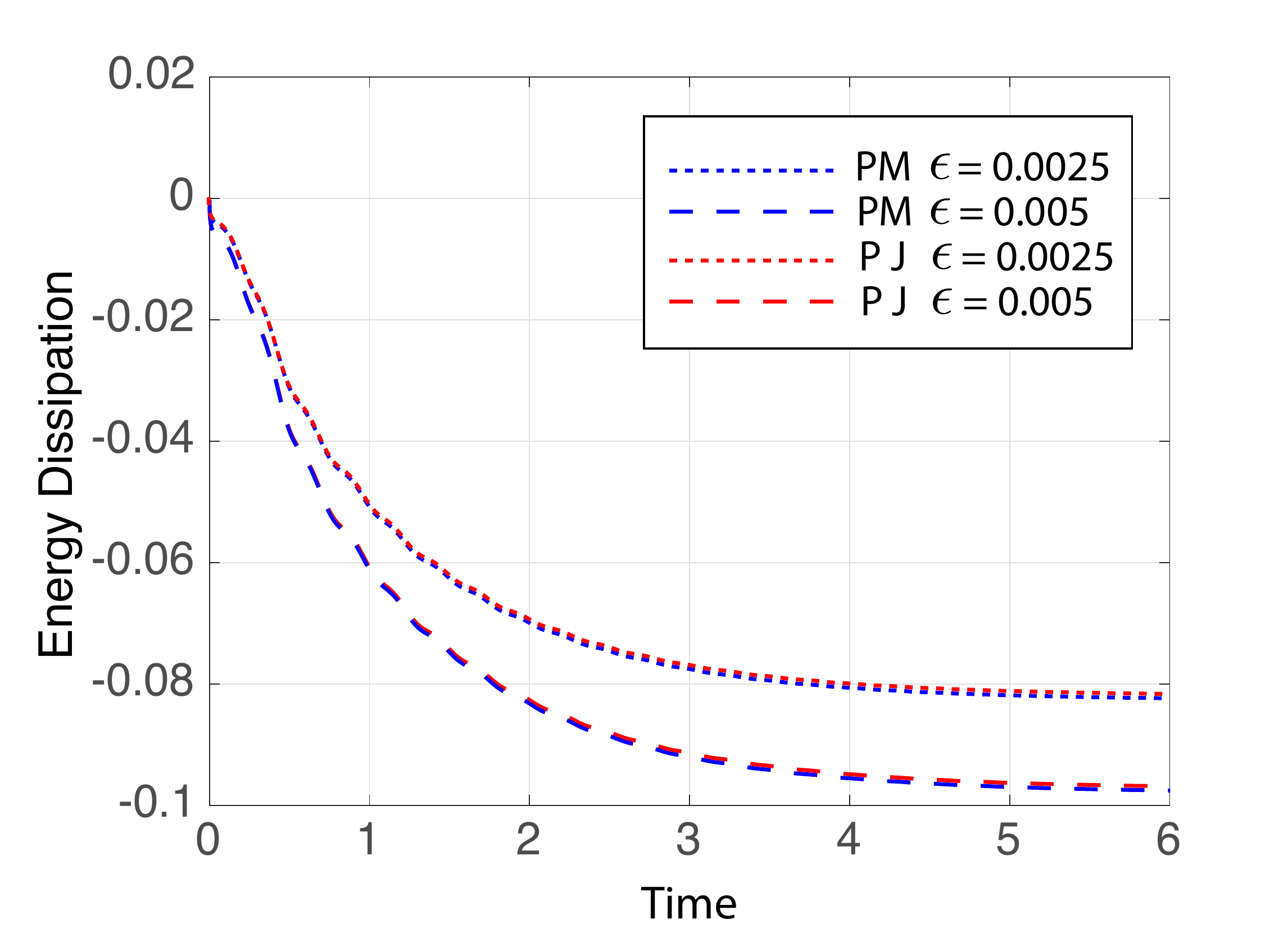}
               \caption{Time evolution of the energy dissipation $E^{n}-E^{0}$ of the binary fluid system in \S \ref{sec-RT}. The blue (red) dotted lines denotes the solution from Primitive (Projection) method with different values of $\epsilon$.}\label{RT-En}
\end{figure}
\newpage
~\\
\newpage
\begin{appendix}
\section{Finite Difference Discretization on a Staggered Grid}
\subsection{Basic Definitions and Properties}\label{app-A1}
Here we use the notations and results for cell-centered functions from \cite{SteveMultigrid-2010, SteveMultigrid2013, SteveMultigrid-2009}. The reader is directed there for complete details. We begin with definitions of grid functions and difference operators needed for our discretization of a two-dimensional staggered grid. Throughout this appendix, we use the following symbols to denote the cell-centered, edge-centered and vertex-centered functions, such that
\begin{align}
{\rm cell~centered~functions}:&~\phi, \psi,\zeta \in \mathcal{C}_{m_{1}\times m_{2}}\cup \mathcal{C}_{\overline{m}_{1}\times m_{2}}\cup \mathcal{C}_{m_{1}\times \overline{m}_{2}}\cup \mathcal{C}_{\overline{m}_{1}\times \overline{m}_{2}},\notag\\
{\rm east~west~edge~centered~functions}:&~u,\gamma \in \mathcal {E}^{ew}_{m_{1}\times m_{2}}\cup\mathcal {E}^{ew}_{m_{1}\times \overline{m}_{2}}\cup\mathcal {E}^{ew}_{\overline{m}_{1}\times m_{2}}\cup\mathcal {E}^{ew}_{\overline{m}_{1}\times \overline{m}_{2}}\notag\\
{\rm north~south~edge~centered~functions}:&~v,\omega \in \mathcal {E}^{ns}_{m_{1}\times m_{2}}\cup\mathcal {E}^{ns}_{m_{1}\times \overline{m}_{2}}\cup\mathcal {E}^{ns}_{\overline{m}_{1}\times m_{2}}\cup\mathcal {E}^{ns}_{\overline{m}_{1}\times \overline{m}_{2}}\notag\\
{\rm vertex~centered~functions}:&~f,g \in \mathcal {V}_{m_{1}\times m_{2}}.\notag
\end{align}
Here we use the function spaces with over-lined subscript $\overline{m}_{1}$ and $\overline{m}_{2}$ to denote the space that include the ghost points in the $x$ and $y$ direction respectively. In component form, we define 
\begin{align}
\phi_{i,j} &:= \phi(x_i,y_j),~~~~u_{i+\frac{1}{2},j}:=u(x_{i+\frac{1}{2}}, y_{j}),\notag\\
v_{i,j+\frac{1}{2}}&:=v(x_{i}, y_{j+\frac{1}{2}}),~~~~f_{i+\frac{1}{2},j+\frac{1}{2}}:=f (x_{i+\frac{1}{2}}, y_{j+\frac{1}{2}}),\notag
\end{align}
where $x_i = (i-\frac{1}{2} ) \cdot h$, $y_j = (j-\frac{1}{2} ) \cdot h$, and $i$ and $j$ can take on integer values.
\subsection{Average and Difference Operators}\label{app-A2}
We define the edge-to-center average and difference operators $a_{x},~d_x : \mathcal {E}^{ew}_{m_{1}\times m_{2}} \rightarrow \mathcal{C}_{m_{1}\times m_{2}} $ and $a_{y},~d_y : \mathcal {E}^{ns}_{m_{1}\times m_{2}} \rightarrow \mathcal{C}_{m_{1}\times m_{2}} $ component-wise via
\begin{align}
a_{x}u_{i,j}&=\cfrac{1}{2}\bigg(u_{i+\frac{1}{2},j}+u_{i-\frac{1}{2},j}\bigg),~~~~d_{x}u_{i,j}=\cfrac{1}{h}\bigg(u_{i+\frac{1}{2},j}-u_{i-\frac{1}{2},j}\bigg),\\
a_{y}v_{i,j}&=\cfrac{1}{2}\bigg(v_{i,j+\frac{1}{2}}+v_{i,j-\frac{1}{2}}\bigg),~~~~~d_{y}v_{i,j}=\cfrac{1}{h}\bigg(v_{i,j+\frac{1}{2}}-v_{i,j-\frac{1}{2}}\bigg),
\end{align}
for $i=1,\cdots,m_{1}$ and $j=1,\cdots ,m_{2}$.\\
The center-to-edge average and difference operators, $A_x$, $D_x$ : $\mathcal {C}_{\overline {m}_{1} \times m_{2}} \rightarrow \mathcal {E}^{ew}_{m_{1}\times m_{2}}$ and $A_{y}$, $D_{y}$ : $\mathcal{C}_{m_{1}\times \overline{m}_{2}}\rightarrow \mathcal{E}^{ns}_{m_{1}\times m_{2}}$ are defined component-wise as
\begin{align}
A_{x}\phi_{i+\frac{1}{2},j} = \cfrac{1}{2}(\phi_{i+1,j}+\phi_{i,j}),~~~~D_{x}\phi_{i+\frac{1}{2},j}=\cfrac{1}{h}(\phi_{i+1,j}-\phi_{i,j}),\\
A_{y}\phi_{i,j+\frac{1}{2}} = \cfrac{1}{2}(\phi_{i,j+1}+\phi_{i,j}), ~~~~D_{y}\phi_{i,j+\frac{1}{2}}=\cfrac{1}{h}(\phi_{i,j+1}-\phi_{i,j}),
\end{align}
for $i=0,\cdots,m_{1}$ and $j=0,\cdots ,m_{2}$.\\
The center-to-vertex average operator $\mathcal{A}$ : $\mathcal {C}_{\overline{m}_{1} \times \overline{m}_{2}} \rightarrow \mathcal {V}_{m_{1}\times m_{2}}$ is defined as
\begin{align}
\mathcal{A} \phi_{i+\frac{1}{2},j+\frac{1}{2}} = \cfrac{1}{4}(\phi_{i+1,j+1}+\phi_{i,j+1}+\phi_{i+1,j}+\phi_{i,j}),
\end{align}
for $i=0,\cdots,m_{1}$ and $j=0,\cdots ,m_{2}$.\\
The edge-to-vertex average and difference operators, $\mathcal{A}_x$, $\mathcal{D}_x$ : $\mathcal {E}^{ns}_{\overline{m}_{1} \times m_{2}} \rightarrow \mathcal {V}_{m_{1}\times m_{2}}$, and $\mathcal{A}_{y}$, $\mathcal{D}_{y}$ : $\mathcal{E}^{ew}_{m_{1}\times \overline{m}_{2}}\rightarrow \mathcal{V}_{m_{1}\times m_{2}}$ are defined as
\begin{align}
\mathcal{A}_{x}v_{i+\frac{1}{2},j+\frac{1}{2}} &= \cfrac{1}{2}(v_{i+1,j+\frac{1}{2}}+v_{i,j+\frac{1}{2}}),~~~~\mathcal{D}_{x}v_{i+\frac{1}{2},j+\frac{1}{2}} = \cfrac{1}{h}(v_{i+1,j+\frac{1}{2}}-v_{i,j+\frac{1}{2}}),\\
\mathcal{A}_{y}u_{i+\frac{1}{2},j+\frac{1}{2}} &= \cfrac{1}{2}(u_{i+\frac{1}{2},j+1}+u_{i+\frac{1}{2},j}),~~~\mathcal{D}_{y}u_{i+\frac{1}{2},j+\frac{1}{2}} = \cfrac{1}{h}(u_{i+\frac{1}{2},j+1}-u_{i+\frac{1}{2},j}),
\end{align}
for $i=0,\cdots,m_{1}$ and $j=0,\cdots ,m_{2}$. The vertex-to-edge average and difference operators, $\mathfrak{A}_x$, $\mathfrak{D}_x$ : $\mathcal {V}_{m_{1}\times m_{2}} \rightarrow \mathcal {E}^{ns}_{m_{1} \times m_{2}}$, and $\mathfrak{A}_{y}$, $\mathfrak{D}_{y}$ : $\mathcal{V}_{m_{1}\times m_{2}}\rightarrow \mathcal{E}^{ew}_{m_{1}\times m_{2}}$ are defined as 
\begin{align}
\mathfrak{A}_{x}f_{i,j+\frac{1}{2}} = \cfrac{1}{2}(f_{i+\frac{1}{2},j+\frac{1}{2}}+f_{i-\frac{1}{2},j+\frac{1}{2}}),~~~~\mathfrak{D}_{x}f_{i,j+\frac{1}{2}} = \cfrac{1}{h}(f_{i+\frac{1}{2},j+\frac{1}{2}}-f_{i-\frac{1}{2},j+\frac{1}{2}}),\\
\mathfrak{A}_{y}g_{i+\frac{1}{2},j} = \cfrac{1}{2}(g_{i+\frac{1}{2},j+\frac{1}{2}}+g_{i+\frac{1}{2},j-\frac{1}{2}}),~~~~\mathfrak{D}_{y}g_{i+\frac{1}{2},j} = \cfrac{1}{h}(g_{i+\frac{1}{2},j+\frac{1}{2}}-g_{i+\frac{1}{2},j-\frac{1}{2}}),
\end{align}
for $i=0,\cdots,m_{1}$ and $j=1,\cdots ,m_{2}$.
\subsection{Weighted Inner-Products}\label{app-inner-product}\label{app-A3}
Based on the above definitions, we define the following 2D weighted grid inner-products:
\begin{align}
(\phi ,\psi) _{2}&= \sum^{m_{1}}_{i=1}\sum^{m_{2}}_{j=1}\phi_{i,j}\psi_{i,j},\\
[u,\gamma]_{ew}&= \big( a_{x}(u\gamma),1\big)_{2},~[v,\omega]_{ns} = \big( a_{y}(v \omega),1\big)_{2},~\langle f,g\rangle_{vc} = \big(\mathcal{A}(fg),1\big)_{2},\label{winner-vv}
\end{align}
where $\phi \in  \mathcal {C}_{m_{1}\times m_{2}}$, $u,\gamma\in \mathcal {E}^{ew}_{m_{1}\times m_{2}}$, $v,\omega\in \mathcal {E}^{ns}_{m_{1}\times m_{2}}$, and $f,g\in \mathcal {V}_{m_{1}\times m_{2}}$. We also define the following combined 2D weighted grid inner-products:
\begin{align}
[\phi~u,\gamma]_{ew}&= \big(\phi, a_{x}( u \gamma)\big)_{2},~[\phi~v,\omega]_{ns} = \big(\phi, a_{y}( v \omega)\big)_{2},~\langle \phi~f,g\rangle_{vc} = \big(\phi,\mathcal{A}(fg)\big)_{2}.\label{winner-cvv}
\end{align}
We also define the one-dimensional inner-products for the edge-centered functions, or cell-centered functions, or the multiple combination of the edge-centered and cell-centered functions. Here for simplicity, we only introduce the one-dimensional inner-product for the edge-centered functions, the others combinations can be defined analogously:
\begin{align}
&\bigg( u_{i+\frac{1}{2},*}\big\vert\ \gamma_{i+\frac{1}{2},*} \bigg) = \sum_{j=1}^{m_{2}} u_{i+\frac{1}{2},j}\gamma_{i+\frac{1}{2},j},\notag\\
&\bigg( v_{*,j+\frac{1}{2}}\big\vert\ \omega_{*,j+\frac{1}{2}} \bigg) = \sum_{i=1}^{m_{1}} v_{i,j+\frac{1}{2}}\omega_{i,j+\frac{1}{2}}\notag.
\end{align}
Here the first is defined for $u, \gamma \in \mathcal{E}^{ew}_{m_{1}\times m_{2}}$, and the second for $v, \omega \in \mathcal {E}^{ns}_{m_{1}\times m_{2}}$. $*$ indicates the sum of the functions in the direction along which the one-dimensional inner-product acts. Note that, throughout this section, all the boundary terms that originated from the summation-by-parts can be eliminated by using the homogeneous Neumann conditions for the cell-centered variables, and the no-slip boundary conditions for the edge centered variables. The results are also valid for the periodic boundary conditions for the cell-centered or edge-centered variables case. For the above definitions, we obtain the following results:
\begin{prop}(Summation-by-parts) if $\phi\in \mathcal{C}_{\overline{m}_{1}\times m_{2}}$, $u\in \mathcal{E}^{ew}_{m_{1}\times m_{2}}$ and $v \in \mathcal {E}^{ns}_{m_{1}\times m_{2}}$ then
\begin{align}
 h^2[D_{x}\phi , u]_{ew} =& -h^2(\phi , d_{x}u)_{2}- h\big( A_{x}\phi_{\frac{1}{2},*}, u_{\frac{1}{2},*}\big)_{1}\notag\\
 &+ h\big(A_{x}\phi_{m_{1}+\frac{1}{2},*}, u_{m_{1}+\frac{1}{2},*}\big)_{1},\label{sum-bp-2ew-2c}\\
h^2[D_{y}\phi, v]_{ns} =& - h^2(\phi ,d_{y}v)_{2}- h\big( A_{y}\phi_{*,\frac{1}{2}}, v_{*,\frac{1}{2}}\big)_{1}\notag\\
&+ h\big(A_{y}\phi_{*,m_{2}+\frac{1}{2}}, v_{*,m_{2}+\frac{1}{2}}\big)_{1}.\label{sum-bp-2ns-2c}
\end{align}
\end{prop}
\begin{prop}
Let $\phi \in \mathcal{C}_{\overline{m}_{1}\times \overline{m}_{2}}$, $u,\gamma \in \mathcal{E}^{ew}_{\overline{m}_{1}\times m_{2}}$ and $v,\omega \in \mathcal{E}^{ns}_{m_{1}\times \overline{m}_{2}}$. Then
\begin{align}
&h^2\big[ A_{x}(\phi~a_{x}u~d_{x}\gamma),\gamma\big]_{ew}+h^2\big[\cfrac{1}{2}\gamma D_{x}(\phi~a_{x} u),\gamma\big]_{ew}\notag\\
&=\cfrac{h}{4}\big(\phi_{m_{1}+1,*}~a_{x}u_{m_{1}+1,*}~\gamma_{m_{1}+\frac{3}{2},*}~,~\gamma_{m_{1}+\frac{1}{2},*}\big)_{1}\notag\\
&-\cfrac{h}{4}\big(\phi_{0,*}~a_{x}u_{0,*}\gamma_{-\frac{1}{2},*}~,~\gamma_{\frac{1}{2},*}\big)_{1},\label{sum-bp-vel-x1}\\
&h^2\big[ A_{y}\big(\phi~a_{y}v~d_{y}\omega\big),\omega\big]_{ns}+h^2\big[\cfrac{1}{2}\omega D_{y} (\phi~a_{y} v),\omega\big]_{ns}\notag\\
&=\cfrac{h}{4}\big(\phi_{*,m_{2}+1}~a_{y}v_{*,m_{2}+1}~\omega_{*,m_{2}+\frac{3}{2}}~,~\omega_{*,m_{2}+\frac{1}{2}}\big)_{1}\notag\\
&-\cfrac{h}{4}\big(\phi_{*,0}~a_{y}v_{*,0}~\omega_{*,-\frac{1}{2}}~,~\omega_{*,\frac{1}{2}}\big)_{2}.\label{sum-bp-vel-y1}
\end{align}
\end{prop}
\begin{prop}
Let $\phi \in \mathcal{C}_{\overline{m}_{1}\times \overline{m}_{2}}$, $u \in \mathcal{E}^{ew}_{\overline{m}_{1}\times\overline{ m}_{2}}$ and $v \in \mathcal{E}^{ns}_{\overline{m}_{1}\times \overline{m}_{2}}$. Then
\begin{align}
&h^2\big[\mathfrak{A}_{y}\big(\mathcal{A}\phi~\mathcal{A}_{x}v\mathcal{D}_{y}u\big),u\big]_{ew}+h^2\big[\cfrac{1}{2}u\mathfrak{D}_{y} (\mathcal{A}\phi\mathcal{A}_{x}v),u\big]_{ew}\\
=&-\cfrac{h}{2}\big(\mathcal{A}\phi_{*-\frac{1}{2},\frac{1}{2}} u_{*-\frac{1}{2},0}u_{*-\frac{1}{2},1}~,~\mathcal{A}_{x}v_{*-\frac{1}{2},\frac{1}{2}}\big)_{1}\notag\\
&+\cfrac{h}{2} \big(\mathcal{A}\phi_{*+\frac{1}{2},m_{2}+\frac{1}{2}}u_{*+\frac{1}{2},m_{2}}u_{*+\frac{1}{2},m_{2}+1}~,~\mathcal{A}_{x}v_{*+\frac{1}{2},m_{2}+\frac{1}{2}}\big)_{1},\label{sum-bp-vel-x2}\\
&h^2\big[\mathfrak{A}_{x}\big(\mathcal{A}\phi~\mathcal{A}_{y}u\mathcal{D}_{x}v\big),v \big]_{ns}+h^2\big[\cfrac{1}{2}v\mathfrak{D}_{x} (\mathcal{A}\phi \mathcal{A}_{y}u),v\big]_{ns}\notag\\
=&-\cfrac{h}{2}\big(\mathcal{A}\phi_{\frac{1}{2},*-\frac{1}{2}} v_{0,*-\frac{1}{2}} v_{1,*-\frac{1}{2}}~,~ \mathcal{A}_{x}u_{\frac{1}{2},*-\frac{1}{2}}\big)_{1}\notag\\
&+\cfrac{h}{2} \big(\mathcal{A}\phi_{m_{1}+\frac{1}{2},*+\frac{1}{2}}v_{m_{1},*+\frac{1}{2}} v_{m_{1}+1,*+\frac{1}{2}}~,~\mathcal{A}_{x}u_{m_{1}+\frac{1}{2},*+\frac{1}{2}}\big)_{1}.\label{sum-bp-vel-y2}
\end{align}
\end{prop}
For the sake of simplicity, we omit all the boundary terms of that originated in the following Propositions.
\begin{prop}
Let $\phi \in \mathcal{C}_{\overline{m}_{1}\times \overline{m}_{2}}$, $u,\gamma \in \mathcal{E}^{ew}_{\overline{m}_{1}\times \overline{m}_{2}}$ and $v,\omega \in \mathcal{E}^{ns}_{\overline{m}_{1}\times \overline{m}_{2}}$. Then
\begin{align}
[D_{x}(\phi ~d_{x}u), \gamma]_{ew}= &-(\phi~d_{x}u, d_{x}\gamma)_{2},\label{sum-lap-3ew-1c2v}\\
[D_{y}(\phi ~d_{y}v), \omega]_{ns}= &-(\phi~d_{y}v, d_{y}\omega)_{2}.\label{sum-lap-3ns-1c2v}
\end{align}
\end{prop}

\begin{prop}
Let $\phi \in \mathcal{C}_{\overline{m}_{1}\times \overline{m}_{2}}$, $u,\gamma \in \mathcal{E}^{ew}_{\overline{m}_{1}\times \overline{m}_{2}}$ and $v,\omega \in \mathcal{E}^{ns}_{\overline{m}_{1}\times \overline{m}_{2}}$. Then
\begin{align}
\big[\mathfrak{D}_{y}(\mathcal{A}\phi \mathcal{D}_{y} u) ,\gamma\big]_{ew}=&- \langle\phi  \mathcal{D}_{y} u , \mathcal{D}_{y} \gamma\rangle_{vc},\label{sum-bp-3ew-1c2v}\\
\big[\mathfrak{D}_{x}(\mathcal{A}\phi \mathcal{D}_{x}v) ,\omega\big]_{ns}=&-\langle\phi  \mathcal{D}_{x}v , \mathcal{D}_{x} \omega\rangle_{vc}.\label{sum-bp-3ns-1c2v}
\end{align}
\end{prop}
\begin{prop}
Let $\phi \in \mathcal{C}_{\overline{m}_{1}\times \overline{m}_{2}}$, $u\in \mathcal{E}^{ew}_{\overline{m}_{1}\times \overline{m}_{2}}$ and $v \in \mathcal{E}^{ns}_{\overline{m}_{1}\times \overline{m}_{2}}$. Then
\begin{align}
\big[D_{x}(\phi~d_{y} v),u\big]_{ew}=&-(\phi~d_{y} v ,d_{x} u)_{2},\label{sum-bp-cnsew-3c}\\
\big[D_{y}(\phi~d_{x} u),v\big]_{ns}=&-(\phi~d_{x} u , d_{y} v)_{2}.\label{sum-bp-cewns-3c}
\end{align}
\end{prop}
\begin{prop}
Let $\phi,\psi,\zeta \in \mathcal{C}_{\overline{m}_{1}\times \overline{m}_{2}}$. Then
\begin{align}
\big(d_{x}(A_{x}\phi~D_{x} \psi ),\zeta\big)_2=&-\big[ \phi~D_{x} \psi, D_{x} \zeta\big]_{ew},\label{sum-bp-3xc-3c}\\
\big(d_{y}(A_{y}\phi~D_{y} \psi ),\zeta\big)_2=&-\big[ \phi~D_{y} \psi, D_{y} \zeta\big]_{ns}.\label{sum-bp-3yc-3c}
\end{align}
\end{prop}
\subsection{Special Treatment for the Advection Term in Cahn-Hilliard equation}\label{App-surface}\label{mass-der-rho-dis}
To let our numerical schemes satisfy the mass conservation together with the energy dissipation law at the fully discrete level, we have employed a special treatment for the advection term in the Cahn-Hilliard equation (\ref{pm-num-phase}) in primitive method and (\ref{pj-num-phase}) in projection method, such that:
\begin{align}
a_{x} (\rho_{ew}D_{x}c u)_{i,j}= \cfrac{1}{2}\bigg(\rho_{i+1,j}\cfrac{c_{i+1,j}-c_{i,j}}{h} u_{i+\frac{1}{2},j}+\rho_{i-1,j}\cfrac{c_{i,j}-c_{i-1,j}}{h} u_{i-\frac{1}{2},j}\bigg),\\
a_{y} (\rho_{ns}D_{y}c v)_{i,j}= \cfrac{1}{2}\bigg(\rho_{i,j+1}\cfrac{c_{i,j+1}-c_{i,j}}{h} v_{i,j+\frac{1}{2}}+\rho_{i,j-1}\cfrac{c_{i,j}-c_{i,j-1}}{h} v_{i,j-\frac{1}{2}}\bigg),
\end{align} 
for $i=1,\dots,m_{1}$ and $j=1,\dots,m_{2}$. Note that, for the sake of simplicity, we omit the upper subscript $n$ or $n+1$ that represent the time step for all the variables. These two terms are so designed that the primitive method (\ref{pm-num-mx})-(\ref{pm-num-chem}) and the projection-type method (\ref{pj-num-mx})-(\ref{pj-num-chem}) can satisfy the mass conservation property and the energy stability at the fully discrete level. In particular, to prove the identities (\ref{afdsf-mass}) and (\ref{pj-dis-mass-con}), we show that
\begin{align}
&\big(a_{x} (\rho_{ew}D_{x}c u),- \alpha\rho\big)_{2}\notag\\
& =\cfrac{ 1}{2}\sum^{m_{1}}_{i=1}\sum^{m_{2}}_{j=1}-\alpha\rho_{i,j}\bigg(\rho_{i+1,j}\cfrac{c_{i+1,j}-c_{i,j}}{h} u_{i+\frac{1}{2},j}+\rho_{i-1,j}\cfrac{c_{i,j}-c_{i-1,j}}{h} u_{i-\frac{1}{2},j}\bigg)\notag\\
&= \cfrac{ 1}{2}\sum^{m_{1}}_{i=1}\sum^{m_{2}}_{j=1}\bigg(\cfrac{\rho_{i+1,j}-\rho_{i,j}}{h} u_{i+\frac{1}{2},j}+\cfrac{\rho_{i,j}-\rho_{i-1,j}}{h} u_{i-\frac{1}{2},j}\bigg)= \big(a_{x} (D_{x}\rho u), 1\big)_{2},\\
&\big(a_{y} (\rho_{ns}D_{y}c v), - \alpha\rho\big)_{2}\notag\\
&= \cfrac{1}{2}\sum^{m_{1}}_{i=1}\sum^{m_{2}}_{j=1}-\alpha\rho_{i,j}\bigg(\rho_{i,j+1}\cfrac{c_{i,j+1}-c_{i,j}}{h} v_{i,j+\frac{1}{2}}+\rho_{i,j-1}\cfrac{c_{i,j}-c_{i,j-1}}{h} v_{i,j-\frac{1}{2}}\bigg)\notag\\
&= \cfrac{1}{2}\sum^{m_{1}}_{i=1}\sum^{m_{2}}_{j=1}\bigg(\cfrac{\rho_{i,j+1}-\rho_{i,j}}{h} v_{i,j+\frac{1}{2}}+\cfrac{\rho_{i,j}-\rho_{i,j-1}}{h} v_{i,j-\frac{1}{2}}\bigg)=\big(a_{y} (D_{y}\rho v), 1\big)_{2}.
\end{align} 
Here we have used the definition of the variable density $\rho$, and the following identity which is analog to (\ref{def-r}) in the temporal discretization:
\begin{align}
&\rho_{i+1,j}-\rho_{i,j}=-\alpha \rho_{i+1,j}\rho_{i,j}(c_{i+1,j}-c_{i,j}),\\
&\rho_{i,j+1}-\rho_{i,j}=-\alpha \rho_{i,j+1}\rho_{i,j}(c_{i,j+1}-c_{i,j}).
\end{align}
\subsection{Special Treatment for the Surface Tension}\label{App-surface}\label{surface-ten}
To let our numerical schemes satisfy the mass conservation together with the energy dissipation law at the fully discrete level, we have employed a special treatment for the surface tension term in the momentum equations (\ref{pm-num-mx}) of our primitive method, and in the projection equations (\ref{pj-num-pjx}) of our projection-type method, such that 
\begin{align}
A_{x}(\rho_{ew} {\bar{\mu}_{c}})_{i+\frac{1}{2},j}~D_{x} c_{i+\frac{1}{2},j}=\cfrac{1}{2}(\rho_{i+1,j}  {\bar{\mu}_{c~i,j}}+\rho_{i,j}  {\bar{\mu}_{c~i+1,j}})\cfrac{c_{i+1,j}-c_{i,j}}{h},\\
A_{y}(\rho_{ns} {\bar{\mu}_{c}})_{i,j+\frac{1}{2}}~D_{y} c_{i,j+\frac{1}{2}}=\cfrac{1}{2}(\rho_{i,j+1}  {\bar{\mu}_{c~i,j}}+\rho_{i,j}  {\bar{\mu}_{c~i,j+1}})\cfrac{c_{i,j+1}-c_{i,j}}{h},
\end{align}
for $i=1,\dots,m_{1}$ and $j=1,\dots,m_{2}$. Note that, for the sake of simplicity, we omit the upper subscript $n$ or $n+1$ that represent the time step for all the variables. These two terms are so designed that our two numerical schemes can satisfy the energy dissipation law at the fully discrete level. In particular, when we multiply the momentum equation (\ref{pm-num-mx}) or (\ref{pj-num-pjx}) by $\boldsymbol{u}=(u,v)$, the surface tension term can be written as 
\begin{align}
\big[A_{x}(\rho_{ew} {\bar{\mu}_{c}})D_{x} c, u\big]_{ew}&=\big(a_{x} (\rho_{ew}D_{x}c~u), {\bar{\mu}_{c}}\big)_{2},\\
\big[A_{y}(\rho_{ns} {\bar{\mu}_{c}})D_{y} c, v\big]_{ns}&=\big(a_{y} (\rho_{ns}D_{y}c~v), {\bar{\mu}_{c}}\big)_{2}.
\end{align}
Note that, throughout this section, all the boundary terms that originated from the summation-by-parts can be eliminated by using the homogeneous Neumann conditions for the cell-centered variables, and the no-slip boundary conditions for the edge centered variables. The results are also valid for the periodic boundary conditions for the cell-centered or edge-centered variables case. 
\subsection{Boundary Conditions}\label{app-bc}
\subsubsection{Cell-Centered Functions}
In this paper we use the Neumann or periodic boundary condition for the cell-centered functions. Specifically, we shall say the cell-centered function $\phi\in \mathcal{C}_{\overline{m}_{1}\times \overline{m}_{2}}$ satisfies homogeneous Neumann boundary conditions if and only if
\begin{align}
&\phi_{m_{1}+1,j}=\phi_{m_{1},j},~~~~\phi_{0,j}=\phi_{1,j},~~~~j=1,\cdots,m_{2},\label{neubcx}\\
&\phi_{i,m_{2}+1}=\phi_{i,m_{2}},~~~~~~\phi_{i,0}=\phi_{i,1},~~~~i=1,\cdots,m_{1}.\label{neubcy}
\end{align}
We use the notation $\boldsymbol{n}\cdot \boldsymbol{\nabla }_{D}\phi=0$ to indicate that $\phi$ satisfies (\ref{neubcx}) and (\ref{neubcy}). The cell-centered function $\phi\in \mathcal{C}_{\overline{m}_{1}\times \overline{m}_{2}}$ satisfies periodic boundary conditions if and only if
\begin{align}
&\phi_{m_{1}+1,j}=\phi_{1,j},~~~~\phi_{0,j}=\phi_{m_{1},j},~~~~j=1,\cdots,m_{2},\label{pe-bcx}\\
&\phi_{i,m_{2}+1}=\phi_{i,1},~~~~~~\phi_{i,0}=\phi_{i,m_{2}},~~~~i=1,\cdots,m_{1}.\label{pe-bcy}
\end{align}
\subsubsection{Edge-Centered Functions}
We use the no-slip or periodic boundary conditions for the edge-centered functions. We shall say the velocity $\boldsymbol{u}=(u,v)$ (for $u\in \mathcal{E}^{ew}_{m_{1}\times m_{2}}$, $v\in \mathcal{E}^{ns}_{m_{1}\times m_{2}}$) satisfies the no-slip boundary conditions $\boldsymbol{u}|_{\Omega}=0$ if and only if
\begin{align}
u_{\frac{1}{2},j}=u_{m_{1}+\frac{1}{2},j}&=0,~~~~j=1,\cdots,m_{2}.\label{noslipbcx}\\
a_y u_{i+\frac{1}{2},\frac{1}{2}}=a_{y} u_{i+\frac{1}{2},m_{2}+\frac{1}{2}}&=0,~~~~i=0,\cdots,m_{1}.\label{noslipbcx}\\
v_{i,\frac{1}{2}}=v_{i,m_{2}+\frac{1}{2}}&=0,~~~~i=1,\cdots,m_{1}.\label{noslipbcy}\\
a_{x}v_{\frac{1}{2},j+\frac{1}{2}}=a_{x}v_{m_{1}+\frac{1}{2},j+\frac{1}{2}}&=0,~~~~j=0,\cdots,m_{2}.\label{noslipbcy}
\end{align}
We shall say the the velocity $\boldsymbol{u}=(u,v)$ (for $u\in \mathcal{E}^{ew}_{m_{1}\times m_{2}}$, $v\in \mathcal{E}^{ns}_{m_{1}\times m_{2}}$) satisfies the boundary condition $\boldsymbol{n}\cdot \boldsymbol{u}|_{\Omega}=0$ if and only if
\begin{align}
u_{\frac{1}{2},j}=u_{m_{1}+\frac{1}{2},j}&=0,~~~~j=1,\cdots,m_{2}.\label{noslipbcx-1}\\
v_{i,\frac{1}{2}}=v_{i,m_{2}+\frac{1}{2}}&=0,~~~~i=1,\cdots,m_{1}.\label{noslipbcy-1}
\end{align}
We shall say the edge-centered function $u\in \mathcal{E}^{ew}_{\overline{m}_{1}\times m_{2}}$ ($v\in \mathcal{E}^{ns}_{m_{1}\times \overline{m}_{2}}$) satisfies the periodic boundary conditions on the east-west (north-south) boundaries if and only if
\begin{align}
&u_{m_{1}+\frac{3}{2},j}=u_{\frac{1}{2},j},~~~~u_{-\frac{1}{2},j}=u_{m_{1}+\frac{1}{2},j},~~~~j=1,\cdots,m_{2},\label{pecbcu}\\
&v_{i,m_{2}+\frac{3}{2}}=v_{i,\frac{1}{2}},~~~~v_{i,-\frac{1}{2}}=v_{i,m_{2}+\frac{1}{2}},~~~~i=1,\cdots,m_{1}.\label{pecbcv}
\end{align}
\subsection{Norms}\label{app-norm}
We define the following norms for cell-centered functions. If $\phi\in\mathcal {C}_{m_{1}\times m_{2}}$, then $\lVert \phi \rVert _{2} : = \sqrt{h^2 (\phi , \phi)_{2}}$. For $\phi \in \mathcal {C}_{\overline{m}_{1} \times \overline{m}_{2}}$, we define the following norms
\begin{align}
\lVert\boldsymbol{\nabla }_{D}\phi \rVert_{2} :&= \sqrt{h^2[D_{x}\phi , D_{x}\phi]_{ew}+h^2[D_{y}\phi , D_{y}\phi]_{ns}},\\
\lVert\sqrt{\psi}\boldsymbol{\nabla }_{D}\phi \rVert_{2} :&= \sqrt{h^2[\psi~D_{x}\phi \rVert D_{x}\phi]_{ew}+h^2[\psi~D_{y}\phi \rVert D_{y}\phi]_{ns}}.\label{norm-cgradc}
\end{align} 
We also define the following norms for the vector of the edge-centered functions $\boldsymbol{u}=(u,v)$ together with the cell-centered function $\phi$ , where $u\in \mathcal{E}^{ew}_{m_{1}\times m_{2}}$, $v$ $\in \mathcal{E}^{ns}_{m_{1}\times m_{2}}$ and $\phi$ $\in \mathcal{C}_{m_{1}\times m_{2}}$, such that
\begin{align}
\lVert \sqrt{\phi} u \rVert_{2} :&= \sqrt{h^2 [\phi u,u ]_{ew} }, ~~~~\lVert \sqrt{\phi} v \rVert_{2} := \sqrt{h^2 [\phi v,v]_{ns} },\\
\lVert \sqrt{\phi}d_{x} u \rVert_{2} :&= \sqrt{h^2 (\phi~d_{x}u,d_{x}u)_{2} }, ~~~~\lVert \sqrt{\phi} d_{y}v \rVert_{2} := \sqrt{h^2 (\phi~d_y v,d_y v)_{2} },\\
\lVert \sqrt{\phi}\mathcal{D}_{y} u \rVert_{2} :&= \sqrt{h^2 \langle \phi~\mathcal{D}_{y}u,\mathcal{D}_{y}u\rangle_{vc} }, ~~~~\lVert \sqrt{\phi}\mathcal{D}_{x} v \rVert_{2} := \sqrt{h^2 \langle \phi~\mathcal{D}_{x}v,\mathcal{D}_{x}v\rangle_{vc} },\\
\lVert \sqrt{\phi} \boldsymbol{u} \rVert_{2} :&= \sqrt{\lVert \sqrt{\phi} u \rVert^2_{2} +\lVert \sqrt{\phi} v \rVert^2_{2}},\label{norm-rho-vel}\\
\lVert \sqrt{\phi} \boldsymbol{\nabla}_{d} \boldsymbol{u} \rVert_{2} :&= \sqrt{ \big\rVert\sqrt{\phi}~d_{x}u\rVert _{2}^{2} +\lVert\sqrt{\phi}~ \mathcal{D}_{y}u \lVert_{2}^{2} +\lVert\sqrt{\phi}~\mathcal{D}_{x}v\lVert^{2}_{2} + \lVert\sqrt{\phi}~d_{y}v\lVert_{2}^{2}},\label{norm-grad-vel}\\
\lVert \sqrt{\phi} \boldsymbol{\nabla}_{d}\cdot \boldsymbol{u} \rVert_{2} :&= \sqrt{\lVert \sqrt{\phi}~d_{x}u \lVert^{2}_{2} +2( \phi~d_{x}u,d_{y}v)_{2}+ \lVert\sqrt{\phi}~d_{y}v\lVert^{2}_{2}}.\label{norm-grad-dot-vel}
\end{align}
\section{Figures for \S 8.1}\label{app-B}
\begin{figure}[h]
        \begin{subfigure}{0.5\textwidth}
                \centering
                \includegraphics[width=\textwidth]{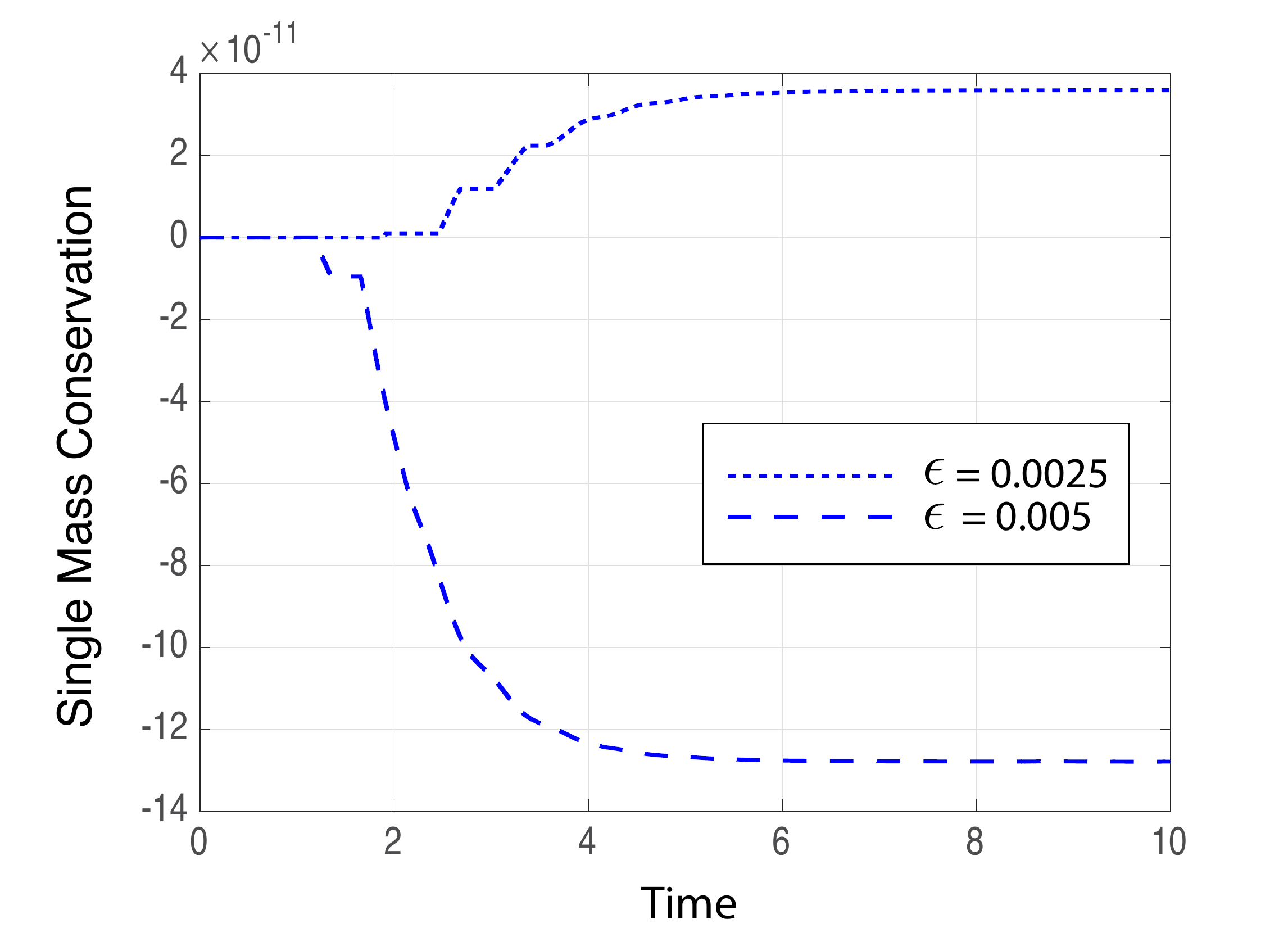}
                               \caption{Primitive method}
        \end{subfigure}%
        \begin{subfigure}{0.5\textwidth}
                \centering
                \includegraphics[width=\textwidth]{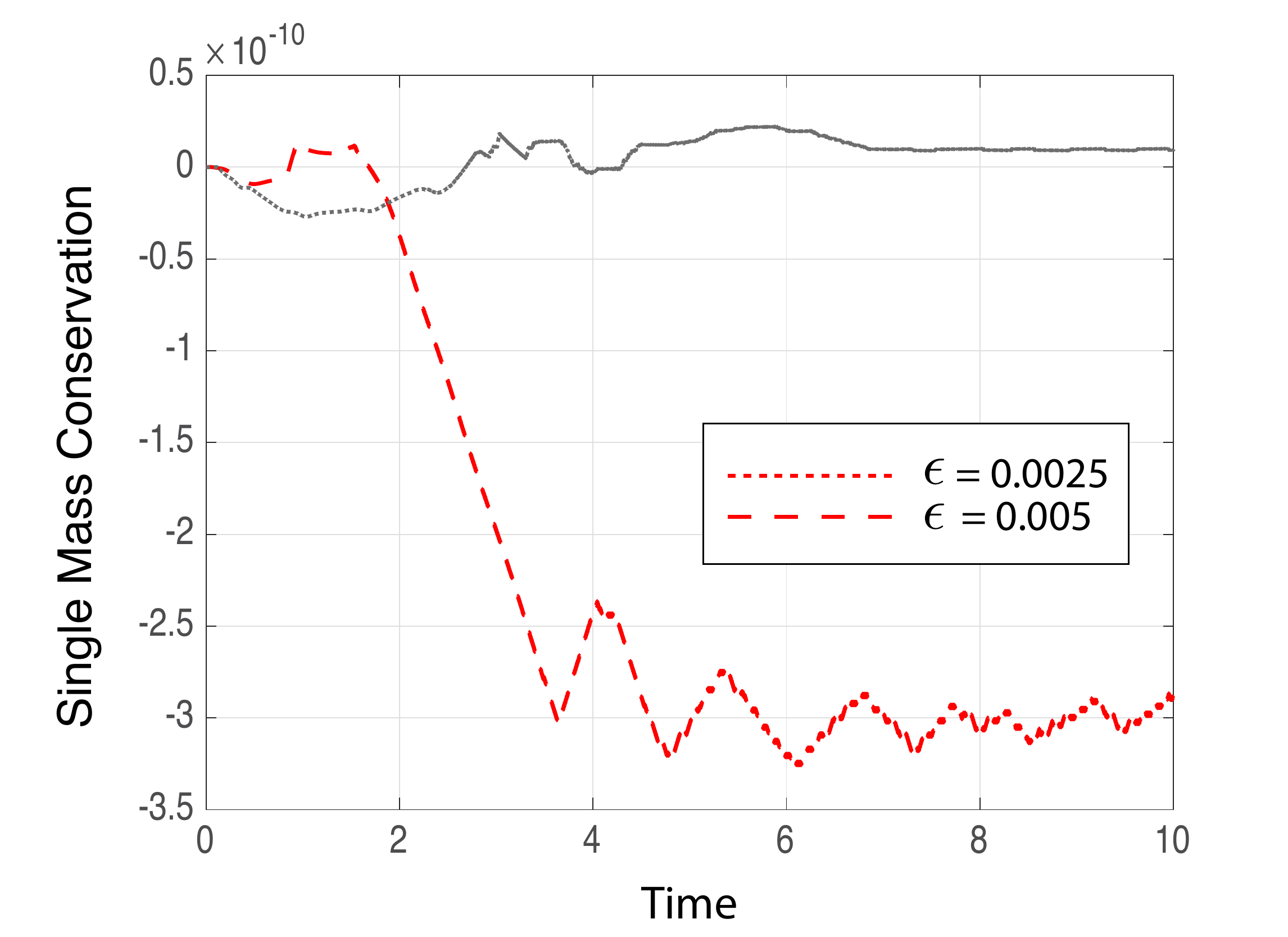}
                               \caption{Projection method}
        \end{subfigure}%
               \caption{Component mass conservation $\int(\rho^{n} c^{n}- \rho^{0} c^{0})$ in Case 1 with density ratio $1:10$ in \S \ref{sec-CW}. (a) Primitive methods; (b) Projection method.}\label{CAP-Den1to10-SMass}
\end{figure}
\begin{figure}[h]
        \begin{subfigure}{0.5\textwidth}
                \centering
                \includegraphics[width=\textwidth]{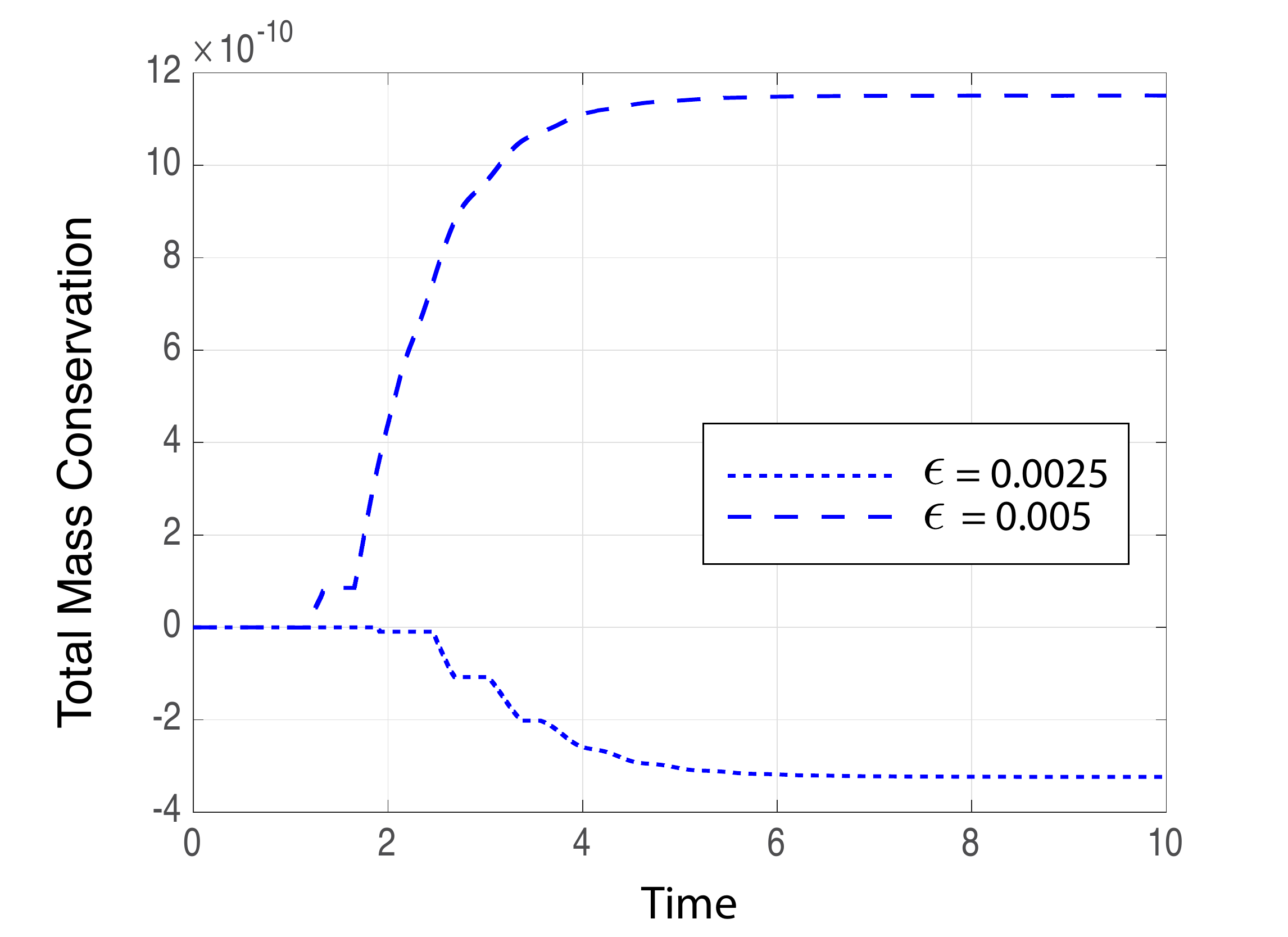}
                               \caption{Primitive method}
        \end{subfigure}%
        \begin{subfigure}{0.5\textwidth}
                \centering
                \includegraphics[width=\textwidth]{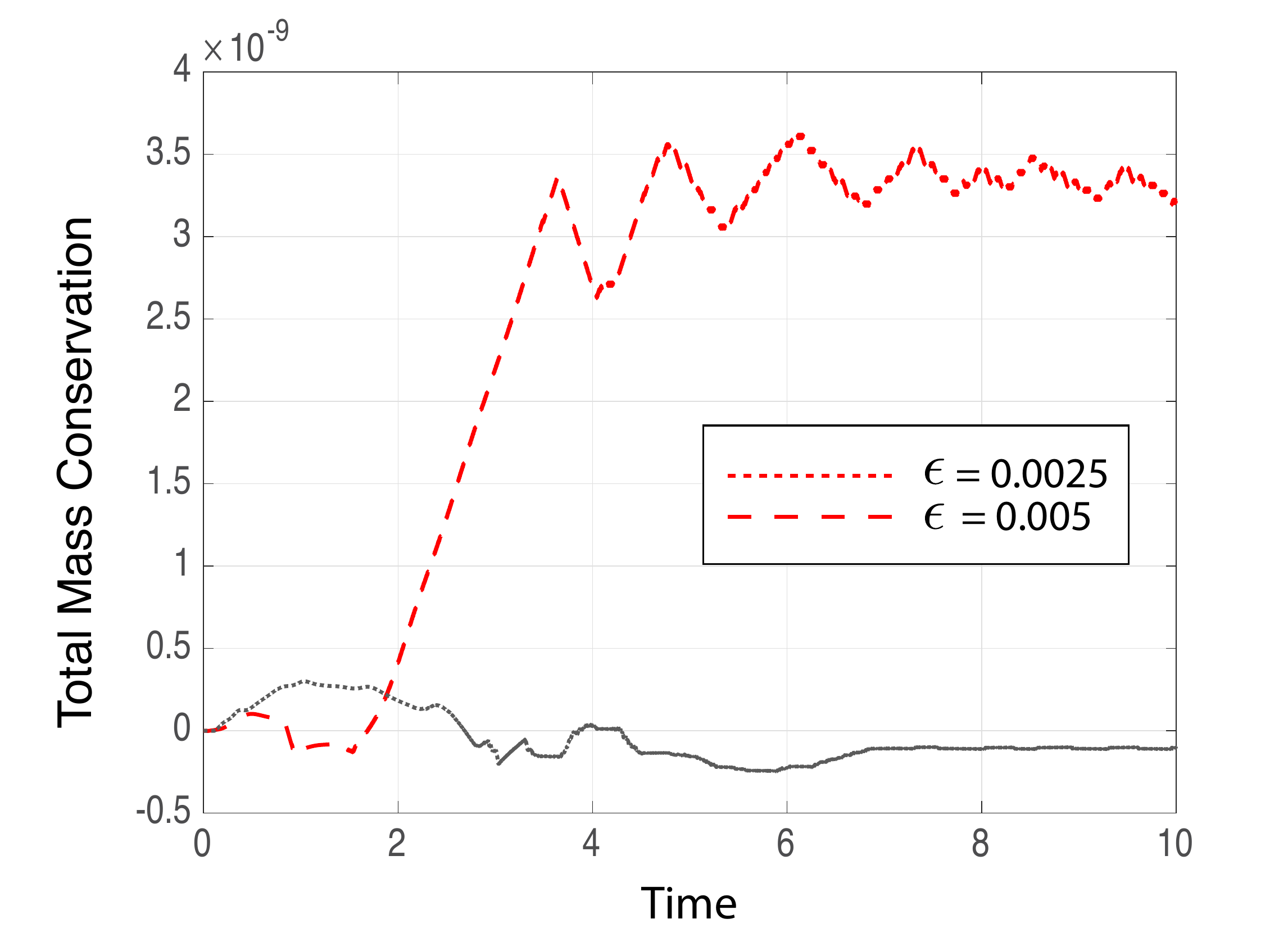}
                               \caption{Projection method}
        \end{subfigure}%
               \caption{Binary fliud mass conservation $\int(\rho^{n}- \rho^{0})$ in Case 1 with density ratio $1:10$ in \S \ref{sec-CW}. (a) Primitive methods; (b) Projection method.}\label{CAP-Den1to10-TMass}
\end{figure}
\end{appendix}
\newpage




\bibliography{ReferenceforPhasefield.bib}{}
\bibliographystyle{plain}

\end{document}